%% file: article.eye-heat-fom-rom-sa.ijnmbe24.tex
\documentclass[a4paper]{article}
\usepackage{fullpage}
\usepackage{graphicx}
\usepackage{xcolor}

\IfFileExists{.git/gitHeadInfo.gin}{
    \usepackage[pcount,grumpy,mark,markifdirty]{gitinfo2}
}{%
    \usepackage[local,pcount,grumpy,mark,markifdirty]{gitinfo2}
}

\input{preamble}
\usepackage{hyperref}
\usepackage[noabbrev]{cleveref}
\usepackage{xurl}

\makeatletter
\title{Model order reduction and sensitivity analysis for complex heat transfer simulations inside the human eyeball}
\author{%
    \texorpdfstring{
        Thomas Saigre\textsuperscript{1}\textsuperscript{$\dagger$}, Christophe Prud'homme\textsuperscript{1}, Marcela Szopos\textsuperscript{2}\\
        \small\textsuperscript{1}Institut de Recherche Mathématique Avancée, UMR 7501 Université de Strasbourg et CNRS\\
        \small\textsuperscript{2}Université Paris Cité, CNRS, MAP5, F-75006 Paris, France\\
        \small\textsuperscript{$\dagger$}Corresponding author contact: \texttt{saigre@math.unistra.fr}
    }{Thomas Saigre, Christophe Prud'homme, Marcela Szopos}
}
\date{\gitReln\  \gitAuthorDate\ (\gitAbbrevHash)}

\definecolor{CustomBlue}{rgb}{0.25, 0.41, 0.88} 
\hypersetup{
    pdftitle={\@title},
    pdfauthor={\@author},
    pdfkeywords={LaTeX, GitHub Actions, Zotero, Overleaf},
    bookmarksnumbered,bookmarksopen,linktocpage,
    colorlinks=true,
    citecolor=CustomBlue,
    linkcolor=CustomBlue,
    urlcolor=blue
}
\makeatother
\renewcommand{\d}{\,\mathrm{d}}

\begin{document}

\maketitle

\begin{abstract}
    Heat transfer in the human eyeball, a complex organ, is significantly influenced by various pathophysiological and external parameters.
    Particularly, heat transfer critically affects fluid behavior within the eye and ocular drug delivery processes.
    Overcoming the challenges of experimental analysis, this study introduces a comprehensive three-dimensional mathematical and computational model to simulate the heat transfer in a realistic geometry.
    Our work includes an extensive sensitivity analysis to address uncertainties and delineate the impact of different variables on heat distribution in ocular tissues.
    To manage the model's complexity, we employed a very fast model reduction technique with certified sharp error bounds, ensuring computational efficiency without compromising accuracy.
    Our results demonstrate remarkable consistency with experimental observations and align closely with existing numerical findings in the literature.
    Crucially, our findings underscore the significant role of blood flow and environmental conditions, particularly in the eye's internal tissues.
    Clinically, this model offers a promising tool for examining the temperature-related effects of various therapeutic interventions on the eye.
    Such insights are invaluable for optimizing treatment strategies in ophthalmology.
\end{abstract}

\tableofcontents

\vspace{\baselineskip}


\input{tex/1-intro}

\input{tex/2-model-phys}

\input{tex/3-model-num}

\input{tex/4-uq}

\input{tex/5-conclusion}

\appendix

\input{tex/A-reproduce}

\section*{Acknowledgements}
\addcontentsline{toc}{section}{Acknowledgements}

The authors would like to acknowledge the support of the platform Cemosis at University of Strasbourg and the French Ministry of Higher Education, Research and Innovation.
This work has benefited from a national grant managed by the French National Research Agency (Agence Nationale de la Recherche) attributed to the Exa-MA project of the NumPEx PEPR program, under the reference ANR-22-EXNU-0002.

\printbibliography{}
\addcontentsline{toc}{section}{References}

\end{document}

%% file: preamble.tex
\usepackage{bbm}
\usepackage{amsmath,amsthm,amsfonts,amssymb}
\usepackage{stmaryrd}
\usepackage{graphicx}
\usepackage{subfigure}
\usepackage{multicol,multirow}
\usepackage{paralist}
\usepackage{textcomp}
\usepackage[ruled]{algorithm2e}
\usepackage[nopatch=footnote]{microtype}
\usepackage{indentfirst}
\usepackage{lmodern}
\usepackage{stmaryrd}
\usepackage{fixmath}
\usepackage{csvsimple}
\usepackage{numprint}
\usepackage{soul}
\usepackage{booktabs}
\SetSymbolFont{stmry}{bold}{U}{stmry}{m}{n}
\usepackage{textgreek}
\usepackage{siunitx}
\usepackage{listofitems} 
\usepackage{pgfplotstable} 
\usepackage{import}
\usepackage{currfile}

\usepackage{csquotes}
\usepackage[backend=biber, style=alphabetic]{biblatex}
\usepackage{pgfplots}
\usepackage{pgfplotstable}
\pgfplotsset{compat=newest}

\usepgfplotslibrary{fillbetween}
\usetikzlibrary{positioning,shapes,shadows,arrows,angles,quotes,decorations.markings}
\usetikzlibrary{mindmap, backgrounds}
\usepgfplotslibrary{groupplots}
\usetikzlibrary{pgfplots.colormaps}

\tikzstyle{abstract}=[rectangle, draw=black, fill=blue!30, drop shadow,text centered, anchor=north, text=black]

\definecolor{colorFeel3}{RGB}{28, 142, 186}
\definecolor{colorFeel2}{RGB}{226, 19, 19}
\definecolor{colorOoi}{RGB}{34, 120, 15}
\definecolor{colorScott}{RGB}{244, 102, 27}
\definecolor{colorLi}{RGB}{108, 2, 119}

\definecolor{colorA}{RGB}{49,140,231}
\definecolor{colorB}{RGB}{238,16,16}

\definecolor{crimson2143940}{RGB}{214,39,40}
\definecolor{darkgray176}{RGB}{176,176,176}
\definecolor{darkorange25512714}{RGB}{255,127,14}
\definecolor{forestgreen4416044}{RGB}{44,160,44}
\definecolor{lightgray204}{RGB}{204,204,204}
\definecolor{steelblue31119180}{RGB}{31,119,180}
\definecolor{mediumpurple148103189}{RGB}{148,103,189}

\newcommand{\logLogSlopeTriangle}[5]
{

    \pgfplotsextra
    {
        \pgfkeysgetvalue{/pgfplots/xmin}{\xmin}
        \pgfkeysgetvalue{/pgfplots/xmax}{\xmax}
        \pgfkeysgetvalue{/pgfplots/ymin}{\ymin}
        \pgfkeysgetvalue{/pgfplots/ymax}{\ymax}

        \pgfmathsetmacro{\xArel}{#1}
        \pgfmathsetmacro{\yArel}{#3}
        \pgfmathsetmacro{\xBrel}{#1-#2}
        \pgfmathsetmacro{\yBrel}{\yArel}
        \pgfmathsetmacro{\xCrel}{\xArel}

        \pgfmathsetmacro{\lnxB}{\xmin*(1-(#1-#2))+\xmax*(#1-#2)} 
        \pgfmathsetmacro{\lnxA}{\xmin*(1-#1)+\xmax*#1} 
        \pgfmathsetmacro{\lnyA}{\ymin*(1-#3)+\ymax*#3} 
        \pgfmathsetmacro{\lnyC}{\lnyA+#4*(\lnxA-\lnxB)}
        \pgfmathsetmacro{\yCrel}{(\lnyC-\ymin)/(\ymax-\ymin)} 

        \coordinate (A) at (rel axis cs:\xArel,\yArel);
        \coordinate (B) at (rel axis cs:\xBrel,\yBrel);
        \coordinate (C) at (rel axis cs:\xCrel,\yCrel);

        \draw[#5]   (A) -- 
                    (B) --
                    (C) -- node[pos=0.5,anchor=west] {\pgfmathprintnumber[precision=2]{#4}}
                    cycle;
    }
}

\newcommand{\vct}[1]{\mathbold{#1}}
\newcommand{\mat}[1]{\mathbold{#1}}
\DeclareMathOperator*{\argmax}{arg\,max}

\newcommand{\E}{\mathbb{E}}
\newcommand{\Em}{\mathcal{E}}
\newcommand{\Dmu}{D^{\mu}}
\newcommand{\N}{\mathcal{N}}
\newcommand{\n}{\vct{n}}
\renewcommand{\P}{\mathbb{P}}
\newcommand{\R}{\mathbb{R}}

\newcommand{\fem}{\mathrm{fem}}
\newcommand{\fpp}{Feel++}
\newcommand{\lb}{\mathrm{lb}}
\newcommand{\rbm}{\mathrm{rbm}}
\newcommand{\test}{\text{test}}
\newcommand{\tol}{\text{tol}}
\newcommand{\train}{\text{train}}
\newcommand{\ub}{\mathrm{ub}}
\newcommand{\var}{\mathrm{var}}
\newcommand{\x}{\vct{x}}

\newcommand{\logN}{\log\text{-}\mathcal{N}}
\newcommand{\unif}{\mathcal{U}}

\newcommand{\norm}[2]{\left\Vert#1\right\Vert_{#2}}

\newcommand{\prm}[1]{\textcolor{blue}{#1}}

\theoremstyle{plain}
\newtheorem{thm}{Theorem}[section]

\theoremstyle{definition}

\newtheorem{rem}[thm]{Remark}

%% file: tex/1-intro.tex
\section{Introduction}
\label{sec:intro}

The development of new technologies allows us to simulate more and more complex models in order to apprehend the world we live in.
In this study, we will focus on a specific model: heat transfer inside the human eyeball.
The temperature of the eyeball may have an impact on the distribution of drugs in the eye, partly due to the aging of the tissues \cite{BHANDARI2020286}.
Hyperthermia is one of the most common treatments for eye tumors~\cite{li2010}, and understanding the mechanism of heat transfer could enhance the efficacy of ophthalmic treatments, such as laser therapy of the retina~\cite{Masters2004}.

Heat transfer is also a key factor in the study of the effects of electromagnetic radiation on the eye, as pointed out in \cite{Hirata2007,doi:10.1142/S0219519409002936}.
The model, originally introduced in \cite{Scott_1988} to examine temperature rises induced by exposure to infrared radiation,
has been expanded upon in subsequent studies \cite{NG2006268, NG2007829, OOI2008252, li2010} using diverse methods for computing heat transfer.

While invasive studies on animals have been conducted \cite{Purslow2005-ky}, non-invasive measurements on human subjects are scarce, complex to perform, and may yield inaccurate results \cite{ROSENBLUTH1977325}.
Most studies focus on temperature measurements at the eye's surface \cite{MAPSTONE1968237, Efron1989OcularST} but report significant differences and identify several sources of uncertainty.
Alternatively, numerical simulations can provide complementary information.
However, in order to guarantee the reliability of such results, a rigorous validation step is required.

The present contribution aims to contribute to these developments, by means of a mathematical and computational modeling approach, combined with a sensitivity analysis study performed thanks to a model reduction technique.
The comparison with data available in the literature, obtained either by measurement on healthy subjects~\cite{Efron1989OcularST} or by other simulations \cite{NG2006268, NG2007829, li2010} will ensure the validity of the approach.

In this model, numerous parameters, both biomechanical and geometrical, are involved.
The present study concentrates on biomechanical parameters, in a large range that include potential extremal conditions.
The variation of these parameters can have a significant impact on the results.
To quantify their impact, we set up a framework to perform a forward uncertainty quantification study, complemented by a sensitivity analysis.
Deterministic sensitivity analysis has already been performed in \cite{Scott_1988, NG2006268, NG2007829, li2010}, using various numerical methods.
In this work, we reproduce and extend these results, to incorporate the effect of blood flow, as suggested for instance in \cite{Scott_1988}.
We also run a global sensitivity analysis, that accounts for stochastic effects, and discriminate among different factors by means of Sobol's indices \cite{Sobol1993SensitivityEF}.
The combination of deterministic and stochastic methods is an effective practice in the field of uncertainty quantification, and has been successfully used in recent various applications, such as \cite{DODIG201448} where the impact of uncertainties on the distribution of the electromagnetic field in the ocular tissue is studied;
or more generally in the human head \cite{SUSNJARA20221,9522096}.
To the best of our knowledge, this is the first time that such a study is performed in the context of bioheat transport in the tissues of the human eyeball.

While Sobol's indices are effective in measuring parameter impact and interactions, the complexity and the significant computational time of our model are very challenging.
To overcome this, we adopt the certified reduced basis method \cite{10.1115/1.1448332, Quarteroni2016} to obtain a reduced model, maintaining its 3D nature while significantly reducing computational demands.
This method aligns with the paradigm observed in patient-specific mathematical models applied to biomedical problems,
ensuring a comprehensive approach involving data integration, model derivation, numerical solving, validation, and uncertainty quantification, as seen in mature research fields like cardiovascular simulations or cerebral hemodynamics.
In ophthalmology, a similar paradigm is imperative due to the richness and heterogeneity of available data, requiring innovative approaches for diagnosis and monitoring.

More generally, the present work aims to contribute to the project Eye2Brain \cite{eye2brain},
that has the ambitious objective to connect the cerebral and ocular environments and contribute in the long term to a better understanding of neurodegenerative diseases~\cite{Guidoboni2020-vr}.
In this context, a model accounting for the combined effects of ocular blood flow and different ocular tissues was proposed in \cite{https://doi.org/10.1002/cnm.3791}0
To incorporate inherent uncertainties and variability, an uncertainty propagation and sensitivity analysis on the component simulating the fluid flows in the eye was developed in \cite{MBE2021}.
We here focus on the heat propagation phenomena, with the perspective of coupling the fluid and thermal contributions in future work.

The structure of the paper is the following.
After the introduction, we describe in \Cref{sec:model-phys} the geometrical model describing the human eyeball,
the biophysical model governing the heat transfer, as well as the parameters involved in the equations.
Next, we present in \Cref{sec:model-num} the methods developed to simulate the full and reduced models, including a step of verification and validation,
to ensure that the mathematical and computational framework is correct.
We report in \Cref{sec:uq} our results of the sensitivity analysis, using two methods: a deterministic one and a stochastic approach.
All the methods are implemented in the open-source software Feel++ \cite{christophe_prud_homme_2023_8272196} and can be reproduced following guidelines described in \Cref{app:reproduce}.
Finally, conclusions and perspectives are outlined in \Cref{sec:conclusion}.

%% file: tex/2-model-phys.tex
\section{Three-dimensional biophysical model}
\label{sec:model-phys}

\input{tex/2.1-model-geo}

We focus on \emph{outputs of interest} that are studied in the literature \cite{Scott_1988, NG2006268}.
These outputs are the temperature values at given locations or the mean temperature on a given domain.
Precisely, we select on points present at the interface of two regions of the eye, as well as the mean temperature over the cornea.
For a precise description of these locations, see \Cref{fig:outputs}.

\begin{figure}
    \centering
    \def\svgwidth{0.5\columnwidth}
    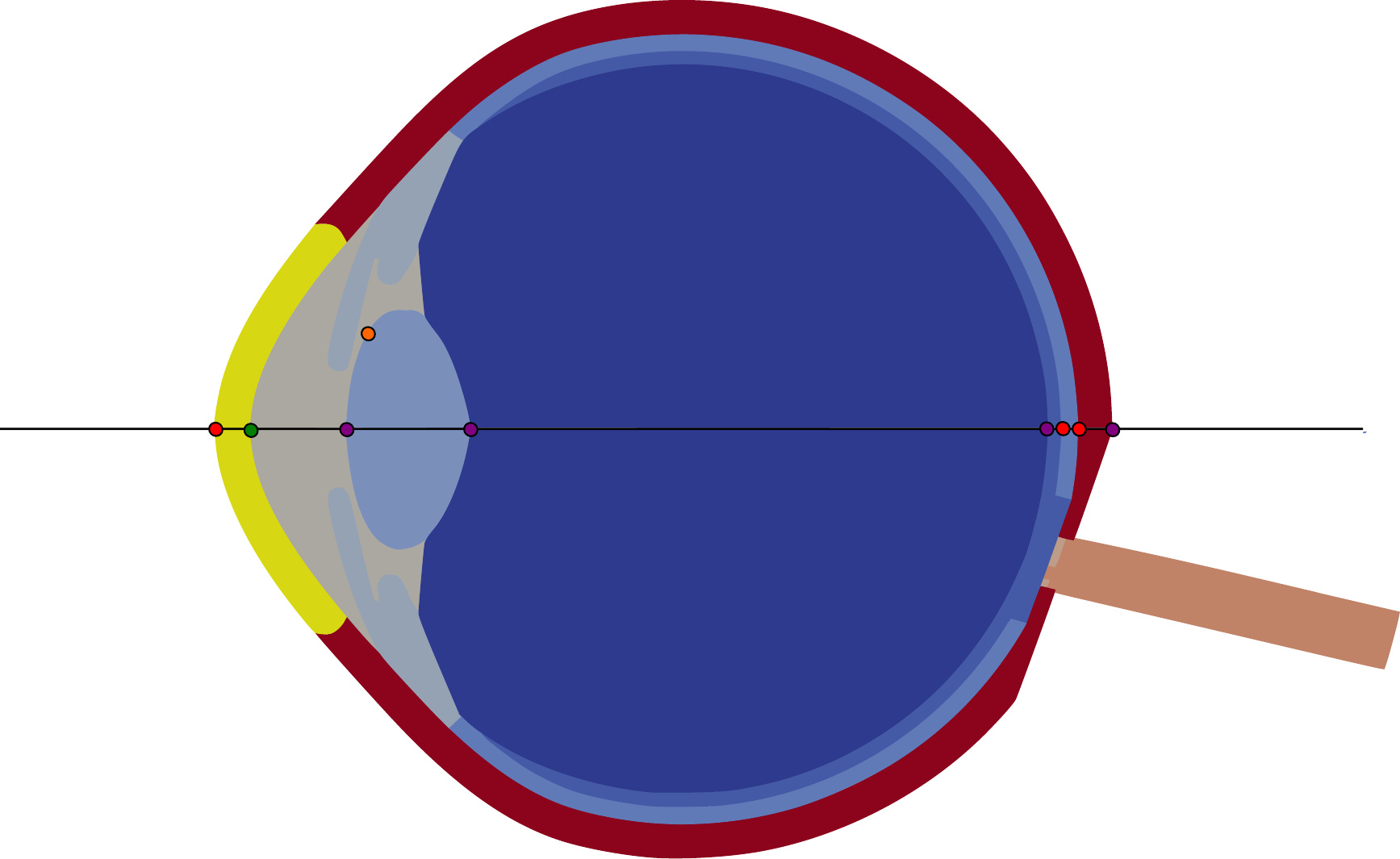
    \caption{Featured geometrical locations for the output of interest (pointwise temperature).}
    \label{fig:outputs}
\end{figure}

\subsection{Biomechanical non-linear continuous model and its linearization}

Based on \Cref{sec:model-geo}, the geometry of the eye can be written as a disjoint union of different regions:
$\Omega = \bigsqcup_{i=1}^{10} \Omega_i$, where $i$ is the index of the subdomain and $\Omega_i$ corresponds to the following regions: cornea, vitreous humor, aqueous humor, retina, iris, choroid, lens, sclera, lamina cribrosa, and optic nerve.


\begin{subequations}
We focus on stationary heat transfer in this domain.
Following \cite{Scott_1988, NG2006268} the steady-state condition of the heat transfer in the human eye can be described by the following system
\begin{equation}
    \nabla\cdot\left(k_i\,\nabla T\right) = 0\qquad\text{ in }\Omega = \sqcup_{i=1}^{10}\Omega_i
    \label{eq:model:heat}
\end{equation}

where:
\begin{itemize}
    \item $i$ is the volume index (cornea, vitreousHumor...),
    \item $T_i$ [\unit{\kelvin}] is the temperature in the domain $\Omega_i$,
    \item $k_i$ [\unit{\watt.\meter^{-1}.\kelvin^{-1}}] is the thermal conductivity of $\Omega_i$.
\end{itemize}

We set the global thermal conductivity $k$ [\unit{\watt.\meter^{-1}.\kelvin^{-1}}] as a discontinuous piece-wise constant function: $k = k_i$ on $\Omega_i$.
The boundary $\partial\Omega$ is decomposed as: $\partial\Omega = \Gamma_\text{amb}\cup\Gamma_\text{body}$ (see \Cref{fig:boundary-conditions}),
where $\Gamma_\text{amb}$ corresponds to the boundary region exposed to the ambient environment and $\Gamma_\text{body}$ the boundary of the internal domain.
Denote by $\n$ the outward normal vector to the domain $\Omega$.
The following boundary conditions are adopted:

\begin{itemize}
    \item To model the exchange between the eye and the ambient air, and incorporate radiative heat transfer we impose the following non-linear Neumann condition:
        \begin{equation}
        -k\dfrac{\partial T}{\partial \n} = h_\text{amb}(T-T_\text{amb}) + \sigma\varepsilon(T^4-T_\text{amb}^4) + E
        \quad\text{on }\Gamma_\text{amb}
        \label{eq:model:neumann}
        \end{equation}
        Three terms are present in this condition to describe different heat loss mechanisms occurring on the cornea:
        (i) The first term in the equation represents the convective heat transfer between the surface of the eye and the surrounding air.
        The parameter $h_\text{amb}$ [\unit{\watt.\meter^{-2}.\kelvin^{-1}}] is the air convective coefficient, and $T_\text{amb}$ [\unit{\kelvin}] is the ambient temperature;
        (ii) the second term represents the radiative heat transfer between the surface of the eye and the surrounding environment,
        where the parameter $\sigma$ is the Stefan-Boltzmann constant ($\sigma = \qty{5.67e-8}{\watt.\meter^{-2}.\kelvin^{-1}}$), and $\varepsilon$ [--] is the emissivity of the surface;
        (iii) the third term represents the heat loss due to tear evaporation.
        The parameter $E$ [\unit{\watt.\meter^{-2}}] represents the heat transfer rate due to evaporation, which depends on the environmental conditions and the tear film characteristics.
        This process causes a cooling effect on the surface of the eye, which can be significant in dry environments or cases of reduced tear production.

    \item  To model the thermal exchanges between the eye and the body, we impose:
        \begin{equation}
        -k\frac{\partial T}{\partial\n} = h_\text{bl}(T-T_\text{bl})
        \quad\text{on }\Gamma_\text{body}
        \label{eq:model:robin}
        \end{equation}
        where the parameter $h_\text{bl}$ [\unit{\watt.\meter^{-2}.\kelvin^{-1}}] is the blood convection coefficient and $T_\text{bl}$ [\unit{\kelvin}] is the blood temperature.
\end{itemize}


Finally, to ensure a continuous flow of heat flux and no temperature jump, we impose at the interface between two adjacent regions $\Omega_i$ and $\Omega_j$ the following condition:

\begin{equation}
\begin{cases}
    \begin{tabular}{rcl}
        $T_i$ & $=$ & $T_j$\\
        $k_i(\nabla T_i\cdot\n_i)$ & $=$ & $-k_j(\nabla T_j\cdot\n_j)$
    \end{tabular}
    \text{on } \partial\Omega_i\cap\partial\Omega_j
\end{cases}
\label{eq:model:interface}
\end{equation}
\label{eq:model}
where $\n_i$ (resp. $\n_j$) denotes the outward normal vector to the domain $\Omega_i$ (resp. $\Omega_j$).

\end{subequations}

\begin{figure}
    \centering
    \def\svgwidth{0.3\columnwidth}
    \import{./fig/eye/vectorized-figures/}{boudaries_color_free.pdf_tex}
    \caption{Description of the physical boundaries and interfaces of the domain $\Omega$.}
    \label{fig:boundary-conditions}
\end{figure}

System (\ref{eq:model:heat}) - (\ref{eq:model:interface}) defines a non-linear problem, denoted $\Em_\text{NL}$ in the sequel.

\begin{rem}
    \label{rem:linearization}
    Note that the condition (\ref{eq:model:neumann}) modeling radiative transfer is non-linear, because of the term in $T^4$,
    which requires a more complex treatment, both from the mathematical standpoint, for the reduced basis method; and from the numerical standpoint, due to extra computational cost.
    As an alternative, a linearization of the condition (\ref{eq:model:neumann}) was proposed in \cite{Scott_1988}:

    \begin{equation*}
        \sigma\varepsilon (T^4-T_\text{amb}^4) = (T-T_\text{amb})\underbrace{\sigma\varepsilon(T^2+T^2_\text{amb})(T+T_\text{amb})}_{=: h_\text{r}},
    \end{equation*}
    which leads to a linear Robin condition.
    The value $h_\text{r}$ stands for the \emph{radiation heat transfer coefficient} and is approximately equal to \qty{6}{\watt.\meter^{-2}.\kelvin^{-1}} \cite{Scott_1988}.

    Condition (\ref{eq:model:neumann}) can hence be rewritten as:

    \begin{equation}
        -k\dfrac{\partial T_i}{\partial \n} = h_\text{amb}(T-T_\text{amb}) + h_\text{r}(T-T_\text{amb}) + E \qquad\text{ on }\Gamma_\text{amb}
        \label{eq:neumann-lin}
    \end{equation}

    The model described by Equations (\ref{eq:model:heat})-(\ref{eq:neumann-lin})-(\ref{eq:model:robin})-(\ref{eq:model:interface}) is further denoted $\Em_\text{L}$.
\end{rem}

\subsection{Model parameters}
\label{sec:parameters}

In the model presented in the previous section, many parameters are involved, but not all of them are directly measurable.
Moreover, inherent uncertainties due to noise and individual variability must be taken into account in the modeling process.
We therefore fixed in a first stage a set of baseline values, corresponding to the nominal values for the human body, according to the literature \cite{Scott_1988, NG2006268} (see \Cref{tab:parameters}).
In a second step, we split the total set of parameters into two subsets:
a first part kept fixed to baseline values, and a second part that varies in a certain range (see \Cref{tab:parameters}).
The aim is to perform a refined sensitivity analysis, that encompasses previously published studies \cite{Scott_1988, NG2006268, NG2007829}, and extends the analysis to a larger parameter space.

\begin{table}
    \centering
    \resizebox{\textwidth}{!}{
    \begin{tabular}{*{5}{c}}
        \toprule
        \textsf{\textbf{Symbol}} & \textsf{\textbf{Name}} & \textsf{\textbf{Dimension}} & \textsf{\textbf{Baseline value}} & \textsf{\textbf{Range}}\\
        \midrule
        $T_\text{amb}$                                  & Ambient temperature                   & [\unit{\kelvin}]                        & 298    & [283.15, 303.15] \\
        $T_\text{bl}$                                   & Blood temperature                     & [\unit{\kelvin}]                        & 310    & [308.3, 312] \\
        $h_\text{amb}$                                  & Ambient air convection coefficient    & [\unit{\watt.\meter^{-2}.\kelvin^{-1}}] & 10     & [8, 100] \\
        $h_\text{bl}$                                   & Blood convection coefficient          & [\unit{\watt.\meter^{-2}.\kelvin^{-1}}] & 65     & [50, 110] \\
        $E$                                             & Evaporation rate                      & [\unit{\watt.\meter^{-2}}]              & 40     & [20, 320] \\
        $k_\text{lens}$                                 & Lens conductivity                     & [\unit{\watt.\meter^{-1}.\kelvin^{-1}}] & 0.4    & [0.21, 0.544] \\
        $k_\text{cornea}$                               & Cornea conductivity                   & [\unit{\watt.\meter^{-1}.\kelvin^{-1}}] & 0.58   & -- \\
        \begin{tabular}{c}$k_\text{sclera}=k_\text{iris}
            =$\\$k_\text{lamina}=k_\text{opticNerve}$\end{tabular}       & Eye envelope components conductivity  & [\unit{\watt.\meter^{-1}.\kelvin^{-1}}] & 1.0042 & -- \\
        $k_\text{aqueousHumor}$                         & Aqueous humor conductivity            & [\unit{\watt.\meter^{-1}.\kelvin^{-1}}] & 0.28   & -- \\
        $k_\text{vitreousHumor}$                        & Vitreous humor conductivity           & [\unit{\watt.\meter^{-1}.\kelvin^{-1}}] & 0.603  & -- \\
        $k_\text{choroid}=k_\text{retina}$              & Vascular beds conductivity            & [\unit{\watt.\meter^{-1}.\kelvin^{-1}}] & 0.52   & -- \\
        $\varepsilon$                                   & Emissivity of the cornea              & [--]                                    & 0.975  & -- \\
        \bottomrule
    \end{tabular}
    }
    \caption{Parameters involved in the model, baseline values and ranges used in the sensitivity analysis.}
    \label{tab:parameters}
\end{table}

Specifically, we set the varying \emph{parameter space} $\Dmu\subset\R^6$ as the Cartesian product of the intervals defined in the last column of \Cref{tab:parameters}.
For the purpose of the sensitivity analysis, an element $\mu=\{T_\text{amb}, T_\text{bl}, h_\text{amb}, h_\text{bl}, E, k_\text{lens}\}\in\Dmu$ is called a \emph{parameter},
and we denote $\bar{\mu}$ the baseline parameter, extracted from the corresponding column in \Cref{tab:parameters}.
The dependence of the model concerning the parameter $\mu$ is emphasized by the notation $\Em_\text{L}(\mu)$ and $\Em_\text{NL}(\mu)$.

%% file: tex/2.1-model-geo.tex
\subsection{Geometry of the human eyeball}
\label{sec:model-geo}

In this section, we describe the realistic three-dimensional geometry that will be used in the sequel.
The model we employ in the present work stems from \cite{https://doi.org/10.1002/cnm.3791},
and was constructed using a CAD (Computer Aided Design) module from SALOME \cite{salome}.
\Cref{fig:geo-eye} shows a cut-away view along a vertical plane of the reconstructed eye anatomy.

\begin{figure}
    \centering

    \def\svgwidth{0.5\columnwidth}
    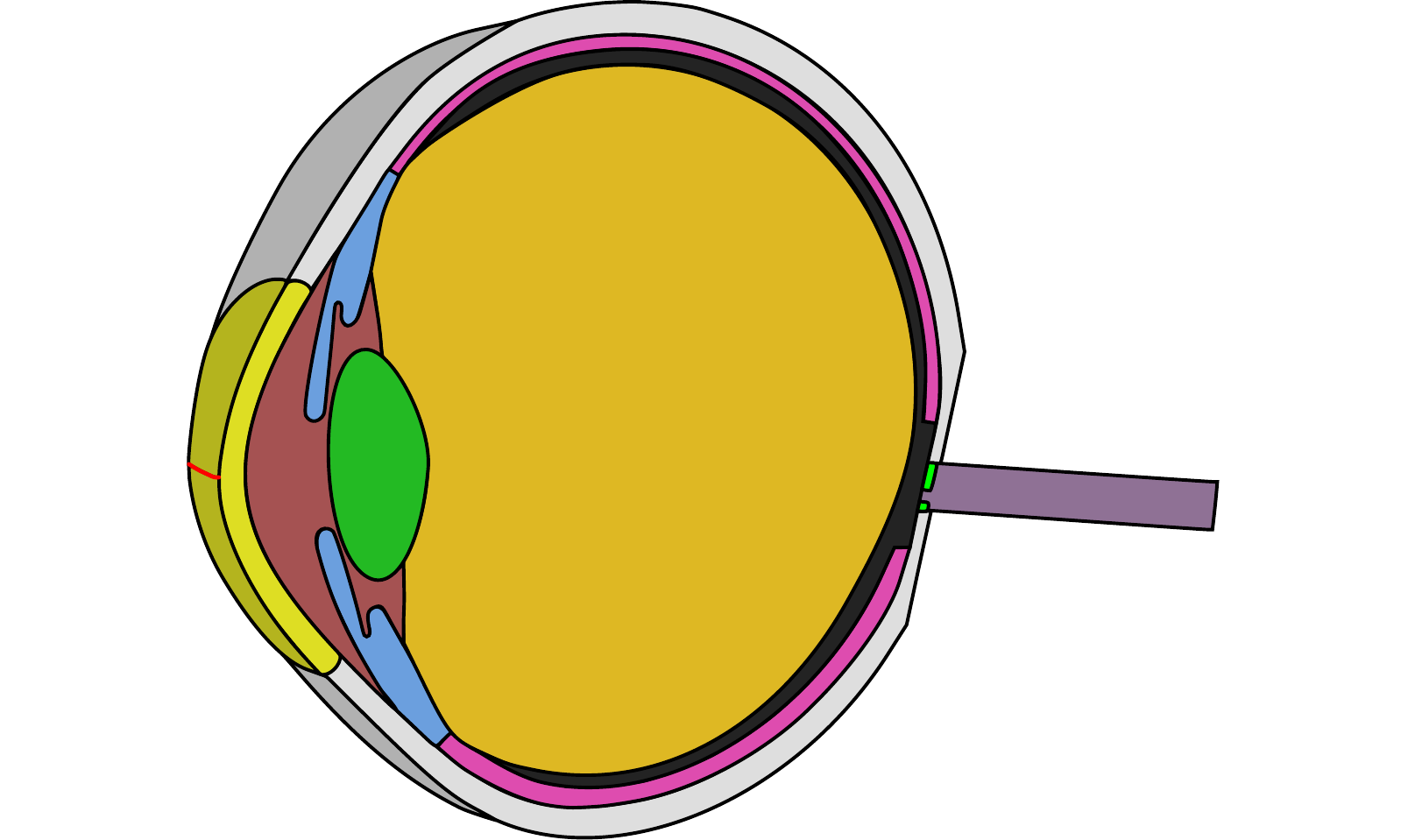
    \caption{Geometrical model of the human eye.}
    \label{fig:geo-eye}
\end{figure}

The eye is composed of several regions, which have different physical properties.
The original geometry contained five subdomains: the sclera, the choroid, the retina, the cornea and the lamina cribrosa.
To have a better assessment of the thermal properties of each part, we further decompose the geometry as follows:
(i) the cornea which allows heat transfer between the eye and the ambient air,
(ii) the envelope of the eye composed of the sclera, the optic nerve, and the lamina cribrosa,
(iii) the vascular beds namely the choroid and the retina, mostly composed of blood vessels,
(iv) the anterior and posterior chambers, filled with aqueous humor,
(v) the lens and
(vi) the vitreous body filled with the vitreous humor, a transparent liquid allowing the light to reach the retina.
In the present model, the optic nerve domain is assumed to be homogeneous, the contribution of the inner vessels is not directly taken into account in heat transfer.

Several more simplified geometrical descriptions were already utilized in the literature to study heat transport in the eye;
mostly in 2D \cite{Scott_1988,NG2006268} or in 3D \cite{NG2007829,li2010}.
In particular, the 3D model developed in \cite{NG2007829} did not incorporate a detailed description of the vascular beds, although previous studies \cite{Scott_1988} and our further sensitivity analysis pointed out the importance of the influence of the blood temperature on the heat distribution.

In order to compare in a first stage our results with previously reported findings \cite{Efron1989OcularST},
we define on the front part of the cornea the \emph{geometrical center of the cornea} (GCC), see \Cref{fig:geo-eye}, which is an imaginary line ``cutting'' the cornea horizontally.
This region is interesting because this part of the eye is easily accessible and the temperature can be measured non-invasively.

%% file: fig/eye/vectorized-figures/eye.pdf_tex
\begingroup%
  \makeatletter%
  \providecommand\color[2][]{%
    \errmessage{(Inkscape) Color is used for the text in Inkscape, but the package 'color.sty' is not loaded}%
    \renewcommand\color[2][]{}%
  }%
  \providecommand\transparent[1]{%
    \errmessage{(Inkscape) Transparency is used (non-zero) for the text in Inkscape, but the package 'transparent.sty' is not loaded}%
    \renewcommand\transparent[1]{}%
  }%
  \providecommand\rotatebox[2]{#2}%
  \newcommand*\fsize{\dimexpr\f@size pt\relax}%
  \newcommand*\lineheight[1]{\fontsize{\fsize}{#1\fsize}\selectfont}%
  \ifx\svgwidth\undefined%
    \setlength{\unitlength}{778.63751221bp}%
    \ifx\svgscale\undefined%
      \relax%
    \else%
      \setlength{\unitlength}{\unitlength * \real{\svgscale}}%
    \fi%
  \else%
    \setlength{\unitlength}{\svgwidth}%
  \fi%
  \global\let\svgwidth\undefined%
  \global\let\svgscale\undefined%
  \makeatother%
  \begin{picture}(1,0.59082573)%
    \lineheight{1}%
    \setlength\tabcolsep{0pt}%
    \put(0,0){\includegraphics[width=\unitlength,page=1]{eye.pdf}}%
    \put(0,0){\includegraphics[width=\unitlength,page=2]{eye.pdf}}%
    \put(0.68975632,0.36259866){\makebox(0,0)[lt]{\lineheight{1.25}\smash{\begin{tabular}[t]{l}Lamina Cribrosa\end{tabular}}}}%
    \put(0,0){\includegraphics[width=\unitlength,page=3]{eye.pdf}}%
    \put(0.75191582,0.14130218){\makebox(0,0)[lt]{\lineheight{1.25}\smash{\begin{tabular}[t]{l}Optic Nerve\end{tabular}}}}%
    \put(0.43676274,0.36714171){\makebox(0,0)[t]{\lineheight{1.25}\smash{\begin{tabular}[t]{c}Vitreous\\body\end{tabular}}}}%
    \put(0,0){\includegraphics[width=\unitlength,page=4]{eye.pdf}}%
    \put(0.34401795,0.19743467){\makebox(0,0)[lt]{\lineheight{1.25}\smash{\begin{tabular}[t]{l}Lens\end{tabular}}}}%
    \put(0,0){\includegraphics[width=\unitlength,page=5]{eye.pdf}}%
    \put(0.66026044,0.53175482){\makebox(0,0)[lt]{\lineheight{1.25}\smash{\begin{tabular}[t]{l}Choroid\end{tabular}}}}%
    \put(0,0){\includegraphics[width=\unitlength,page=6]{eye.pdf}}%
    \put(0.69906994,0.05141589){\makebox(0,0)[lt]{\lineheight{1.25}\smash{\begin{tabular}[t]{l}Retina\end{tabular}}}}%
    \put(0.0556473,0.53499595){\color[rgb]{0,0,0}\makebox(0,0)[lt]{\lineheight{1.25}\smash{\begin{tabular}[t]{l}Sclera\end{tabular}}}}%
    \put(0,0){\includegraphics[width=\unitlength,page=7]{eye.pdf}}%
    \put(0.13181095,0.06790147){\makebox(0,0)[t]{\lineheight{1.25}\smash{\begin{tabular}[t]{c}Aqueous\\humor\end{tabular}}}}%
    \put(0,0){\includegraphics[width=\unitlength,page=8]{eye.pdf}}%
    \put(-0.00216348,0.43464768){\makebox(0,0)[lt]{\lineheight{1.25}\smash{\begin{tabular}[t]{l}Cornea\end{tabular}}}}%
    \put(-0.28,0.2928148){\makebox(0,0)[lt]{\lineheight{1.25}\smash{\begin{tabular}[t]{r}Geometrical center\\of the cornea (\emph{GCC})\end{tabular}}}}%
    \put(0,0){\includegraphics[width=\unitlength,page=9]{eye.pdf}}%
  \end{picture}%
\endgroup%

%% file: fig/eye/vectorized-figures/eye-cut.pdf_tex
\begingroup%
  \makeatletter%
  \providecommand\color[2][]{%
    \errmessage{(Inkscape) Color is used for the text in Inkscape, but the package 'color.sty' is not loaded}%
    \renewcommand\color[2][]{}%
  }%
  \providecommand\transparent[1]{%
    \errmessage{(Inkscape) Transparency is used (non-zero) for the text in Inkscape, but the package 'transparent.sty' is not loaded}%
    \renewcommand\transparent[1]{}%
  }%
  \providecommand\rotatebox[2]{#2}%
  \newcommand*\fsize{\dimexpr\f@size pt\relax}%
  \newcommand*\lineheight[1]{\fontsize{\fsize}{#1\fsize}\selectfont}%
  \ifx\svgwidth\undefined%
    \setlength{\unitlength}{832.78134155bp}%
    \ifx\svgscale\undefined%
      \relax%
    \else%
      \setlength{\unitlength}{\unitlength * \real{\svgscale}}%
    \fi%
  \else%
    \setlength{\unitlength}{\svgwidth}%
  \fi%
  \global\let\svgwidth\undefined%
  \global\let\svgscale\undefined%
  \makeatother%
  \begin{picture}(1,0.61345329)%
    \lineheight{1}%
    \setlength\tabcolsep{0pt}%
    \put(0,0){\includegraphics[width=\unitlength,page=1]{eye-cut.pdf}}%
    \put(0.11396182,0.31553233){\color[rgb]{0,0,0}\makebox(0,0)[lt]{\lineheight{1.25}\smash{\begin{tabular}[t]{l}$O$\end{tabular}}}}%
  \ifdefined\onlyOG\else
    \put(0.17765216,0.27805574){\color[rgb]{0,0,0}\makebox(0,0)[lt]{\lineheight{1.25}\smash{\begin{tabular}[t]{l}$A$\end{tabular}}}}%
    \put(0.24251669,0.32742332){\color[rgb]{0,0,0}\makebox(0,0)[lt]{\lineheight{1.25}\smash{\begin{tabular}[t]{l}$B$\end{tabular}}}}%
    \put(0.25601415,0.38929956){\color[rgb]{0,0,0}\makebox(0,0)[lt]{\lineheight{1.25}\smash{\begin{tabular}[t]{l}$B_1$\end{tabular}}}}%
    \put(0.31685661,0.27363192){\color[rgb]{0,0,0}\makebox(0,0)[lt]{\lineheight{1.25}\smash{\begin{tabular}[t]{l}$C$\end{tabular}}}}%
    \put(0.6993713,0.2749902){\color[rgb]{0,0,0}\makebox(0,0)[lt]{\lineheight{1.25}\smash{\begin{tabular}[t]{l}$D$\end{tabular}}}}%
    \put(0.74682681,0.34361386){\color[rgb]{0,0,0}\makebox(0,0)[lt]{\lineheight{1.25}\smash{\begin{tabular}[t]{l}$D_1$\end{tabular}}}}%
    \put(0.7788091,0.26877225){\color[rgb]{0,0,0}\makebox(0,0)[lt]{\lineheight{1.25}\smash{\begin{tabular}[t]{l}$F$\end{tabular}}}}%
\fi
    \put(0.82230531,0.32540976){\color[rgb]{0,0,0}\makebox(0,0)[lt]{\lineheight{1.25}\smash{\begin{tabular}[t]{l}$G$\end{tabular}}}}%
  \end{picture}%
\endgroup%

%% file: fig/eye/vectorized-figures/boudaries_color_free.pdf_tex
\begingroup%
  \makeatletter%
  \providecommand\color[2][]{%
    \errmessage{(Inkscape) Color is used for the text in Inkscape, but the package 'color.sty' is not loaded}%
    \renewcommand\color[2][]{}%
  }%
  \providecommand\transparent[1]{%
    \errmessage{(Inkscape) Transparency is used (non-zero) for the text in Inkscape, but the package 'transparent.sty' is not loaded}%
    \renewcommand\transparent[1]{}%
  }%
  \providecommand\rotatebox[2]{#2}%
  \newcommand*\fsize{\dimexpr\f@size pt\relax}%
  \newcommand*\lineheight[1]{\fontsize{\fsize}{#1\fsize}\selectfont}%
  \ifx\svgwidth\undefined%
    \setlength{\unitlength}{954.79925537bp}%
    \ifx\svgscale\undefined%
      \relax%
    \else%
      \setlength{\unitlength}{\unitlength * \real{\svgscale}}%
    \fi%
  \else%
    \setlength{\unitlength}{\svgwidth}%
  \fi%
  \global\let\svgwidth\undefined%
  \global\let\svgscale\undefined%
  \makeatother%
  \begin{picture}(1,0.72800968)%
    \lineheight{1}%
    \setlength\tabcolsep{0pt}%
    \put(0,0){\includegraphics[width=\unitlength,page=1]{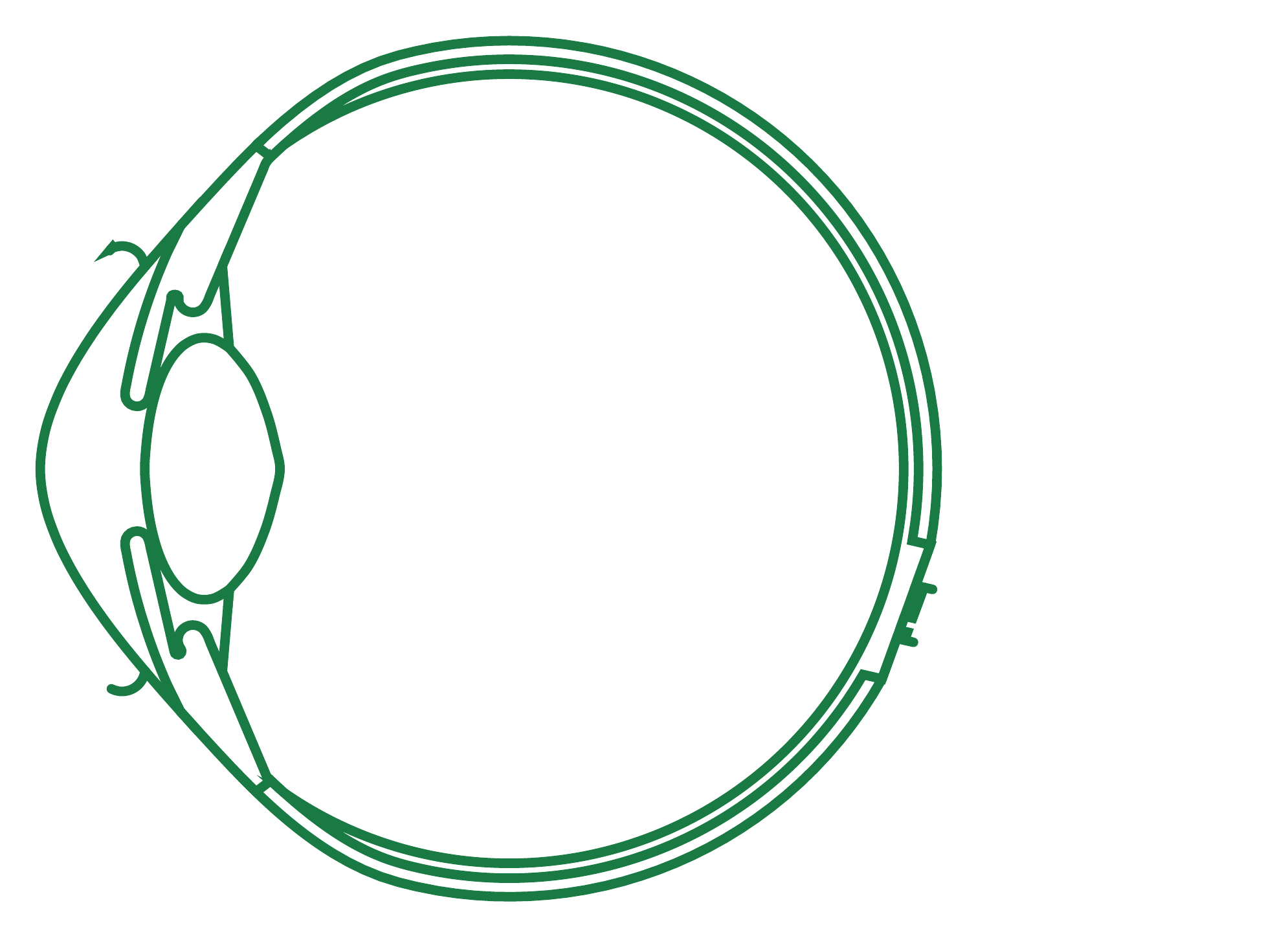}}%
    \put(0,0){\includegraphics[width=\unitlength,page=2]{boundaries_color_free.pdf}}%
    \put(-0.2,0.45){\color[rgb]{0.1372549,0,0.83137255}\makebox(0,0)[lt]{\lineheight{1.25}\smash{\begin{tabular}[t]{l}$\Gamma_\text{amb}$\end{tabular}}}}%
    \put(0,0){\includegraphics[width=\unitlength,page=3]{boundaries_color_free.pdf}}%
    \put(0.8,0.45){\color[rgb]{0.78431373,0,0.2627451}\makebox(0,0)[lt]{\lineheight{1.25}\smash{\begin{tabular}[t]{l}$\Gamma_\text{body}$\end{tabular}}}}%
  \end{picture}%
\endgroup%

%% file: tex/3-model-num.tex
\section{Mathematical and computational framework}
\label{sec:model-num}

This section outlines the mathematical and computational framework, including the variational formulation derivation, the high fidelity finite element method (FEM) resolution technique, and the construction of reduced basis metamodel.
It is followed by the presentation of numerical results, verification and validation steps.

\input{tex/3.1-variational-form}

\paragraph*{High fidelity FEM resolution}

We present here the discretization approach and briefly describe the in-house computational framework we developed.

In the sequel, we set $V:= H^1(\Omega)$ and focus on \emph{outputs of interest}, $s_k(\mu)$, for $k\in\llbracket 1, n_\text{output}\rrbracket$ given by the formula $s_k(\mu) = \ell_k(u(\mu); \mu)$,
where $\ell$ is a bounded linear form and $u(\mu)$ is the solution of \Cref{eq:variational-formulation}.

Denote by $V_h\subset V$ the approximate functional space of dimension $\N$, $h$ standing for the discretization of the space, for a finite element approach.
The previous problem is equivalent to:
\begin{subequations}
\begin{align}
    \mat{A}_L(\mu) \vct{T}^\fem(\mu) &= \vct{f}_L(\mu)\\
    s_k(\mu) &= \vct{L}_k(\mu)^T \vct{T}^\fem(\mu)
\end{align}
\label{eq:fe-syst}
\end{subequations}
with $\mat{A}(\mu)\in\R^{\N\times\N}$, $\vct{f}(\mu)\in\R^{\N}$, $\vct{L}_k(\mu)\in\R^{\N}$, and $k$ is the index of the output.
The vector $\vct{T}^\fem(\mu)\in V_h \simeq\R^{\N}$ is the solution, and $s(\mu)\in\R$ is the computed output.

More precisely, the steps run during resolution are given in \Cref{algo:hf}.

\begin{algorithm}
    \KwIn{$\mu\in\Dmu$}
    Construct $\mat{A}(\mu)$, $\vct{f}(\mu)$, $\vct{L}_k(\mu)$\;
    Solve $\mat{A}(\mu)\vct{T}^\fem(\mu) = \vct{f}(\mu)$\;
    Compute outputs $s_k(\mu) = \vct{L}_k(\mu)^T\vct{T}^\fem(\mu)$\;
    \KwOut{Numerical solution $\vct{T}^\fem(\mu)$ and outputs $s_k(\mu)$}
    \caption{High fidelity resolution.}
    \label{algo:hf}
\end{algorithm}

To establish numerical results, we implement \Cref{algo:hf} in the framework of the open-source library \fpp{} \cite{christophe_prud_homme_2023_8272196}\footnote{Soure code: \url{https://github.com/feelpp/feelpp/}},
and specifically the \textsf{heat} toolbox\footnote{See documentation: \url{https://docs.feelpp.org/toolboxes/latest/heat/toolbox.html}}
where both models $\Em_\text{NL}(\mu)$ and $\Em_\text{L}(\mu)$ can be simulated, with both $\P_1$ and $\P_2$ piece-wise polynomials \cite{Ern2021-mi}.
The solution strategy uses conjugate gradient method solver preconditionned by an algrebraic multigrid method; in the context of $\Em_\text{NL}$, non-linear iterations are required.

All the results presented in this document are available and can be reproduced, refer to \Cref{app:reproduce} for more details.
All subsequent computational simulations are performed on the same machine equipped with the following hardware: \texttt{AMD EPYC 7552 48-Core Processor}.


\subsection{Reduced order modeling with the reduced basis method}
\label{sec:rbm}

We now introduce the reduced basis metamodel \cite{10.1115/1.1448332, Rozza2008-hx, Quarteroni2016}.
The goal of the reduced basis method (RBM) is to approximate the solution of the parametrized-PDE described by Equations (\ref{eq:model:heat})-(\ref{eq:neumann-lin})-(\ref{eq:model:robin})-(\ref{eq:model:interface}).
For complex geometries and biomechanical problems, such as the one described in \Cref{sec:model-geo}, numerical solving has a prohibitive cost, especially for studies of uncertainty quantification, requiring the resolution of the system for many parameters.
We briefly present the implemented strategy, following \cite{10.1115/1.1448332}.

Recall that $\vct{T}^\fem(\mu)\in V_h$ can be written as $\vct{T}^\fem(\mu)=\displaystyle\sum_{n=1}^\N T^\fem_n(\mu)\phi_{h,n}$, where $(\phi_{h,n})_{n\in\llbracket 1,\N\rrbracket}$ is a basis of $V_h$.

The main idea of RBM is to construct a low-dimension subspace $V_N\subset V_h$, of dimension $N$ with $N\ll \N$, such that the approximation error is small:
$\Vert \vct{T}^\fem(\mu) - \vct{T}^\rbm(\mu)\Vert \leqslant \varepsilon_\tol$, while the procedure to compute $\vct{T}^\rbm(\mu)$ is efficient and stable.

The reduced equivalent of the variational form \Cref{eq:variational-formulation} is:
given $\mu\in\Dmu$, find $\vct{T}^{\rbm, N}(\mu)\in V_N$ such that:
\begin{equation}
    a_L(\vct{T}^{\rbm, N}(\mu), \vct{v}; \mu) = f_L(\vct{v}; \mu) \quad \forall \vct{v}\in V_N
    \label{eq:pb-var-rbm-reduced}
\end{equation}

The reduced space $V_N$ is constructed from \emph{snapshots}, which are high fidelity solutions.
The RBM consists of two main phases:
(i) the \emph{offline stage}, where the reduced space is constructed, and
(ii) the \emph{online stage}, where the reduced space is used to compute the solution of the system.
The first step is performed only once and can be costly, while the second step is performed for each parameter $\mu$ and is efficient.

During the offline stage, snapshots are computed for a set of parameters $\{\mu_i\}_{i=1}^N$.
This gives a family of vectors $\big(\vct{T}^\fem(\mu_i)\big)_{1\leqslant i\leqslant N}\subset V_h$.
The reduced space is defined by $V_N := \text{span}\left(\vct{\xi}_i\right)$, where $(\vct{\xi}_i)_{1\leqslant i\leqslant N}$ is an orthonormal family of vectors, obtained by the Gram-Schmidt process applied to the snapshots $\{\vct{T}^\fem(\mu_i)\}_{1\leqslant i\leqslant N}$.
We define the snapshots matrix $\mat{Z}_N = \left[\vct{\xi}_1,\cdots,\vct{\xi}_N\right]\in \R^{\N,N}$.

The snapshots can be selected in different ways.
The first approach is to select the snapshots randomly in the parameter space, but this could lead to a poor approximation of the solution \cite{buffa2012}.
Another approach is to select the snapshots greedily, by selecting the parameter that maximizes the error between the reduced solution and the high fidelity solution, see \Cref{sec:generation-reduced-basis}.

Setting $\mat{A}_N(\mu) = \mat{Z}_N^T \mat{A}_L(\mu) \mat{Z}_N \in \R^{N\times N}$ and $\vct{f}_N(\mu) = \mat{Z}_N^T \vct{f}_L(\mu) \in \R^N$, we obtain the reduced algebraic system of size $N$:
\begin{subequations}
\begin{align}
    \mat{A}_N(\mu) \vct{T}^{\rbm,N}(\mu) &= \vct{f}_N(\mu)\label{eq:fe-syst-reduced:pb}\\
    s_{k,N}(\mu) &= \vct{L}_{k,N}(\mu)^T \vct{T}^{\rbm,N}(\mu)
\end{align}
\label{eq:fe-syst-reduced}
\end{subequations}
the same process applying for the outputs $\vct{L}_k$.

Thanks to the linearity of the model $\Em_\text{L}(\mu)$, we can further write the following \emph{affine decomposition}:
for $T, v \in V$,

\begin{subequations}
\begin{equation}
    a_L(T, v; \mu) = \sum_{q=1}^{Q_a} \beta_A^q(\mu) a^q_L(T, v)
\end{equation}
with
\begin{align}
    \beta_A^1(\mu) &= \prm{k_\text{lens}} & a^1_L(T, v) &= \int_{\Omega_\text{lens}}\nabla T\cdot\nabla v\d x\\
    \beta_A^2(\mu) &= \prm{h_\text{amb}}  & a^2_L(T, v) &= \int_{\Gamma_\text{amb}} Tv\d\sigma\\
    \beta_A^3(\mu) &= \prm{h_\text{bl}}   & a^3_L(T, v) &= \int_{\Gamma_\text{body}} Tv\d\sigma\\
    \beta_A^4(\mu) &= 1                   & a^4_L(T, v) &= \int_{\Gamma_\text{amb}}h_\text{r} Tv\d\sigma + \sum_{i\neq\text{lens}}k_i\int_{\Omega_i}\nabla T\cdot\nabla v\d x
\end{align}
\label{eq:decomposition-a}
\end{subequations}
and
\begin{subequations}
\begin{align}
    f_L(v; \mu) = \sum_{p=1}^{Q_f} \beta_F^p(\mu) f^p_L(v)
\end{align}
with
\begin{align}
    \beta_F^1(\mu) &= \prm{h_\text{amb}T_\text{amb}} + h_\text{r}\prm{T_\text{amb}} - \prm{E} & f^1(v) &= \int_{\Gamma_\text{amb}} v\d\sigma\\
    \beta_F^2(\mu) &= \prm{h_\text{bl}T_\text{bl}} & f^2(v) &= \int_{\Gamma_\text{body}}v\d\sigma
\end{align}
\label{eq:decomposition-f}
\end{subequations}
where $Q_a = 4$ and $Q_f = 2$.
We furthermore define the algebraic matrices $\mat{A}_L^q\in\R^{\N\times\N}$ and vectors $\vct{f}_L^p\in\R^{\N}$,
so the following equality holds:
\begin{equation}
    \mat{A}_L(\mu) = \sum_{q=1}^{Q_a} \beta_A^q(\mu) \mat{A}_L^q, \quad
    \vct{f}_L(\mu) = \sum_{p=1}^{Q_f} \beta_F^p(\mu) \vct{f}_L^p
\end{equation}

From this decomposition and \Cref{eq:fe-syst-reduced:pb}, we obtain the following algebraic system:
\begin{equation}
    \mat{A}_N(\mu) = \sum_{q=1}^{Q_a} \beta_A^q(\mu)\underbrace{\mat{Z}_N^T \mat{A}_L^q\mat{Z}_N}_{\mat{A}^q_N}
    \label{eq:fe-syst-reduced:matrices}
\end{equation}

We set $\mat{A}_N^q := \mat{Z}_N^T \mat{A}_L^q\mat{Z}_N \in \R^{N\times N}$.
The matrices $\mat{A}_N^q\in\R^{N\times N}$ are independent of $\mu$ and can be computed only once and stored.
The same process applies to $\vct{f}_N(\mu)$ and $\vct{L}_{k,N}(\mu)$:

\begin{equation}
    \vct{f}_N(\mu) = \sum_{q=1}^{Q_f} \beta_F^q(\mu) \vct{f}^q_N, \quad
    \vct{L}_{k,N}(\mu) = \sum_{q=1}^{Q_\ell} \beta_\ell^q(\mu) \vct{L}^q_{k,N}
\end{equation}
For the outputs we study in this work, the decomposition of $\vct{L}_{k, N}$ has only one term since the output does not depend on the parameters.

This decomposition allows implementing an \emph{offline/online procedure}.
During the \emph{offline phase}, the basis of $V_N$ is constructed from the snapshots, as well as the matrices $\mat{A}_N^q$, $\vct{f}_N^q$, and $\vct{L}_{k, N}^q$ are computed and stored.
More details about this construction are given in \Cref{sec:generation-reduced-basis}.
This procedure is costly and is performed only once for the problem.
During the \emph{online phase}, the reduced system \Cref{eq:fe-syst-reduced} is solved for any parameter $\mu$.
The entire procedure is synthesized in \Cref{algo:offline-online}.

During the offline stage, two approaches can be used to select the size of the reduced basis $N$:
(i) an approach where we set the size of the reduced basis $N$ to a fixed value, and
(ii) an approach where we set a tolerance $\varepsilon_\tol$ on the error committed on the output.
The second approach is more interesting since it allows having a reduced basis of size $N$ that is adapted to the desired tolerance.

\begin{algorithm}
    \SetKwProg{Fn}{}{:}{}
    \Fn{Offline procedure}{
        \KwIn{Parameters $\mu_1, \cdots, \mu_N\in\Dmu$}
        Compute snapshots $\vct{T}^\fem(\mu_1), \cdots, \vct{T}^\fem(\mu_N)$\;
        Construct $\mat{Z}_N \gets \left[\vct{\xi_1}, \cdots, \vct{\xi_N}\right]$ (orthonormal)\;
        Construct the reduced matrices $(\mat{A}_N^q)_{1\leqslant q\leqslant Q_a}$, $(\vct{F}_N^p)_{1\leqslant p\leqslant Q_f}$, $(\vct{L}_{k, N})_{1\leqslant k\leqslant n_\text{output}}$\;
        \KwOut{Reduced basis and reduced matrices, stored.}
    }

    \Fn{Online procedure}{
        \KwIn{$\mu\in\Dmu$}
        Assemble $\mat{A}_N(\mu)$, $\vct{F}_N(\mu)$, $\vct{L}_{k, N}(\mu)$ using the saved matrices and the affine decomposition\;
        Solve $\mat{A}_N(\mu)\vct{u}_N(\mu) = \vct{F}_N(\mu)$\;
        Compute the output $s_{k,N}(\mu) = \vct{L}_{k, N}(\mu)^T\vct{u}_N(\mu)$\;
        \KwOut{$\vct{u}_N(\mu), s_{k,N}(\mu)$.}
    }

    \caption{Offline and online stages of the RBM.}
    \label{algo:offline-online}
\end{algorithm}

\input{tex/3.2-error-bound}

\subsubsection{Generation of the reduced basis}
\label{sec:generation-reduced-basis}

The a posteriori error estimator introduced earlier provides an efficient criterion to select the desired dimension of the reduced space $N$, in the offline phase.
Given a fixed tolerance $\varepsilon_\tol$, we can greedily select the greatest $N$ such that the error bound is smaller than the tolerance.
In this section, we describe an algorithm to generate the reduced basis.

For this algorithm, a large set of parameters $\Xi_\train\subset \Dmu$ is required.
This set is called the \emph{training set}, and is generated log-randomly.
A first snapshot is computed for a given parameter $\mu_0\in\Dmu$.
To get the $N+1$-th snapshot to be inserted in the basis, we select the parameter $\mu^\star$ that maximizes the error bound $\Delta_N(\mu)$, for $\mu\in\Xi_\train$.
This step is performed until a selected tolerance for the maximal error bound is reached.
The greedy algorithm is presented in \Cref{algo:Greedy}.

\begin{algorithm}
    \caption{Greedy algorithm to construct the reduced basis.}
    \label{algo:Greedy}
    \KwIn{$\mu_0\in \Dmu$, $\Xi_\text{train}\subset \Dmu$ and $\varepsilon_\tol>0$}
    $S\gets [\mu_0]$\;
    \While {$\Delta_N^{\max} > \varepsilon_\tol$}{
        $\mu^\star\leftarrow \displaystyle\argmax_{\mu\in\Xi_\text{train}}\Delta_N(\mu)$ (and $\displaystyle\Delta_N^\text{max}\leftarrow \max_{\mu\in\Xi_\text{train}}\Delta_N(\mu)$)\;
        $V_{N+1} \gets \left\{\vct{\xi} = \vct{T}^\fem(\mu^\star)\right\} \cup V_{N}$\;
        Append $\mu^\star$ to $S$\;
        $N \gets N+1$\;
    }
    \KwOut{Sample $S$, reduced basis $V_N$}
\end{algorithm}

\input{tex/3.3-numerical-results}

%% file: tex/3.1-variational-form.tex
\subsection{Continuous and discrete model}
\label{sec:variational-formaultion}

We first compute the variational formulation of the linearized model $\Em_\text{L}(\mu)$ described in \Cref{sec:model-phys}.

Let $v\in H^1(\Omega)$ be a test function.
As the union $\Omega=\bigsqcup_i \Omega_i$ is disjoint, we have:

\begin{equation}
    \int_\Omega -\nabla\cdot(k\nabla T)v\d x = \sum_i\int_{\Omega_i}-\nabla\cdot(k_i\nabla T)v\d x
\end{equation}

Hence, using Green's theorem:
\begin{subequations}
\begin{equation}
    \sum_i\int_{\Omega_i}-\nabla\cdot(k_i\nabla T)v\d x = 0 \Leftrightarrow \sum_i\int_{\Omega_i}k_i\nabla T\cdot \nabla v\d x - \int_{\partial\Omega_i} k_i\dfrac{\partial T}{\partial\n_i}v\d \sigma = 0
\end{equation}
with boundary and interface conditions \Cref{eq:model:robin,,eq:neumann-lin,,eq:model:interface}, we obtain

\begin{align}
    \nonumber &\sum_i k_i\int_{\Omega_i}\nabla T\cdot\nabla v \d x+ \int_{\Gamma_\text{amb}}\left[h_\text{amb}T + h_\text{r}T\right]v\d \sigma + \int_{\Gamma_\text{body}}h_\text{bl}Tv\d \sigma =\\
        &\qquad\qquad\int_{\Gamma_\text{amb}}\left[h_\text{amb}T_\text{amb} + h_\text{r}T_\text{amb} - E\right]v\d \sigma + \int_{\Gamma_\text{body}}h_\text{bl}T_\text{bl}v\d \sigma
\end{align}
\label{eq:variational-formulation-comp}
\end{subequations}

The previous equation is equivalent to:

\begin{subequations}

\begin{equation}
    a_L(T,v; \mu) = f_L(v; \mu)
\end{equation}
with:
\begin{align}
    \nonumber
    a_L(T, v; \mu) &:= \prm{k_\text{lens}}\int_{\Omega_\text{lens}}\nabla T\cdot \nabla v\d x + \sum_{i\neq\text{lens}}k_i\int_{\Omega_i}\nabla T\cdot\nabla v\d x \\
        &\hspace{4cm}+ \int_{\Gamma_\text{amb}}\left[\prm{h_\text{amb}}T + h_\text{r}T\right]v\d \sigma + \int_{\Gamma_\text{body}}\prm{h_\text{bl}}Tv\d \sigma\\
    f_L(v; \mu) &:= \int_{\Gamma_\text{amb}}\left[\prm{h_\text{amb}}\prm{T_\text{amb}} + h_\text{r}\prm{T_\text{amb}} - \prm{E}\right]v\d \sigma + \int_{\Gamma_\text{body}}\prm{h_\text{bl}}\prm{T_\text{bl}}v\d \sigma
\end{align}
\label{eq:variational-form}
\end{subequations}

The problem statement is therefore: for $\mu\in\Dmu$ given, find the output of interest $s(\mu)\in\R$ given by

\begin{equation}
    s(\mu) = \ell(T(\mu)),
\end{equation}
where $T(\mu)\in H^1(\Omega)$ is solution to
\begin{equation}
    \label{eq:variational-formulation}
    a_L(T(\mu), v; \mu) = f_L(v; \mu)\quad \forall v\in H^1(\Omega).
\end{equation}

The functional $\ell$ returns the desired output of interest, which can be the mean temperature in a selected region
\emph{e.g.\ }$\ell(T(\mu)) = \frac1{|\Omega_\text{cornea}|}\int_{\Omega_\text{cornea}}T(\mu)\d x$,
or the temperature at a fixed point \emph{e.g.\ }$\ell(T(\mu)) = \left<\delta_O, T(\mu)\right>$.

\begin{thm}
    Let $\mu\in\Dmu$ fixed.
    The problem (\ref{eq:variational-form}) is well-posed for $v\in H^1(\Omega)$:
    there exists a unique $T(\mu)\in H^1(\Omega)$ such that $a_L(T(\mu), v; \mu) = f_L(v; \mu)$ for all $v\in H^1(\Omega)$.
    If $T(\mu)\in \mathcal{C}^1(\bar{\Omega})\cap \mathcal{C}^2(\Omega)$, then $T(\mu)$ is solution to problem $\Em_\text{L}(\mu)$.
\end{thm}

\begin{proof}
    The result is a straightforward application of the Lax-Milgram theorem \cite{Ern2021-mi}, and of the regularity of $T$.
\end{proof}

\begin{rem}
    The well-posedness of the fully non-linear problem $\Em_\text{NL}(\mu)$ can also be obtained by the mean of a variational approach, in the spirit of \cite{Milka1993}.
\end{rem}

%% file: tex/3.2-error-bound.tex
\subsubsection{Error estimates}
\label{sec:rbm-error-estimates}

\paragraph*{Error bound}

Given the reduced basis approximation $\vct{T}^{\rbm, N}(\mu)$ of the high fidelity resolution solution $\vct{T}^\fem(\mu)$ for $\mu\in\Dmu$, we defined the \emph{field error}
\begin{equation}
    \vct{e}(\mu):= \vct{T}^\fem(\mu) - \vct{T}^{\rbm, N}(\mu)
    \label{eq:error-field}
\end{equation}

We want to construct quantities $\Delta_N(\mu)$ and $\Delta_N^s(\mu)$ such that

\begin{equation}
    \norm{\vct{e}(\mu)}{V} \leqslant \Delta_N(\mu) \quad\text{and}\quad s(\mu)-s_N(\mu) \leqslant \Delta_N^s(\mu)
\end{equation}

Those quantities are named \emph{a posterior error bounds} \cite{10.1115/1.1448332,Rozza2007}.
To quantify the \emph{sharpness} and \emph{rigor} properties of the error bound, we introduce the \emph{effectivity}:

\begin{equation}
    \eta_N(\mu):= \dfrac{\Delta_N(\mu)}{\norm{\vct{e}(\mu)}{V}}
    \qquad
    \eta_N^s(\mu):= \dfrac{\Delta_N^s(\mu)}{s(\mu)-s_N(\mu)}
    \label{eq:effectivity}
\end{equation}

It has been proven in \cite{10.1115/1.1448332} that the error bound is \emph{rigorous}, namely that it is always greater than the error;
and \emph{sharp}, namely that it is as close as possible to the actual error.

These properties can be summarized as:
\begin{equation}
    1 \leqslant \eta_N(\mu) \leqslant \eta_{\ub}(\mu) \quad \forall \mu\in\Dmu
\end{equation}
where $\eta_{\ub}(\mu)$ is a the sharpness of the bound and is proven to be bounded \cite{10.1115/1.1448332} when $N$ increases.

Finally, to construct the reduced space, we require the error bound to be \emph{efficient}, that is, its evaluation is independent of the size of the high fidelity space $\N$.
This is critical when heuristic algorithms are used to construct the reduced space, such as the Greedy algorithm discussed in \Cref{sec:generation-reduced-basis}.

Such an error bound can be constructed efficiently from the \emph{residual} $r$ of the problem (\ref{eq:variational-form}),
a lower bound $\alpha_\lb(\mu)$ of the coercivity constant $\alpha(\mu)$ of $a_L(\cdot, \cdot; \mu)$, and the affine decomposition of $a_L$ and $f_L$:
\begin{equation}
    \alpha(\mu) = \inf_{v\in V}\frac{a(v, v; \mu)}{\norm{v}{V}^2},
    \quad
    r(v, \mu):= \ell(v; \mu) - a(u_N(\mu), v; \mu) \quad \forall v\in V
    \quad\text{and}\quad
    \Delta_N := \dfrac{\norm{r(\cdot, \mu)}{V'}}{\alpha_\lb(\mu)}
\end{equation}
For more details, refer to \cite{10.1115/1.1448332}.

%% file: tex/3.3-numerical-results.tex
\subsection{Numerical results}
\label{sec:numerical-results}

\input{tex/3.4-HF-validation.tex}

\subsubsection{Linearized model}
\label{sec:linearized-model}

We now compare the results obtained after solving $\Em_\text{NL}(\bar{\mu})$ against $\Em_\text{L}(\bar{\mu})$.
We denote the solution of the nonlinear model $\Em_\text{NL}(\bar{\mu})$ by $T_\text{NL}(\bar{\mu})$,
and by $T_\text{L}(\bar{\mu})$ the solution of the linearized model $\Em_\text{L}(\bar{\mu})$,
and compute the relative error:

\begin{equation}
    e_\text{lin}(\mu) = \dfrac{\norm{T_\text{NL}(\mu)-T_\text{L}(\mu)}{L^2(\Omega)}}{\norm{T_\text{NL}(\mu)}{L^2(\Omega)}}
\end{equation}

For $\mu=\bar{\mu}$, we get $e_\text{lin} = \pgfmathprintnumber[/pgf/number format/.cd, std, fixed zerofill, precision=4]{4.1760143214956727e-07}$.
In \Cref{fig:diffLin}, we plot the difference between the two solutions $|T_\text{NL}(\x) - T_\text{L}(\x)|$ for $\x\in\Omega$.
We notice that the difference is the largest on the front of the eye, where the boundary condition has been changed.
Whereas at the back of the eye, the solutions are superposed.
We also compute the maximal difference: $\pgfmathprintnumber[/pgf/number format/.cd,std,fixed zerofill,precision=4]{0.002166748046875}$~\unit{\kelvin}.
Therefore, we consider in the sequel that the linearized model does not induce a significant error in the results.

\begin{figure}
    \centering
    \begin{tikzpicture}[scale=0.8]
        \begin{axis}[
            colorbar,
            colormap/jet, 
            enlargelimits=false,
            point meta max=0.002166748046875,
            point meta min=0,
            axis line style = {draw=none},
            tick style = {draw=none},
            xtick = \empty, ytick = \empty,
            colorbar style={
                ylabel = {$|T_\text{NL} - T_\text{L}|$},
                height=0.7*\pgfkeysvalueof{/pgfplots/parent axis height},
                at={(1.05,0.5)}, 
                anchor=west, 
            },
            width=0.8\textwidth
        ]
            \addplot graphics [includegraphics cmd=\pgfimage, xmin=0, xmax=1, ymin=0, ymax=1] {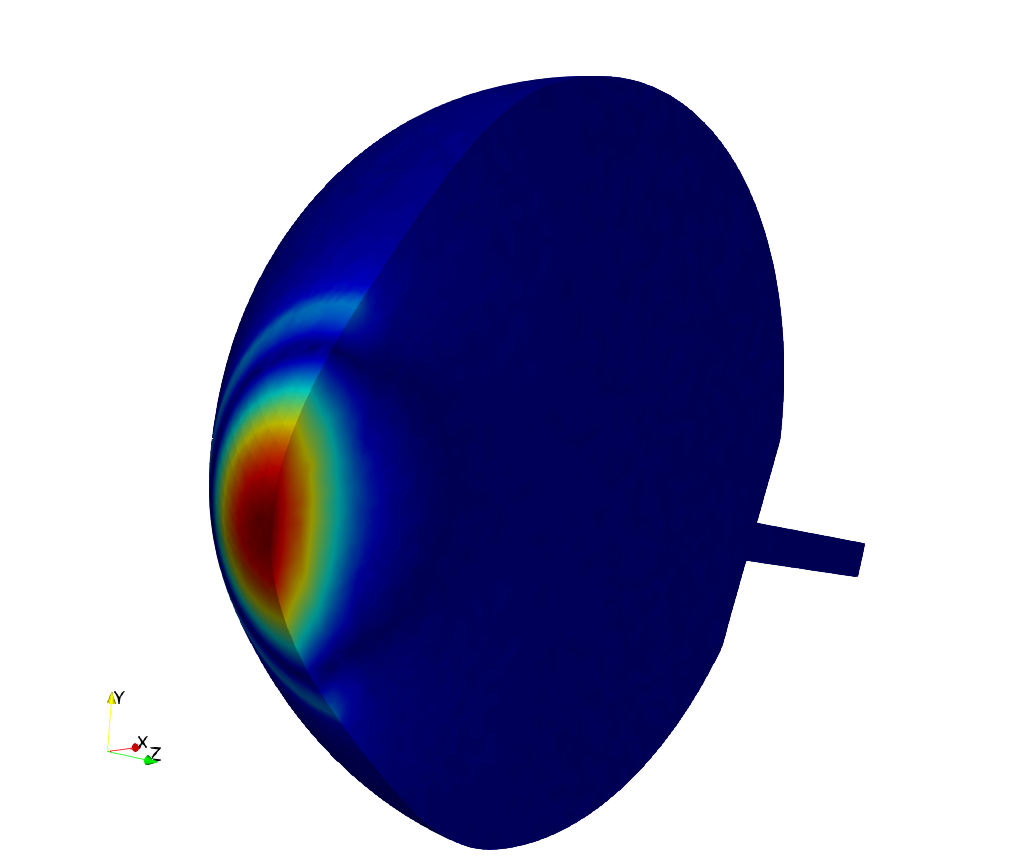};
        \end{axis}
    \end{tikzpicture}
    \caption{Difference of the temperature between the full model and the linearized model, computed on the mesh \texttt{M3} with $\P_2$ elements, and the baseline values $\bar{\mu}$ for the parameters.}
    \label{fig:diffLin}
\end{figure}

\Cref{fig:hf:linear} displays the results of the simulation of the linear model $\Em_\text{L}(\mu)$, for three parameters $\mu$:
$\bar{\mu}$ the baseline value parameters, $\mu_{\text{min}}$ (resp. $\mu_\text{max}$) where each component is the lowest (resp. highest) bound of its range of values.

\begin{figure}
    \centering
    \def\mysize{0.3\textwidth}
    \begin{tabular}{*{3}{c}}
        $\bar{\mu}$ & $\mu_{\text{min}}$ & $\mu_{\text{max}}$\\

        \includegraphics[width=\mysize]{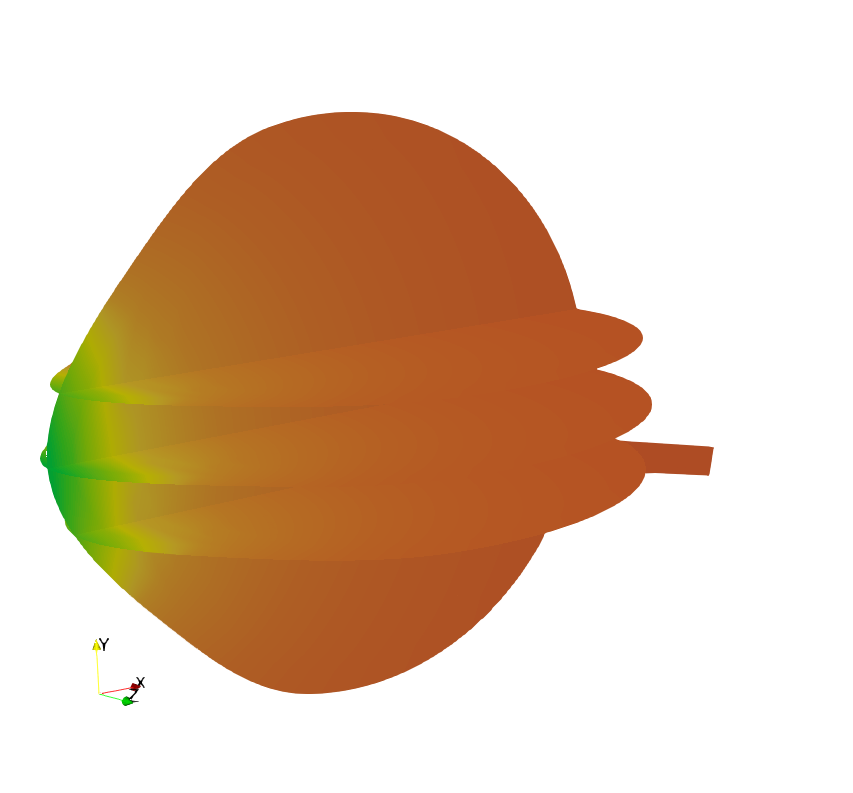} &
        \includegraphics[width=\mysize]{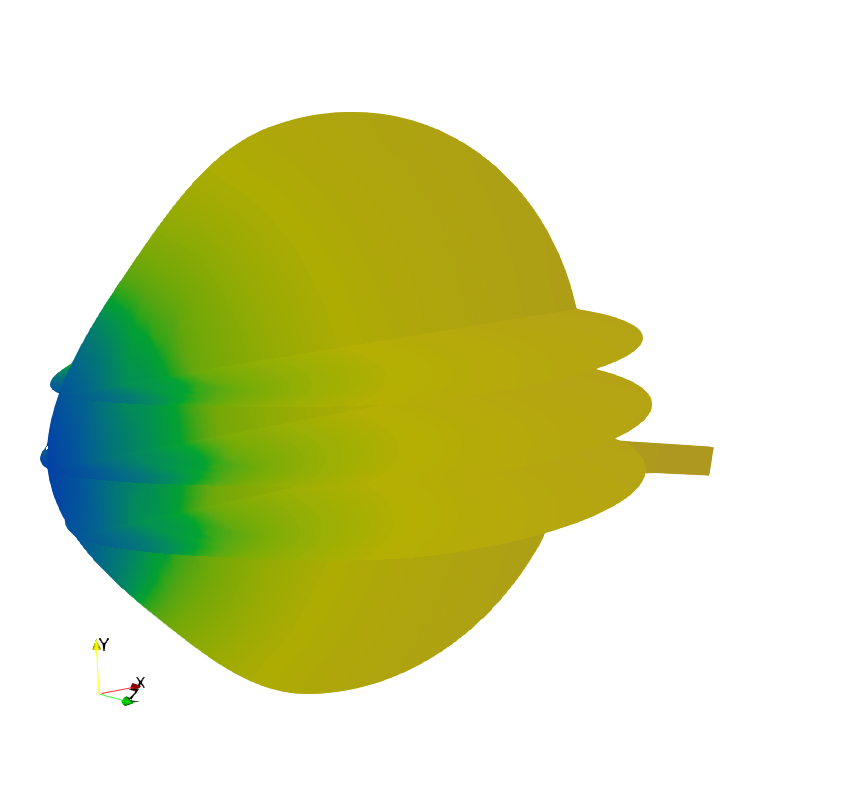} &
        \includegraphics[width=\mysize]{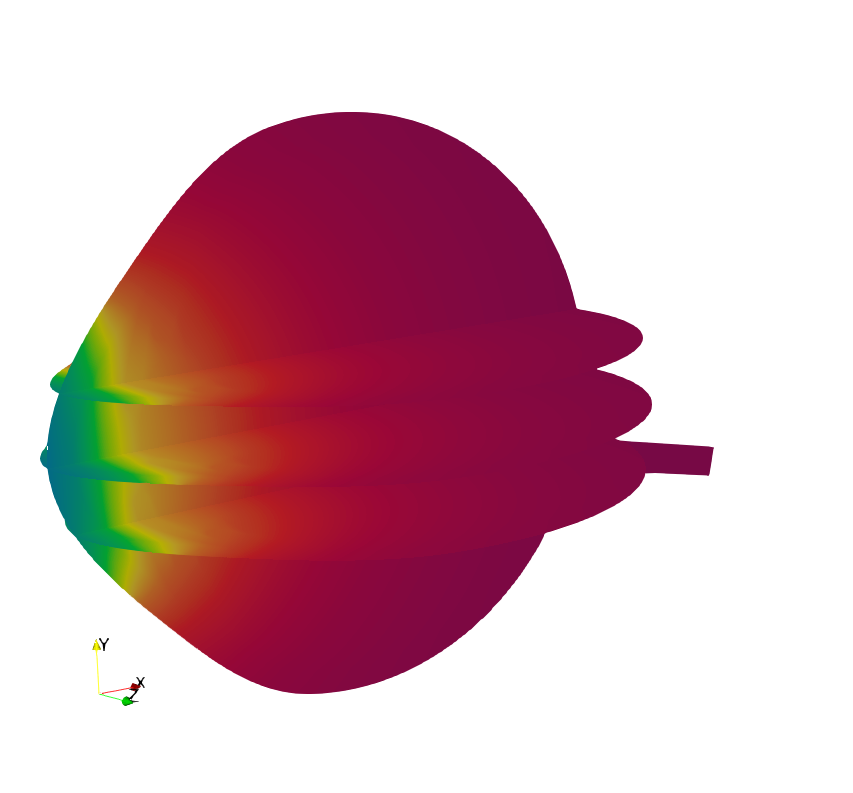} \\

        \multicolumn{3}{c}{\begin{tikzpicture}
            \begin{axis}[
                colorbar horizontal,
                colormap={blues}{rgb255(0.0cm)=(5,97,254);
                rgb255(0.023809523809523808cm)=(5,108,247);
                rgb255(0.047619047619047616cm)=(5,119,239);
                rgb255(0.07142857142857142cm)=(5,130,232);
                rgb255(0.09523809523809523cm)=(5,139,222);
                rgb255(0.11904761904761904cm)=(5,148,212);
                rgb255(0.14285714285714285cm)=(5,157,202);
                rgb255(0.16666666666666666cm)=(5,166,191);
                rgb255(0.19047619047619047cm)=(5,174,179);
                rgb255(0.21428571428571427cm)=(5,183,167);
                rgb255(0.23809523809523808cm)=(5,193,153);
                rgb255(0.2619047619047619cm)=(5,202,140);
                rgb255(0.2857142857142857cm)=(5,211,126);
                rgb255(0.30952380952380953cm)=(5,220,109);
                rgb255(0.3333333333333333cm)=(5,228,91);
                rgb255(0.35714285714285715cm)=(4,237,74);
                rgb255(0.38095238095238093cm)=(69,242,39);
                rgb255(0.40476190476190477cm)=(125,245,28);
                rgb255(0.42857142857142855cm)=(164,249,11);
                rgb255(0.4523809523809524cm)=(194,251,8);
                rgb255(0.47619047619047616cm)=(224,252,5);
                rgb255(0.5cm)=(254,254,3);
                rgb255(0.5238095238095238cm)=(254,243,20);
                rgb255(0.5476190476190477cm)=(254,232,37);
                rgb255(0.5714285714285714cm)=(254,220,55);
                rgb255(0.5952380952380952cm)=(254,208,55);
                rgb255(0.6190476190476191cm)=(254,196,55);
                rgb255(0.6428571428571429cm)=(254,183,55);
                rgb255(0.6666666666666666cm)=(254,171,55);
                rgb255(0.6904761904761905cm)=(254,159,55);
                rgb255(0.7142857142857143cm)=(254,147,55);
                rgb255(0.7380952380952381cm)=(254,132,55);
                rgb255(0.7619047619047619cm)=(254,118,55);
                rgb255(0.7857142857142857cm)=(254,104,55);
                rgb255(0.8095238095238095cm)=(253,84,53);
                rgb255(0.8333333333333334cm)=(251,66,48);
                rgb255(0.8571428571428571cm)=(252,37,53);
                rgb255(0.8809523809523809cm)=(242,29,64);
                rgb255(0.9047619047619048cm)=(230,20,74);
                rgb255(0.9285714285714286cm)=(218,10,84);
                rgb255(0.9523809523809523cm)=(203,11,91);
                rgb255(0.9761904761904762cm)=(189,11,98);
                rgb255(1.0cm)=(174,12,105);},
                colorbar style={
                    width=0.6\textwidth,
                    tick label style={font=\footnotesize},
                },
                ymin=0, ymax=1,
                point meta min=302,
                point meta max=310,
                axis lines=none,
            ]
                \addplot [draw=none] coordinates {(302,0) (310,0)};
            \end{axis}
        \end{tikzpicture}}
    \end{tabular}
    \caption{Distribution of the temperature [\unit{\kelvin}] in the eyeball from the linear model $\Em_\text{L}(\mu)$.}
    \label{fig:hf:linear}
\end{figure}

\subsubsection{Verifications of the reduced basis model}

We compare the results of the reduced basis method with the output of the high fidelity FEM model.
We generate a sample $\Xi_\test$ of 100 parameters in $\Dmu$.
For $\mu\in\Xi_\test$, we compute on the one hand $T_O^\fem(\mu)$, the value of the temperature at point $O$ from the model $\Em_\text{L}(\mu)$,
and on the other hand $T_O^{\rbm, N}(\mu)$, the value of the temperature for the reduced basis model, with a basis of size $N$.
In \Cref{fig:fem-vs-rbm:error}, the value of the error $|T_O^\fem(\mu) - T_O^{\rbm, N}(\mu)|$ is plotted for each $\mu\in\Xi_\test$,
for various reduced basis sizes $N$.
Statistics on the error committed over the sample $\Xi_\test$ are displayed in \Cref{fig:fem-vs-rbm:stats},
as well as the effectivity $\eta_N(\mu)$ in \Cref{fig:fem-vs-rbm:effectivity}.

We observe that even for small values of $N$, the error on the output is remarkably small: an error of $10^{-4}$ is reached for $N=6$.
On the other hand, we find that the convergence on the output is twice as fast as the convergence on the field,
as predicted by the theoretical error estimate \cite[Eq. (36)]{10.1115/1.1448332}.

Note that the anticipated error behavior aligns with theory when the output functional maintains continuity.
In this work, we deviate from the standard case, attributed to the utilization of the Dirac functional in output computation used for pointwise evaluation.
Nevertheless, a similar behavior is observed, and additional insights into this phenomenon will be provided in future work.

\begin{figure}
    \centering
    \def\scl{1}
    \subfigure[Error on RBM for various reduced basis sizes with error bound $\Delta_N^s(\mu)$ \label{fig:fem-vs-rbm:error}]{
        \input{fig/eye/results/convergence/FEM-vs-RBM/convergence_errorbound.tikz}
    }
    \def\scl{0.88}
    \subfigure[Convergence of the errors on the field and the output on point $O$, for $\mu\in\Xi_\test$, for various reduced basis sizes. The maximal, minimal, and mean values are represented.\label{fig:fem-vs-rbm:stats}]{
        \input{fig/eye/results/convergence/FEM-vs-RBM/conv_max.tikz}
    }
    \subfigure[Stability of the effectivity $\eta_N^s(\mu) = \frac{\Delta_N^s(\mu)}{|T^\fem_O(\mu) - T^{\rbm, N}_O(\mu)|}$ and $\eta_N^\text{en}(\mu) = \frac{\Delta_N^\text{en}(\mu)}{\norm{T^\fem(\mu) - T^{\rbm, N}(\mu)}{\mu}}$ for $\mu\in\Xi_\test$, for various reduced basis sizes. The full red line represents the theoretical lower bound of the effectivity, \textrm{i.e.\ } 1.\label{fig:fem-vs-rbm:effectivity}]{
        \input{fig/eye/results/convergence/FEM-vs-RBM/conv_effectivity.tikz}
    }

    \caption{Comparison of the temperature between the full order model and the reduced basis model, tested over a sample $\Xi_\test\subset\Dmu$ of 100 parameters.
    }
    \label{fig:fem-vs-rbm}
\end{figure}

\Cref{tab:time-execution} offers a comparative analysis of execution times for solving the heat transfer problem.
We first discuss the execution times for the high-fidelity solution, encompassing both $\P_1$ and $\P_2$  finite-element discretizations.
The measured time, denoted as $t_\text{exec}$, includes assembling and solving the problem.
In contrast, we also evaluate the execution time of the online phase of our certified reliable reduced basis model.
This comparison highlights a significant reduction in the time required to assemble and solve the problem using our advanced reduced basis approach.
Importantly, this efficiency does not compromise accuracy; the results from the reduced basis model are effectively exact with respect to the high fidelity model.
As anticipated in our earlier scalability analysis, we achieve remarkable computational gains with our model, reinforcing the benefits of our approach in both precision and performance.

\begin{table}
    \centering
    \begin{tabular}{*{5}{c}}
        \toprule
                        & \multicolumn{3}{c}{Finite element resolution}                                    & Reduced model \\
                        & \multicolumn{3}{c}{$T^\fem(\mu)$}                                                & $T^{\rbm, N}(\mu), \Delta_N(\mu)$ \\
        \cmidrule(lr){2-4}
        \cmidrule(lr){5-5}
                        & $\P_1$               & $\P_2$ (\texttt{np=1})         &  $\P_2$ (\texttt{np=12})  &  \\
        \midrule
        Problem size    & $\N = 207~845$       & \multicolumn{2}{c}{$\N = 1~580~932$}                       & $N = 10$ \\
        $t_\text{exec}$ & \qty{5.534}{\second} & \qty{62.432}{\second}          & \qty{10.76}{\second}      & \qty{2.88e-04}{\second}\\
        speed-up        & 11.69                & 1                              & 5.80                      & \textbf{\qty[text-series-to-math]{2.17e5}{}}\\
        \bottomrule
    \end{tabular}
    \caption{Times of execution, using mesh \texttt{M3} for high fidelity simulations.}
    \label{tab:time-execution}
\end{table}

The reduced bases constructed for the various outputs of interest are generated with the greedy algorithm (\Cref{algo:Greedy}),
using a maximal tolerance $\varepsilon_\tol=\pgfmathprintnumber{1e-6}$ and a maximal size for the basis $N = 20$.
In practice, the tolerance is reached for $N = 10$ to 12.

\input{tex/3.5-validation}

%% file: tex/3.4-HF-validation.tex
\subsubsection{Mesh convergence}
\label{sec:verify:mesh-convergence}

In \Cref{sec:model-geo}, we detailed the geometry of the eyeball, derived from computer-aided design (CAD) data.
As illustrated in \Cref{fig:geo-eye}, certain regions exhibit greater complexity than others.
For instance, the lamina cribrosa is notably thinner, while the iris presents a less uniform structure.
Achieving an effective mesh requires a well-distributed arrangement of elements.
This is attainable through the application of a specialized meshing algorithm designed to tailor the mesh according to the geometric intricacies.
Utilizing the MMG library \cite{mmg}, we have generated a family of meshes with varying levels of refinement.
These meshes are used to our subsequent simulation processes.

\Cref{tab:mesh} displays the characteristics of the meshes, such as their characteristic size $h$ and the number of degrees of freedom (nDof) for both $\P_1$ and $\P_2$ finite element discretizations.

\begin{table}
\begin{center}
\begin{tabular}{*{4}{c}}%
    \toprule
    \textbf{Mesh} & $h$ & \textbf{nDof $\P_1$} & \textbf{nDof $\P_2$}\\
    \midrule
    \begin{filecontents*}{fig/eye/dat/meshInformation3D.csv}
mesh,hMin,hAvg,hMax,nDofPUn,nDofPDeux
M0,0.08253314319606485,0.8601407872894455,3.6397105677646873,47284,326928
M1,0.07282756078896943,0.7394462340915404,3.5151197649726473,68993,472693
M2,0.07480314942197873,0.5111165276670979,1.6468131140227362,120581,882826
M3,0.048449681989969025,0.46941219303323484,1.1248355069633602,207845,1580932
M4,0.05113700080520801,0.2938035522149193,0.6010898063737832,995906,7870352
M5,0.030347056736296795,0.15036394143841295,0.3537776109126652,7360346,58655836
\end{filecontents*}
\csvreader[head to column names, late after line=\\]{fig/eye/dat/meshInformation3D.csv}{}{
        \texttt{\mesh} &
        \pgfmathprintnumber{\hAvg} &
        \pgfmathprintnumber{\nDofPUn} &
        \pgfmathprintnumber{\nDofPDeux}
    }
    \bottomrule
\end{tabular}
\end{center}
\caption{Characteristics of the meshes.}
\label{tab:mesh}
\end{table}

In this section, we detail the outcomes of our mesh convergence analysis.
This analysis involves solving the given problem across various mesh configurations and subsequently comparing the resultant data.
To conduct this study, we first solve the model denoted as  $\Em_\text{NL}(\bar{\mu})$.
Following this, we compute the output $T^\fem_{\text{NL}, O}$ representing the temperature at the cornea's center as determined by the high-fidelity model.
The primary objective of this analysis is to ascertain whether the obtained temperature values demonstrate convergence towards a consistent value.
\Cref{fig:conv-output} illustrates the results of our mesh convergence study, clearly indicating a pattern of satisfactory convergence.
We select for further comparisons the values obtained for \texttt{M3} and for $\P_2$.

\begin{figure}
    \centering
    \input{fig/eye/results/convergence/conv_output_nl.tex}
    \caption{Temperature at the center of the cornea computed with the high-fidelity model $\Em_\text{NL}(\bar{\mu})$, depending on the level of refinement of the mesh.}
    \label{fig:conv-output}
\end{figure}

\subsubsection{Scalability}

In this section, we explore the scalability of our computational framework.
This involves measuring the time required to solve the model in relation to the number of MPI parallel processes utilized.
The time measured pertains to the duration necessary for assembling the algebraic system and solving the problem, as per \Cref{algo:hf}.
Our experiments utilized mesh \texttt{M3}, with both $\P_1$ and  $\P_2$ discretizations.
The results, presented in \Cref{fig:scalability}, demonstrate satisfactory scalability:
the execution time decreases as the number of parallel processes increases.
However, we observed that beyond 12 processes, the reduction in execution time becomes less significant.
Consequently, for optimal efficiency, we have selected 12 processes for our subsequent analyses.
This study sets the stage for a subsequent comparison with a reduced-order model, which employs a reduced basis with reliable, certified output bounds derived from the high-fidelity solutions.
This comparison aims to highlight that, while parallel computing can accelerate the high-fidelity computation, the reduced-order approach offers even more substantial computational gains.

\begin{figure}
    \centering
    \def\scl{0.95}
    \input{fig/eye/results/convergence/scalability.tikz}
    \input{fig/eye/results/convergence/scalability_speedup.tikz}
    \caption{Time of execution to run $\Em_\text{L}^\N(\bar{\mu})$ and corresponding speed-up, for an increasing number of parallel processes.
    Simulations are performed on the \texttt{M3}.}
    \label{fig:scalability}
\end{figure}

%% file: fig/eye/results/convergence/conv_output_nl.tex
\begin{tikzpicture}

\definecolor{darkgray176}{RGB}{176,176,176}
\definecolor{darkorange25512714}{RGB}{255,127,14}
\definecolor{lightgray204}{RGB}{204,204,204}
\definecolor{steelblue31119180}{RGB}{31,119,180}

\begin{axis}[
    legend cell align={left},
    legend style={fill opacity=0.8, draw opacity=1, text opacity=1, draw=lightgray204},
    log basis x={10},
    tick align=outside,
    tick pos=left,
    unbounded coords=jump,
    xmajorgrids, ymajorgrids, xminorgrids,
    x grid style={darkgray176},
    xlabel={Number of degrees of freedom $\N$},
    xmode=log,
    xtick style={color=black},
    y grid style={darkgray176},
    ylabel={$T_{\text{NL},O}^\fem$ [K]},
    ymin=309.381830780537, ymax=309.390535547796,
    ytick style={color=black},
    yticklabel style={scaled ticks=false,
                                /pgf/number format/fixed,
                                /pgf/number format/precision=3},
]
\addplot [semithick, colorA, mark=*, mark size=2, mark options={solid}]
table {%
    120581 309.388748368823
    207845 309.387459458495
    995906 309.386518710653
    7360346 309.383311103605
};
\addlegendentry{$\P_1$}
\addplot [semithick, colorB!50!black, mark=*, mark size=4, mark options={solid}, forget plot]
table {%
    1580932 309.382354692412
};
\addplot [semithick, colorB, mark=*, mark size=2, mark options={solid}]
table {%
    326928 309.386607120489
    472693 309.385179086414
    882826 309.382478681363
    1580932 309.382354692412
    7870352 309.382226451776
    58655836 nan
};
\addlegendentry{$\P_2$}
\draw (axis cs:326928,309.386607120489) node[
anchor=base west,
text=black,
rotate=0.0
]{\texttt{M0}};
\draw (axis cs:472693,309.385179086414) node[
    anchor=base west,
    text=black,
    rotate=0.0
]{\texttt{M1}};
\draw (axis cs:120581,309.388748368823) node[
    anchor=base west,
    text=black,
    rotate=0.0
]{\texttt{M2}};
\draw (axis cs:882826,309.382478681363) node[
    anchor=base west,
    text=black,
    rotate=0.0
]{\texttt{M2}};
\draw (axis cs:207845,309.387459458495) node[
    anchor=base west,
    text=black,
    rotate=0.0
]{\texttt{M3}};
\draw (axis cs:1580932,309.382354692412) node[
    anchor=base west,
    text=black,
    rotate=0.0
]{\texttt{M3}};
\draw (axis cs:995906,309.386518710653) node[
    anchor=base west,
    text=black,
    rotate=0.0
]{\texttt{M4}};
\draw (axis cs:7870352,309.382226451776) node[
    anchor=base west,
    text=black,
    rotate=0.0
]{\texttt{M4}};
\draw (axis cs:7360346,309.383311103605) node[
    anchor=base west,
    text=black,
    rotate=0.0
]{\texttt{M5}};

\end{axis}

\end{tikzpicture}

%% file: fig/eye/results/convergence/scalability.tikz
\begin{tikzpicture}[scale=\scl]

    \definecolor{darkgray176}{RGB}{176,176,176}
    \definecolor{lightgray204}{RGB}{204,204,204}

    \begin{axis}[
        legend cell align={left},
        legend style={fill opacity=1, draw opacity=1, text opacity=1, draw=lightgray204},
        tick align=outside,
        tick pos=left,
        ymode=log,
        xmajorgrids, ymajorgrids,
        x grid style={darkgray176},
        xlabel={Number of parallel processes \texttt{np}},
        xtick style={color=black},
        y grid style={darkgray176},
        ylabel={Time to solve the problem [s]},
        ytick style={color=black}
    ]
    \addplot [semithick, colorA, mark=*, mark size=2, mark options={solid}]
        table [x=np, y=P1] {fig/eye/results/convergence/scalability.dat};
    \addlegendentry{$\P_1$}
    \addplot [semithick, colorB, mark=*, mark size=2, mark options={solid}]
        table [x=np, y=P2] {fig/eye/results/convergence/scalability.dat};
    \addlegendentry{$\P_2$}
    \end{axis}

\end{tikzpicture}

%% file: fig/eye/results/convergence/scalability_speedup.tikz
\begin{tikzpicture}[scale=\scl]

    \definecolor{darkgray176}{RGB}{176,176,176}
    \definecolor{lightgray204}{RGB}{204,204,204}

    \begin{axis}[
        legend cell align={left},
        legend style={fill opacity=1, draw opacity=1, text opacity=1, draw=lightgray204, anchor=north west, at={(0.03,0.97)}},
        tick align=outside,
        tick pos=left,
        xmajorgrids, ymajorgrids,
        x grid style={darkgray176},
        xlabel={Number of parallel processes \texttt{np}},
        xtick style={color=black},
        y grid style={darkgray176},
        ylabel={Speed-up},
        ytick style={color=black}
    ]
    \addplot [semithick, colorA, mark=*, mark size=2, mark options={solid}]
        table [x=np, y=speedupP1] {fig/eye/results/convergence/scalability.dat};
    \addlegendentry{$\P_1$}
    \addplot [semithick, colorB, mark=*, mark size=2, mark options={solid}]
        table [x=np, y=speedupP2] {fig/eye/results/convergence/scalability.dat};
    \addlegendentry{$\P_2$}
    \end{axis}

\end{tikzpicture}

%% file: fig/eye/results/convergence/FEM-vs-RBM/convergence_errorbound.tikz
\begin{tikzpicture}[scale=\scl]

\definecolor{3cf210509b}{RGB}{0,128,0}
\definecolor{bc8ae11743}{RGB}{128,0,128}
\definecolor{763bd53097}{RGB}{255,0,0}
\definecolor{416c48ee9e}{RGB}{0,0,255}
\definecolor{5efb101de4}{RGB}{255,165,0}

\pgfplotstableread[col sep=comma]{\currfiledir/data/pointO_errorOnOutput.csv}\val
\pgfplotstableread[col sep=comma]{\currfiledir/data/pointO_errorBound.csv}\err

\begin{axis}[
  xlabel={$T_O^\fem$ [K]},
  ylabel={$|T_O^\fem - T_O^{\rbm, N}|$ [K]},
  ymode=log,
  xmin=289.99322094635886, xmax=306.57046532302854,
  legend columns=5,
  transpose legend,
  legend style={
    fill opacity=0.8,
    draw opacity=1,
    text opacity=1,
    at={(1.03,1.0)},
    anchor=north west
  },
]

  \path[name path=axis] (axis cs:289.99322094635886,1e-9) -- (axis cs:306.57046532302854,1e-9);

  \addplot+ [only marks, name path=N2, mark=*, opacity=0.5, mark options={solid, fill=763bd53097, color=763bd53097}]
    table [col sep=comma, x=T_fem, y=2]
    {\val};
  \addlegendentry{$N=2$}

  \addplot+ [mark=none, name path=error2, opacity=0.2, forget plot]
    table [col sep=comma, x=T_fem, y=2]
    {\err};

  \addplot+ [only marks, name path=N4, mark=*, opacity=0.5, mark options={solid, fill=416c48ee9e, color=416c48ee9e}]
    table [col sep=comma, x=T_fem, y=4]
    {\val};
  \addlegendentry{$N=4$}

  \addplot+ [mark=none, name path=error4, opacity=0.2, forget plot]
    table [col sep=comma, x=T_fem, y=4]
    {\err};

  \addplot+ [only marks, name path=N6, mark=*, opacity=0.5, mark options={solid, fill=3cf210509b, color=3cf210509b}]
    table [col sep=comma, x=T_fem, y=6]
    {\val};
  \addlegendentry{$N=6$}

  \addplot+ [mark=none, name path=error6, opacity=0.2, forget plot]
    table [col sep=comma, x=T_fem, y=6]
    {\err};

  \addplot+ [only marks, name path=N8, mark=*, opacity=0.5, mark options={solid, fill=5efb101de4, color=5efb101de4}]
    table [col sep=comma, x=T_fem, y=8]
    {\val};
  \addlegendentry{$N=8$}

  \addplot+ [mark=none, name path=error8, opacity=0.2, forget plot]
    table [col sep=comma, x=T_fem, y=8]
    {\err};

  \addplot+ [only marks, name path=N10, mark=*, opacity=0.5, mark options={solid, fill=bc8ae11743, color=bc8ae11743}]
    table [col sep=comma, x=T_fem, y=10]
    {\val};
  \addlegendentry{$N=10$}

  \addplot+ [mark=none, name path=error10, opacity=0.2, forget plot]
    table [col sep=comma, x=T_fem, y=10]
    {\err};

  \addplot [thick, color=763bd53097, fill=763bd53097, fill opacity=0.2] fill between [of=error2 and error4];
  \addlegendentry{$\Delta_2^s(\mu)$}
  \addplot [thick, color=416c48ee9e, fill=416c48ee9e, fill opacity=0.2] fill between [of=error4 and error6];
  \addlegendentry{$\Delta_4^s(\mu)$}
  \addplot [thick, color=3cf210509b, fill=3cf210509b, fill opacity=0.2] fill between [of=error6 and error8];
  \addlegendentry{$\Delta_6^s(\mu)$}
  \addplot [thick, color=5efb101de4, fill=5efb101de4, fill opacity=0.2] fill between [of=error8 and error10];
  \addlegendentry{$\Delta_8^s(\mu)$}
  \addplot [thick, color=bc8ae11743, fill=bc8ae11743, fill opacity=0.2] fill between [of=error10 and axis];
  \addlegendentry{$\Delta_{10}^s(\mu)$}

\end{axis}
\end{tikzpicture}

%% file: fig/eye/results/convergence/FEM-vs-RBM/conv_max.tikz
\pgfplotstableread[col sep=comma]{\currfiledir/data/convergence_error_O.csv}\errors
\begin{tikzpicture}[scale=\scl]

\begin{axis}[
    log basis y={10},
    tick align=outside,
    tick pos=left,
    x grid style={darkgray176},
    xlabel={$N$},
    xmin=1.6, xmax=10.4,
    xtick style={color=black},
    y grid style={darkgray176},
    ylabel={Error computed},
    ymode=log,
    ytick style={color=black},
    legend style={fill opacity=0.8, draw opacity=1, text opacity=1, draw=lightgray204, nodes={scale=0.8, transform shape}}
]
\addplot [name path=Mmax, semithick, steelblue31119180, mark=none, mark size=3, mark options={solid}, forget plot] table[y=outputMax] {\errors};
\addplot [name path=Mmin, semithick, steelblue31119180, mark=none, mark size=3, mark options={solid}, forget plot] table[y=outputMin] {\errors};
\addplot [name path=Mmean, semithick, steelblue31119180, mark=*, mark size=3, mark options={solid}, dashed, forget plot] table[y=outputMean] {\errors};
\addplot [color=steelblue31119180, fill=steelblue31119180, opacity=0.2] fill between [of=Mmax and Mmin];
\addlegendentry{$|T^\fem_O(\mu) - T^{\rbm, N}_O(\mu)|$}
\addplot+[colorB, draw=none, mark=none, forget plot] table[x=N, y={create col/linear regression={y=outputMean}}]{\errors};
\xdef\slopeO{\pgfplotstableregressiona}

\addplot [name path=MFmax, semithick, forestgreen4416044, mark=none, mark size=3, mark options={solid}, forget plot] table[y=fieldMax] {\errors};
\addplot [name path=MFmin, semithick, forestgreen4416044, mark=none, mark size=3, mark options={solid}, forget plot] table[y=fieldMin] {\errors};
\addplot [name path=MFmean, semithick, forestgreen4416044, mark=*, mark size=3, mark options={solid}, dashed, forget plot] table[y=fieldMean] {\errors};
\addplot [color=forestgreen4416044, fill=forestgreen4416044, opacity=0.2] fill between [of=MFmax and MFmin];
\addlegendentry{$\norm{T^\fem(\mu) - T^{\rbm, N}(\mu)}{\mu}$}
\addplot+[colorB, draw=none, mark=none, forget plot] table[x=N, y={create col/linear regression={y=fieldMean}}]{\errors};
\xdef\slopeField{\pgfplotstableregressiona}

\logLogSlopeTriangle{0.15}{0.1}{0.2}{\slopeO}{steelblue31119180}
\logLogSlopeTriangle{0.15}{0.1}{0.1}{\slopeField}{forestgreen4416044}

\end{axis}

\end{tikzpicture}

%% file: fig/eye/results/convergence/FEM-vs-RBM/conv_effectivity.tikz
\begin{tikzpicture}[scale=\scl]

\definecolor{darkgray176}{RGB}{176,176,176}

\begin{axis}[
log basis y={10},
tick align=outside,
tick pos=left,
x grid style={darkgray176},
xlabel={$N$},
xmin=1.6, xmax=10.4,
xtick style={color=black},
y grid style={darkgray176},
ylabel={$\eta_N(\mu)$},
ymode=log,
ytick style={color=black},
legend style={fill opacity=0.8, draw opacity=1, text opacity=1, draw=lightgray204}
]
\addplot [name path=Mmax, semithick, steelblue31119180, mark=none, mark size=3, mark options={solid}, forget plot]
table {%
2 1066.8884150652973
4 275.5130537402244
6 5175.283748411159
8 227.51000861915077
10 327.78324189653006
};
\addplot [name path=Mmin, semithick, steelblue31119180, mark=none, mark size=3, mark options={solid}, forget plot]
table {%
2 3.6036293609127763
4 4.409972845317161
6 2.769743503906592
8 2.918187278624139
10 5.222510930047634
};
\addplot [name path=Mmean, semithick, steelblue31119180, mark=*, mark size=3, mark options={solid}, dashed, forget plot]
table {%
2 38.83312731510946
4 13.689024117830636
6 87.9511500702313
8 16.080183702477342
10 15.459982089346433
};
\addplot [color=steelblue31119180, fill=steelblue31119180, opacity=0.2] fill between [of=Mmax and Mmin];
\addlegendentry{$\eta_N^s(\mu)$}

\addplot [name path=MFmax, semithick, forestgreen4416044, mark=none, mark size=3, mark options={solid}, forget plot]
table {%
2 1.5343413306443425
4 1.4774853602885867
6 1.4666194896909972
8 1.4666340697067723
10 1.5079660029253543
};
\addplot [name path=MFmin, semithick, forestgreen4416044, mark=none, mark size=3, mark options={solid}, forget plot]
table {%
2 1.3528010083235558
4 1.3138862049682665
6 1.2603680389132252
8 1.3041573261775927
10 1.2920726294213412
};
\addplot [name path=MFmean, semithick, forestgreen4416044, mark=*, mark size=3, mark options={solid}, dashed, forget plot]
table {%
2 1.439823770140758
4 1.4082529053910926
6 1.407758311277011
8 1.3879751517983197
10 1.4039526659084272
};
\addplot [color=forestgreen4416044, fill=forestgreen4416044, opacity=0.2] fill between [of=MFmax and MFmin];
\addlegendentry{$\eta_N^\text{en}(\mu)$}

\addplot[mark=none, red, samples=2] table {
    1.6 1
    10.4 1
};
\end{axis}

\end{tikzpicture}

%% file: tex/3.5-validation.tex
\subsection{Validation and comparison with previous studies}
\label{sec:validation}

We present in this section a thorough comparison between the results of this work and previously published data on the temperature of the eye,
obtained either by experimental procedures or via computational modeling.
Note that only scarce data are available for the entire human eyeball, since most of the measurement techniques estimated only the surface temperature of the cornea.
In particular, \cite{Efron1989OcularST} gathers the outputs of 19 studies conducted with various instruments (mercury bulbs, liquid crystal thermometers or infrared thermometers),
and the mean value reported, according to \cite{NG2006268}, is $T_O^\text{exp} = \qty{307.15}{\kelvin}$.
The temperature at the center of the cornea computed with baseline value from our model is $T_O^\fem(\bar{\mu}) = \pgfmathprintnumber{306.0215598687341}~\unit{\kelvin}$,
which lies in the interval of results found in the literature (see \cite[Table 1]{Efron1989OcularST} and \cite[Table 9]{NG2006268}).

Additionally, in \cite{Efron1989OcularST}, the temperature is measured along an imaginary horizontal line, the \emph{Geometrical Center of the Cornea} (GCC), as described in \Cref{fig:geo-eye}, on a panel of 21 subjects.
The experimental data are displayed in \Cref{fig:res-gcc}, together with the findings of the present work. 
On the horizontal axis, the distance to the center of the eye is represented, and on the vertical axis is the temperature difference to the central one (mean value and standard deviation).
Note that as the geometry of the simulated eye is not the same as the one used in the experiment, we scaled the results over the $x$-axis.
The result shows that the high fidelity model is able to closely replicate the same behavior as the one experimentally measured,
and the model $\Em_\text{L}(\bar{\mu})$ provides very close values (see \Cref{sec:linearized-model}).
Moreover, thanks to the error bound introduced in \Cref{sec:rbm-error-estimates} for the RBM, the approach is considered to be valid for the sensitivity analysis procedure hereafter.

\pgfplotstableread[col sep=comma]{fig/eye/results/GCC/gcc_efron.csv}\gccEfron
\pgfplotstableread[col sep=comma]{fig/eye/results/GCC/gccFeelFullP2.csv}\gccfeel
\pgfplotstableread[col sep=comma]{fig/eye/results/GCC/gcc_li_3D.csv}\gccli

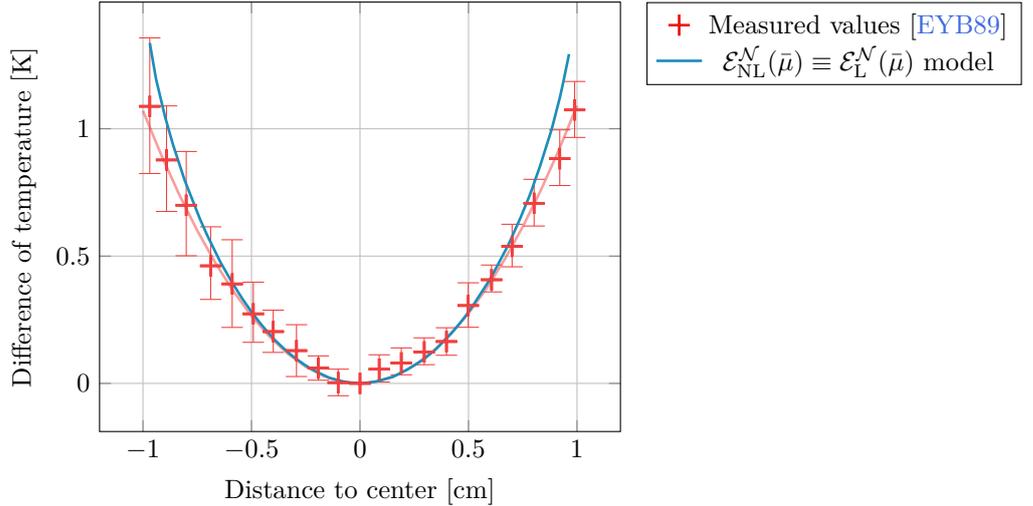
\begin{figure}
    \centering
    \begin{tikzpicture}[trim axis left, trim axis right]
        \begin{axis}
        [legend style = { at = {(1.05,1)}, anchor=north west, nodes={scale=1, transform shape}},
        xlabel = {Distance to center [\unit{\cm}]}, ylabel = {Difference of temperature [\unit{\kelvin}]}, grid = both,
        xmin=-1.2e-2, xmax=1.2e-2, domain=-1e-2:1e-2,
        scaled x ticks={real:1e-2}, xtick scale label code/.code={\,}
        ]
        \addplot+[colorB, mark=none, line width=1pt, opacity=0.4, forget plot] {10807.7*x*x + 0.999464*x};
        \addplot+[colorB, only marks, mark=+, line width=1pt, mark size=4pt] table[col sep=comma,x=dx, y=mean_difference] {\gccEfron};
        \addlegendentry{Measured values \cite{Efron1989OcularST}}
        \addplot+[colorB!80, only marks, mark=+, line width=1pt, mark size=4pt, forget plot] plot [error bars/.cd, y dir=both, y explicit]
            table[col sep=comma,x=dx, y=mean_difference, y error plus=std_up, y error minus=std_down] {\gccEfron};

        \addplot+[colorFeel3, mark=none, line width=1pt, mark size=4pt] table[col sep=comma,x=dx_scaled, y=dT] {\gccfeel};
        \addlegendentry{$\Em_\text{NL}^\N(\bar{\mu}) \equiv \Em_\text{L}^\N(\bar{\mu})$ model}

        \end{axis}
    \end{tikzpicture}
    \caption{Temperature on the GCC: experimental data (mean and standard deviation) vs. numerical results. In this analysis, we cannot distinguish graphical difference between the linear and non-linear models.}
    \label{fig:res-gcc}

\end{figure}

In \Cref{fig:res:line}, we present a comparative analysis between the results of our current study and various numerical findings reported in existing literature.
This comparison features temperatures calculated along a line traversing the eye’s center, the specific location of which is depicted in \Cref{fig:outputs}.
This comparative approach is crucial as it verifies the accuracy of our computed values, encompassing not just the corneal surface but also the eye’s internal tissue structures.
It is noteworthy that our analysis includes a mix of both 2D and 3D results, derived from both non-linear and linearized models.

\begin{figure}
    \centering
    \input{fig/eye/results/line/line.tikz}
    \caption{Temperature on a line going through the center of the eye (see \Cref{fig:outputs}), comparison with numerical results from literature.}
    \label{fig:res:line}
\end{figure}
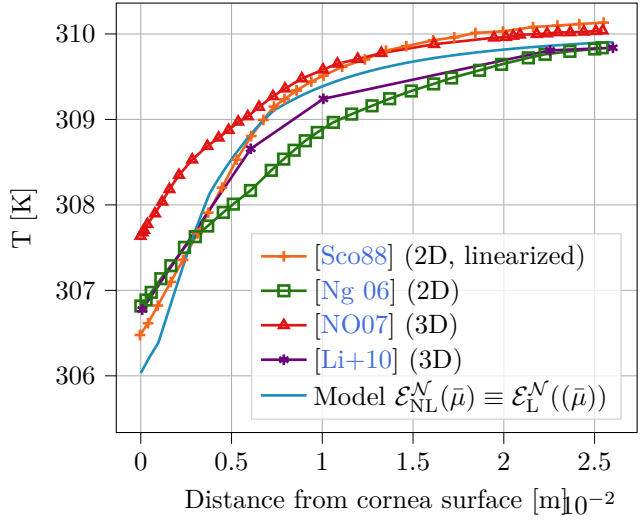

The results show a very good agreement between the findings of the present study and previously reported temperature results,
along the different locations in the eyeball.

%% file: fig/eye/results/line/line.tikz
\begin{tikzpicture}

\begin{axis}[
legend cell align={left},
legend style={
  fill opacity=0.8,
  draw opacity=1,
  text opacity=1,
  at={(0.97,0.03)},
  anchor=south east,
  draw=lightgray204
},
tick align=outside,
tick pos=left,
unbounded coords=jump,
x grid style={darkgray176},
xlabel={Distance from cornea surface [\unit{\meter}]},
xmajorgrids,
xmin=-0.00134906503070622, xmax=0.0273460256000898,
xtick style={color=black},
y grid style={darkgray176},
ylabel={T [\unit{\kelvin}]},
ymajorgrids,
ymin=305.336744421907, ymax=310.361140249203,
ytick style={color=black}
]

\addplot [semithick, colorScott, mark=+, line width=1pt]
table {%
-4.47427293064e-05 306.477586206897
0.0004026845637583 306.615517241379
0.0009395973154362 306.822413793103
0.00165548098434 307.098275862069
0.0023266219239373 307.356896551724
0.0031319910514541 307.66724137931
0.0037136465324384 307.908620689655
0.0044742729306487 308.201724137931
0.0052796420581655 308.529310344828
0.0060850111856823 308.805172413793
0.0067561521252796 308.994827586207
0.0073378076062639 309.15
0.0079194630872483 309.236206896552
0.0085906040268456 309.339655172414
0.0085906040268456 309.339655172414
0.0093959731543624 309.443103448276
0.0100671140939597 309.512068965517
0.0110961968680089 309.615517241379
0.0123489932885906 309.701724137931
0.0135123042505592 309.805172413793
0.0146308724832214 309.856896551724
0.0161073825503355 309.925862068966
0.0172706935123042 309.960344827586
0.0184340044742729 310.012068965517
0.0199552572706935 310.029310344828
0.0216107382550335 310.081034482759
0.0229530201342281 310.098275862069
0.0241610738255033 310.115517241379
0.0255033557046979 310.13275862069
};
\addlegendentry{\cite{Scott_1988} (2D, linearized)}

\addplot [semithick, colorOoi, mark=square, line width=1pt]
table {%
4.8283349055e-06 306.817436489607
0.0002914281299808 306.886720554273
0.0005781803987899 306.976789838337
0.0011075183781673 307.136143187067
0.0016588125752203 307.28856812933
0.0024086275737566 307.503348729792
0.0030037325570048 307.628060046189
0.0036868657092671 307.752771362587
0.0045463093224473 307.912124711316
0.0050971969228767 308.009122401848
0.006044668705071 308.168475750577
0.0071907629379045 308.404041570439
0.0078739469147448 308.535681293303
0.0084468923820056 308.6396073903
0.0090639027583514 308.750461893764
0.0097248255700484 308.847459584296
0.0106059712780145 308.965242494226
0.0116190067657678 309.062240184757
0.0127420774647887 309.159237875289
0.013710997218879 309.242378752887
0.0149000382200826 309.332448036951
0.0161330424812152 309.41558891455
0.0171238676284032 309.484872979215
0.018665021305663 309.574942263279
0.0198759168753862 309.644226327945
0.0213509477767296 309.720438799076
0.0222755484988452 309.762009237875
0.0236402392414533 309.796651270208
0.0250048791594834 309.824364896074
0.0254231146114562 309.838221709007
};
\addlegendentry{\cite{NG2006268} (2D)}

\addplot [semithick, colorFeel2, mark=triangle, line width=1pt]
table {%
-1.11814071495e-05 307.634988452656
9.91587515857e-05 307.676558891455
0.0001873902189116 307.704272517321
0.0003639548027193 307.773556581986
0.0007830034479393 307.898267898383
0.0011801467000618 308.036836027714
0.0015773407767621 308.18233256351
0.0021067295807175 308.348614318707
0.0028342834141105 308.528752886836
0.0036717199850372 308.688106235566
0.0042886287122271 308.785103926097
0.0048394654880785 308.875173210162
0.0053683460462544 308.972170900693
0.0058969724815405 309.034526558891
0.0065139828578863 309.145381062356
0.0072851441791627 309.270092378753
0.0079460161662817 309.360161662818
0.0088711759587548 309.477944572748
0.0100383115668607 309.581870669746
0.010831073333767 309.651154734411
0.0119538390853202 309.706581986143
0.0132527628240575 309.775866050808
0.0161584547701915 309.879792147806
0.0194600701785772 309.956004618938
0.0201203322707608 309.962933025404
0.0205165606804801 309.976789838337
0.0208907820479458 309.99064665127
0.0218151286471717 309.997575057737
0.0223873625703412 310.004503464203
0.0229816543603421 310.018360277136
0.0236638726702013 310.018360277136
0.024412162931399 310.025288683603
0.0250724250235826 310.032217090069
0.0254465955664704 310.039145496536
};
\addlegendentry{\cite{NG2007829} (3D)}

\addplot [semithick, colorLi, mark=asterisk, line width=1pt]
table {%
7.85821448835e-05 306.775105385357
0.0060603629224551 308.656386300426
0.0100614907459299 309.241114038966
0.0225482370796333 309.807932358439
0.0260095412473588 309.838407958814
};
\addlegendentry{\cite{li2010} (3D)}

\addplot [semithick, colorFeel3, line width=1pt]
table {%
-3.37781384600003e-05 nan
6.22123479999265e-07 306.0322265625
3.50223854199989e-05 306.0455932617
6.94226473600002e-05 306.0589599609
0.0001038229093 306.0723571777
0.00013822410256 306.085723877
0.000172624364499999 306.0990905762
0.000207024626439999 306.1124572754
0.00024142488837 306.1258239746
0.000275825150309999 306.1391906738
0.000310225412249999 306.1525878906
0.00034462567419 306.1660461426
0.00037902593613 306.179473877
0.00041342712939 306.1929321289
0.000447827391329999 306.206237793
0.000482227653269999 306.2190856934
0.00051662791521 306.2309875488
0.000551028177139999 306.2428588867
0.000585428439079999 306.2547302246
0.00061982870102 306.2666015625
0.00065422896296 306.2784729004
0.000688629224899999 306.2903747559
0.000723030418159999 306.3022460938
0.000757430680099999 306.3141174316
0.00079183094204 306.3259887695
0.00082623120398 306.337890625
0.000860631465909999 306.3497619629
0.00089503172785 306.3616333008
0.00092943198979 306.3735046387
0.000963832251729999 306.3853759766
0.000998233444989999 306.407409668
0.00103263370693 306.4299316406
0.00106703396887 306.4524230957
0.00110143423081 306.4749450684
0.00113583449275 306.4974365234
0.00117023475468 306.5199584961
0.00120463501662 306.5424499512
0.00123903527856 306.5649719238
0.0012734355405 306.5874633789
0.00130783673376 306.6099853516
0.0013422369957 306.6327209473
0.00137663725764 306.6555175781
0.00141103751958 306.6782836914
0.00144543778152 306.7010803223
0.00147983804345 306.7238464355
0.00151423830539 306.7466430664
0.00154863856733 306.7694396973
0.00158303882927 306.7922058105
0.00161744002253 306.815032959
0.00165184028447 306.8378601074
0.00168624054641 306.8606872559
0.00172064080835 306.8834838867
0.00175504107028 306.9061279297
0.00178944133222 306.9287414551
0.00182384159416 306.951385498
0.0018582418561 306.9739990234
0.00189264304936 306.9966125488
0.0019270433113 307.0192260742
0.00196144357324 307.0418701172
0.00199584383518 307.0644836426
0.00203024409712 307.087097168
0.00206464435905 307.1097106934
0.00209904462099 307.1323242188
0.00213344488293 307.1549377441
0.00216784514487 307.1773071289
0.00220224633813 307.1996154785
0.00223664660007 307.2219543457
0.00227104686201 307.2441711426
0.00230544712395 307.2661437988
0.00233984738589 307.2879638672
0.00237424764782 307.3097839355
0.00240864790976 307.3316040039
0.0024430481717 307.3534240723
0.00247744936496 307.3752441406
0.0025118496269 307.397064209
0.00254624988884 307.4189453125
0.00258065015078 307.4409484863
0.00261505041272 307.4629516602
0.00264945067466 307.4849853516
0.00268385093659 307.5063476562
0.00271825119853 307.5268554688
0.00275265146047 307.5473937988
0.00278705265373 307.5679321289
0.00282145291567 307.588470459
0.00285585317761 307.6089782715
0.00289025343955 307.6295166016
0.00292465370149 307.6502380371
0.00295905396343 307.6709594727
0.00299345422536 307.6917724609
0.0030278544873 307.7116699219
0.00306225568056 307.7312011719
0.0030966559425 307.7507324219
0.00313105620444 307.7702636719
0.00316545646638 307.7898254395
0.00319985672832 307.8093566895
0.00323425699026 307.8291320801
0.00326865725219 307.8489990234
0.00330305751413 307.8688659668
0.00333745777607 307.8880004883
0.00337185896933 307.9067993164
0.00340625923127 307.9255371094
0.00344065949321 307.9442443848
0.00347505975515 307.9629516602
0.00350946001709 307.9816894531
0.00354386027903 308.0003967285
0.003578260540965 308.0191040039
0.003612660802903 308.0378112793
0.003647061996164 308.0565490723
0.003681462258103 308.0752563477
0.003715862520042 308.0938110352
0.00375026278198 308.1122131348
0.003784663043919 308.1293640137
0.003819063305857 308.1416625977
0.003853463567796 308.1539916992
0.003887863829734 308.1663208008
0.003922264091673 308.1786193848
0.003956665284934 308.1909484863
0.003991065546873 308.203338623
0.004025465808811 308.2158508301
0.00405986607075 308.2282714844
0.004094266332688 308.2402954102
0.004128666594627 308.2519836426
0.004163066856566 308.263671875
0.004197467118504 308.2753601074
0.004231868311765 308.2870483398
0.004266268573704 308.2987670898
0.004300668835642 308.3104553223
0.004335069097581 308.3221435547
0.00436946935952 308.3338317871
0.004403869621458 308.3451538086
0.004438269883397 308.3564453125
0.004472670145335 308.3676757812
0.004507070407274 308.37890625
0.004541471600535 308.3901367188
0.004575871862474 308.4011535645
0.004610272124412 308.4120178223
0.004644672386351 308.4228820801
0.004679072648289 308.4337463379
0.004713472910228 308.4446105957
0.004747873172166 308.4555358887
0.004782273434105 308.4664916992
0.004816673696044 308.4771118164
0.004851074889305 308.4873962402
0.004885475151243 308.4977111816
0.004919875413182 308.5079956055
0.00495427567512 308.5183105469
0.004988675937059 308.5286865234
0.005023076198998 308.5390319824
0.005057476460936 308.549407959
0.005091876722875 308.559753418
0.005126277916136 308.5701293945
0.005160678178074 308.5805053711
0.005195078440013 308.5905761719
0.005229478701952 308.6003112793
0.00526387896389 308.6097717285
0.005298279225829 308.6192321777
0.005332679487767 308.628692627
0.005367079749706 308.6381835938
0.005401480011644 308.6476745605
0.005435881204906 308.6571960449
0.005470281466844 308.6667175293
0.005504681728783 308.6763000488
0.005539081990721 308.6858520508
0.00557348225266 308.6954345703
0.005607882514598 308.7050476074
0.005642282776537 308.7146911621
0.005676683038476 308.7243041992
0.005711084231737 308.7339477539
0.005745484493675 308.743560791
0.005779884755614 308.7532043457
0.005814285017552 308.7625732422
0.005848685279491 308.7714538574
0.00588308554143 308.7803039551
0.005917485803368 308.788848877
0.005951886530968 308.797454834
0.005986286792907 308.8061218262
0.006020687054845 308.8147583008
0.006055087316784 308.823425293
0.006089488044384 308.8320617676
0.006123888306322 308.8406982422
0.006158288568261 308.8493652344
0.006192688830199 308.858001709
0.006227089557799 308.8666687012
0.006261489819738 308.875213623
0.006295890081676 308.8832702637
0.006330290343615 308.8913269043
0.006364691071215 308.8993835449
0.006399091333153 308.907409668
0.006433491595092 308.9152832031
0.00646789185703 308.9231262207
0.006502292118969 308.9309997559
0.006536692846569 308.938873291
0.006571093108507 308.9467773438
0.006605493370446 308.9548034668
0.006639893632385 308.9623718262
0.006674294359984 308.9700012207
0.006708694621923 308.9776611328
0.006743094883862 308.9852905273
0.0067774951458 308.9928588867
0.0068118958734 309.0000610352
0.006846296135339 309.0072937012
0.006880696397277 309.0144348145
0.006915096659216 309.0215759277
0.006949497386816 309.0287475586
0.006983897648754 309.0358581543
0.007018297910693 309.0429992676
0.007052698172631 309.0498046875
0.00708709843457 309.0565490723
0.00712149916217 309.063293457
0.007155899424108 309.0701599121
0.007190299686047 309.0769348145
0.007224699947985 309.0833740234
0.007259100675585 309.0898742676
0.007293500937524 309.096282959
0.007327901199462 309.1024780273
0.007362301461401 309.1084899902
0.007396702189001 309.1145629883
0.007431102450939 309.1191711426
0.007465502712878 309.1231994629
0.007499902974817 309.1272277832
0.007534303236755 309.1313171387
0.007568703964355 309.1354675293
0.007603104226294 309.1396179199
0.007637504488232 309.1437683105
0.007671904750171 309.1478881836
0.007706305477771 309.1520080566
0.007740705739709 309.1561279297
0.007775106001648 309.1602478027
0.007809506263586 309.1643371582
0.007843906991186 309.1684570312
0.007878307253125 309.172454834
0.007912707515063 309.1764526367
0.007947107777002 309.1804504395
0.007981508504602 309.1844482422
0.00801590876654 309.1884155273
0.008050309028479 309.1924133301
0.008084709290417 309.1963806152
0.008119109552356 309.200378418
0.008153510279956 309.204284668
0.008187910541894 309.2081298828
0.008222310803833 309.2119750977
0.008256711065772 309.2158203125
0.008291111793371 309.2196655273
0.00832551205531 309.2235107422
0.008359912317249 309.227355957
0.008394312579187 309.2312011719
0.008428713306787 309.2350158691
0.008463113568726 309.2386474609
0.008497513830664 309.2422790527
0.008531914092603 309.2459411621
0.008566314354541 309.2495727539
0.008600715082141 309.2532653809
0.00863511534408 309.2569274902
0.008669515606018 309.2606201172
0.008703915867957 309.2643127441
0.008738316595557 309.2679748535
0.008772716857495 309.2714538574
0.008807117119434 309.2749023438
0.008841517381372 309.2783508301
0.008875918108972 309.2817993164
0.008910318370911 309.2852478027
0.008944718632849 309.2887268066
0.008979118894788 309.2922058105
0.009013519622388 309.295715332
0.009047919884327 309.2992248535
0.009082320146265 309.3027038574
0.009116720408204 309.3059692383
0.009151120670142 309.3092651367
0.009185521397742 309.3125305176
0.009219921659681 309.315826416
0.009254321921619 309.3190917969
0.009288722183558 309.3223571777
0.009323122911158 309.3256530762
0.009357523173096 309.328918457
0.009391923435035 309.3322143555
0.009426323696973 309.3355102539
0.009460724424573 309.3386535645
0.009495124686512 309.341796875
0.00952952494845 309.3449401855
0.009563925210389 309.3480529785
0.009598325937989 309.3511962891
0.009632726199927 309.3543395996
0.009667126461866 309.3574523926
0.009701526723805 309.3605957031
0.009735927218574 309.3636169434
0.009770327480512 309.3665466309
0.009804727975282 309.3694763184
0.00983912823722 309.3724060059
0.009873528731989 309.3753662109
0.009907928993928 309.378326416
0.009942329488697 309.3812561035
0.009976729750636 309.3842163086
0.010011130245405 309.3871459961
0.010045530507343 309.3901062012
0.010079931002113 309.3930358887
0.010114331264051 309.3959350586
0.01014873175882 309.398651123
0.010183132020759 309.4013977051
0.010217532282698 309.4041137695
0.010251932777467 309.4068603516
0.010286333039405 309.409576416
0.010320733534175 309.412322998
0.010355133796113 309.4150695801
0.010389534290882 309.4178161621
0.010423934552821 309.4205627441
0.01045833504759 309.4233093262
0.010492735309529 309.4260559082
0.010527135804298 309.4288024902
0.010561536066237 309.4315490723
0.010595936561006 309.4341125488
0.010630336822944 309.4366760254
0.010664737317714 309.439239502
0.010699137579652 309.4418029785
0.010733537841591 309.4443664551
0.01076793833636 309.4469299316
0.010802338598298 309.4494934082
0.010836739093068 309.4520568848
0.010871139355006 309.4546508789
0.010905539849775 309.457244873
0.010939940111714 309.4598388672
0.010974340606483 309.4622802734
0.011008740868422 309.4647216797
0.011043141363191 309.4671325684
0.01107754162513 309.4695129395
0.011111942119899 309.4719238281
0.011146342381837 309.4743347168
0.011180742876607 309.476776123
0.011215143138545 309.4792175293
0.011249543633314 309.4814758301
0.011283943895253 309.483795166
0.011318344157192 309.4860839844
0.011352744651961 309.4884033203
0.011387144913899 309.4907226562
0.011421545408669 309.4930114746
0.011455945670607 309.4953308105
0.011490346165376 309.4976196289
0.011524746427315 309.4999389648
0.011559146922084 309.5021972656
0.011593547184023 309.5044555664
0.011627947562377 309.5067138672
0.01166234794073 309.508972168
0.011696748319084 309.5111694336
0.011731148697438 309.5132751465
0.011765549075792 309.515411377
0.011799949454146 309.5175170898
0.0118343498325 309.5196533203
0.011868750210854 309.5217590332
0.011903150589208 309.5238647461
0.011937550967562 309.525970459
0.011971951345916 309.5280761719
0.012006351724269 309.5301818848
0.012040752102623 309.5322875977
0.012075152364562 309.5343933105
0.012109552742916 309.5364990234
0.01214395312127 309.5386352539
0.012178353499624 309.5407409668
0.012212753877977 309.5428466797
0.012247154256331 309.5448608398
0.012281554634685 309.546875
0.012315955013039 309.5488891602
0.012350355391393 309.5508728027
0.012384755769747 309.5528869629
0.012419156148101 309.5548400879
0.012453556526455 309.5567932129
0.012487956904809 309.5586547852
0.012522357283163 309.5605163574
0.012556757661516 309.5623779297
0.0125911579816626 309.564239502
0.0126255583600165 309.5661010742
0.0126599587383704 309.5679931641
0.0126943591167243 309.5698852539
0.0127287594950782 309.5717773438
0.0127631598152245 309.5736694336
0.0127975601935783 309.5755310059
0.0128319605719322 309.5774230957
0.0128663609502861 309.5793151855
0.01290076132864 309.5812072754
0.0129351617069939 309.5830993652
0.0129695620853478 309.5849304199
0.013003962405494 309.5867004395
0.0130383627838479 309.5884399414
0.0130727631622018 309.5901794434
0.0131071635405557 309.5919494629
0.0131415639189096 309.5936889648
0.0131759642972635 309.5954589844
0.0132103646465136 309.5971984863
0.0132447650248675 309.5989685059
0.0132791654032214 309.6007385254
0.0133135657815753 309.6025085449
0.0133479661453772 309.6042480469
0.0133823665091792 309.6058959961
0.0134167668875331 309.6075439453
0.013451167258611 309.6091918945
0.013485567629689 309.6108398438
0.0135199679971289 309.612487793
0.0135543683682069 309.6141357422
0.0135887687392848 309.6157836914
0.0136231691103627 309.617401123
0.0136575694778027 309.6190185547
0.0136919698561566 309.6206359863
0.0137263702199586 309.6222839355
0.0137607705983125 309.6239013672
0.0137951709621145 309.6255187988
0.0138295713404683 309.6270446777
0.0138639717188222 309.6286315918
0.0138983720680723 309.6301879883
0.0139327724464262 309.6317749023
0.0139671728247801 309.6333312988
0.014001573203134 309.6349182129
0.014035973552384 309.6364746094
0.0140703739598418 309.6379699707
0.014104774279988 309.6394958496
0.0141391746583419 309.6409912109
0.0141735750366958 309.6424865723
0.0142079754150497 309.6439819336
0.0142423757934036 309.6455078125
0.0142767761717575 309.6470031738
0.0143111765501114 309.6484680176
0.0143455768702576 309.6499023438
0.0143799772486115 309.6513671875
0.0144143776269654 309.6528015137
0.0144487780053193 309.6542663574
0.0144831783836732 309.6557006836
0.0145175787620271 309.6571350098
0.014551979140381 309.6585998535
0.014586379518735 309.6600341797
0.014620779897089 309.6614685059
0.014655180275443 309.6628723145
0.014689580653797 309.6642150879
0.01472398103215 309.6655883789
0.014758381294089 309.6669311523
0.014792781672443 309.6683044434
0.014827182050797 309.6696777344
0.014861582429151 309.6710205078
0.014895982807505 309.6723937988
0.014930383185858 309.6737670898
0.014964783564212 309.6751403809
0.014999183942566 309.6765136719
0.01503358432092 309.6778564453
0.015067984699274 309.6791381836
0.015102385077628 309.6804504395
0.015136785455982 309.6817321777
0.015171185834336 309.683013916
0.01520558621269 309.6842956543
0.015239986591044 309.6855773926
0.015274386969397 309.6868896484
0.015308787231336 309.6881713867
0.01534318760969 309.689453125
0.015377587988044 309.6907043457
0.015411988366398 309.6919555664
0.015446388744752 309.6931762695
0.015480789123105 309.6944274902
0.015515189501459 309.6956481934
0.015549589879813 309.6968994141
0.015583990141752 309.6981201172
0.015618390636521 309.6993408203
0.01565279089846 309.7005615234
0.015687191393229 309.7018127441
0.015721591655167 309.7030029297
0.015755992149937 309.7042236328
0.015790392411875 309.7054138184
0.015824792906644 309.7066040039
0.015859193168583 309.7077636719
0.015893593663352 309.7089538574
0.015927993925291 309.7101135254
0.01596239442006 309.7112731934
0.015996794681999 309.7124328613
0.016031195176768 309.7135925293
0.016065595438706 309.7147521973
0.016099995700645 309.7158813477
0.016134396195414 309.717010498
0.016168796457353 309.7181396484
0.016203196952122 309.7192382812
0.01623759721406 309.7203674316
0.01627199770883 309.721496582
0.016306397970768 309.7225952148
0.016340798465537 309.7237243652
0.016375198727476 309.724822998
0.016409599222245 309.7259521484
0.016443999484184 309.7270507812
0.016478399978953 309.7281188965
0.016512800240892 309.7291870117
0.016547200735661 309.730255127
0.016581600997599 309.7313232422
0.016616001492369 309.7323608398
0.016650401754307 309.7334289551
0.016684802016246 309.7344970703
0.016719202511015 309.7355651855
0.016753602772954 309.7366333008
0.016788003267723 309.737701416
0.016822403529661 309.7387390137
0.016856804024431 309.7397766113
0.016891204286369 309.740814209
0.016925604781138 309.7418212891
0.016960005043077 309.7428283691
0.016994405537846 309.7438354492
0.017028805799785 309.7448120117
0.017063206294554 309.7458190918
0.017097606556492 309.7468261719
0.017132007051262 309.7478027344
0.0171664073132 309.7488098145
0.017200807575139 309.7498168945
0.017235208069908 309.7508239746
0.017269608331847 309.7518005371
0.017304008826616 309.752746582
0.017338409088554 309.753692627
0.017372809583324 309.7546386719
0.017407209845262 309.7555847168
0.017441610340031 309.7565307617
0.01747601060197 309.7574768066
0.017510410863909 309.7584228516
0.017544811591508 309.7593688965
0.017579211853447 309.7603149414
0.017613612115386 309.7612609863
0.017648012377324 309.7622070312
0.017682413104924 309.7631225586
0.017716813366863 309.7640380859
0.017751213628801 309.7649536133
0.01778561389074 309.765838623
0.017820014152678 309.7667541504
0.017854414880278 309.7676391602
0.017888815142217 309.7685546875
0.017923215404155 309.7694702148
0.017957615666094 309.7703552246
0.017992016393694 309.771270752
0.018026416655632 309.7721557617
0.018060816917571 309.7730407715
0.018095217179509 309.7738952637
0.018129617907109 309.7747497559
0.018164018169048 309.775604248
0.018198418430986 309.7764587402
0.018232818692925 309.7773132324
0.018267218954864 309.7781677246
0.018301619682463 309.7790222168
0.018336019944402 309.779876709
0.018370420206341 309.7807312012
0.018404820468279 309.7815856934
0.018439221195879 309.7824401855
0.018473621457818 309.7832946777
0.018508021719756 309.7841491699
0.018542421981695 309.7850036621
0.018576822709295 309.7857971191
0.018611222971233 309.7865905762
0.018645623233172 309.7873840332
0.01868002349511 309.7881774902
0.01871442422271 309.7890014648
0.018748824484649 309.7897949219
0.018783224746587 309.7905883789
0.018817625008526 309.7913818359
0.018852025270464 309.7922058105
0.018886425998064 309.7929992676
0.018920826260003 309.7937927246
0.018955226521941 309.7945861816
0.01898962678388 309.7954101562
0.01902402751148 309.7962036133
0.019058427773418 309.7969970703
0.019092828035357 309.7977905273
0.019127228297296 309.7985534668
0.019161629024895 309.7993164062
0.019196029286834 309.8000793457
0.019230429548773 309.8008422852
0.019264829810711 309.8016052246
0.019299230538311 309.8023681641
0.01933363080025 309.8031311035
0.019368031062188 309.803894043
0.019402431324127 309.8046569824
0.019436831586065 309.8054199219
0.019471232313665 309.8061523438
0.019505632575604 309.8069152832
0.019540032837542 309.8076171875
0.019574433099481 309.8083496094
0.019608833827081 309.8090820312
0.019643234089019 309.8097839355
0.019677634350958 309.8105163574
0.019712034612896 309.8112182617
0.019746435340496 309.8119506836
0.019780835602435 309.8126525879
0.019815235864373 309.8133544922
0.019849636126312 309.8140869141
0.019884036388251 309.8147888184
0.01991843711585 309.8154907227
0.019952837377789 309.8161621094
0.019987237639728 309.8168640137
0.020021637901666 309.8175354004
0.020056038629266 309.8182067871
0.020090438891205 309.8189086914
0.020124839153143 309.8195800781
0.020159239415082 309.8202819824
0.020193640142682 309.8209533691
0.02022804040462 309.8216552734
0.020262440666559 309.8223571777
0.020296840928497 309.8230285645
0.020331241656097 309.8236999512
0.020365641918036 309.8243713379
0.020400042179974 309.825012207
0.020434442441913 309.8256835938
0.020468842703852 309.8263549805
0.020503243431451 309.8269958496
0.02053764369339 309.8276367188
0.020572043955329 309.8282775879
0.020606444217267 309.828918457
0.020640844944867 309.8295593262
0.020675245206806 309.8302001953
0.020709645468744 309.8308410645
0.020744045730683 309.8315124512
0.020778446458283 309.8321533203
0.020812846720221 309.8327941895
0.02084724698216 309.833404541
0.020881647244098 309.8340454102
0.020916047971698 309.8346557617
0.020950448233637 309.8352661133
0.020984848495575 309.8359069824
0.021019248757514 309.836517334
0.021053649019452 309.8371276855
0.021088049747052 309.8377685547
0.021122450008991 309.8383789062
0.021156850270929 309.8389892578
0.021191250532868 309.8395996094
0.021225651260468 309.8401794434
0.021260051522406 309.8407897949
0.021294451784345 309.8413696289
0.021328852046284 309.8419799805
0.021363252773883 309.8425598145
0.021397653035822 309.8431396484
0.021432053297761 309.84375
0.021466453559699 309.844329834
0.021500853821638 309.844909668
0.021535254083576 309.8455200195
0.021569654345515 309.8460998535
0.021604055538776 309.8467102051
0.021638455800715 309.8472900391
0.021672856062653 309.8478393555
0.021707256324592 309.8484191895
0.02174165658653 309.8489685059
0.021776056848469 309.8495178223
0.021810457110407 309.8500976562
0.021844857372346 309.8506469727
0.021879258565607 309.8512268066
0.021913658827546 309.851776123
0.021948059089484 309.852355957
0.021982459351423 309.8529052734
0.022016859613361 309.8534851074
0.0220512598753 309.8540344238
0.022085660137239 309.8545837402
0.022120060399177 309.8551330566
0.022154460661116 309.855682373
0.022188861854377 309.8562316895
0.022223262116315 309.8567504883
0.022257662378254 309.8572998047
0.022292062640193 309.8578186035
0.022326462902131 309.8583374023
0.02236086316407 309.8588867188
0.022395263426008 309.8594055176
0.022429663687947 309.859954834
0.022464064881208 309.8604736328
0.022498465143147 309.8609924316
0.022532865405085 309.861541748
0.022567265667024 309.8620605469
0.022601665928962 309.8625793457
0.022636066190901 309.8630981445
0.022670466452839 309.8636474609
0.022704866714778 309.8641357422
0.022739266976717 309.8646240234
0.022773668169978 309.8651428223
0.022808068431916 309.8656311035
0.022842468693855 309.8661499023
0.022876868955793 309.8666381836
0.022911269217732 309.8671569824
0.022945669479671 309.8676452637
0.022980069741609 309.8681640625
0.023014470003548 309.8686523438
0.023048871196809 309.869140625
0.023083271458747 309.8696594238
0.023117671720686 309.8701477051
0.023152071982625 309.8706359863
0.023186472244563 309.8711242676
0.023220872506502 309.8716125488
0.02325527276844 309.8721008301
0.023289673030379 309.8725585938
0.023324073292317 309.873046875
0.023358474485579 309.8735351562
0.023392874747517 309.8740234375
0.023427275009456 309.8744812012
0.023461675271394 309.8749694824
0.023496075533333 309.8754577637
0.023530475795271 309.8759155273
0.02356487605721 309.876373291
0.02359927631915 309.8768310547
0.02363367658109 309.8772888184
0.02366807777435 309.877746582
0.02370247803629 309.8782043457
0.02373687829823 309.8786621094
0.02377127856016 309.879119873
0.0238056788221 309.8795471191
0.02384007908404 309.8800048828
0.02387447934598 309.8804626465
0.02390887960792 309.8808898926
0.02394328080118 309.8813476562
0.02397768106312 309.8817749023
0.02401208132506 309.882232666
0.024046481587 309.8826904297
0.02408088184893 309.8831176758
0.02411528211087 309.8835754395
0.02414968237281 309.8840026855
0.02418408263475 309.8844909668
0.02421848289669 309.884979248
0.02425288408995 309.8854675293
0.02428728435189 309.8859558105
0.02432168461383 309.8864440918
0.02435608487576 309.8869628906
0.0243904851377 309.8874511719
0.02442488539964 309.8879394531
0.02445928566158 309.8884277344
0.02449368592352 309.8889160156
0.02452808711678 309.8894042969
0.02456248737872 309.8899230957
0.02459688764066 309.8903808594
0.0246312879026 309.890838623
0.02466568816453 309.8912963867
0.02470008842647 309.8917541504
0.02473448868841 309.8922424316
0.02476888895035 309.8927001953
0.02480328921229 309.893157959
0.02483769040555 309.8936157227
0.02487209066749 309.8941040039
0.02490649092943 309.8945617676
0.02494089119137 309.8950195312
0.0249752914533 309.8954772949
0.02500969171524 309.8959350586
0.02504409197718 309.8964233398
0.02507849223912 309.8967895508
0.02511289343238 309.8970031738
0.02514729369432 309.8972473145
0.02518169395626 309.8974914551
0.0252160942182 309.8977050781
0.02525049448014 309.8979492188
0.02528489474207 309.8981628418
0.02531929500401 309.8984069824
0.02535369526595 309.8986206055
0.02538809552789 309.8988647461
0.02542249672115 309.8991088867
0.02545689698309 309.8993225098
0.02549129724503 309.8995361328
0.02552569750697 309.8997497559
0.02556009776891 309.8999938965
0.02559449803084 309.9002075195
0.02562889829278 309.9004211426
0.02566329855472 309.9006652832
0.02569769974798 309.9008789062
0.02573210000992 309.9010925293
0.02576650027186 309.9013366699
0.0258009005338 309.901550293
0.02583530079574 309.901763916
0.02586970105767 309.9019775391
0.02590410131961 309.9022216797
0.02593850158155 309.9024353027
0.02597290184349 309.9026489258
0.02600730303675 309.9028625488
0.02604170329869 309.9030761719
};
\addlegendentry{Model $\Em_\text{NL}^\N(\bar{\mu})\equiv\Em_\text{L}^\N((\bar\mu))$}
\end{axis}

\end{tikzpicture}

%% file: tex/4-uq.tex
\section{Uncertainty quantification}
\label{sec:uq}

The \emph{uncertainty quantification} (UQ) allows quantifying the uncertainty of the model parameters on various outputs.
In the present work, we focus on forward UQ, that is, we want to quantify the uncertainty and the sensitivity of the output of the model, given the uncertainty of the input parameters.
More precisely, two studies are performed:
(i) an uncertainty propagation, to understand how the uncertainties of the inputs of the model are propagated to the output via the computational model, and
(ii) a \emph{sensitivity analysis} (SA) to assess the impact of varying selected parameters on several outputs of interest, namely temperature at specific locations in the eye.
Their locations are detailed in \Cref{fig:outputs}.

The SA is conducted in two different approaches.
First, to recapitulate findings from the literature, we performed a deterministic SA, where for each simulation,
only one parameter is allowed to vary in a given range, whereas the others are fixed to their baseline value.
In a second stage, we extended the SA to a stochastic framework, where each selected parameter follows a given random distribution
and the impact on the quantity of interest is assessed via sensitivity indices.
The advantage of the latter is the global perspective provided by this method and its ability to capture high-order interactions among several input parameters.


\def\ech{0.62}
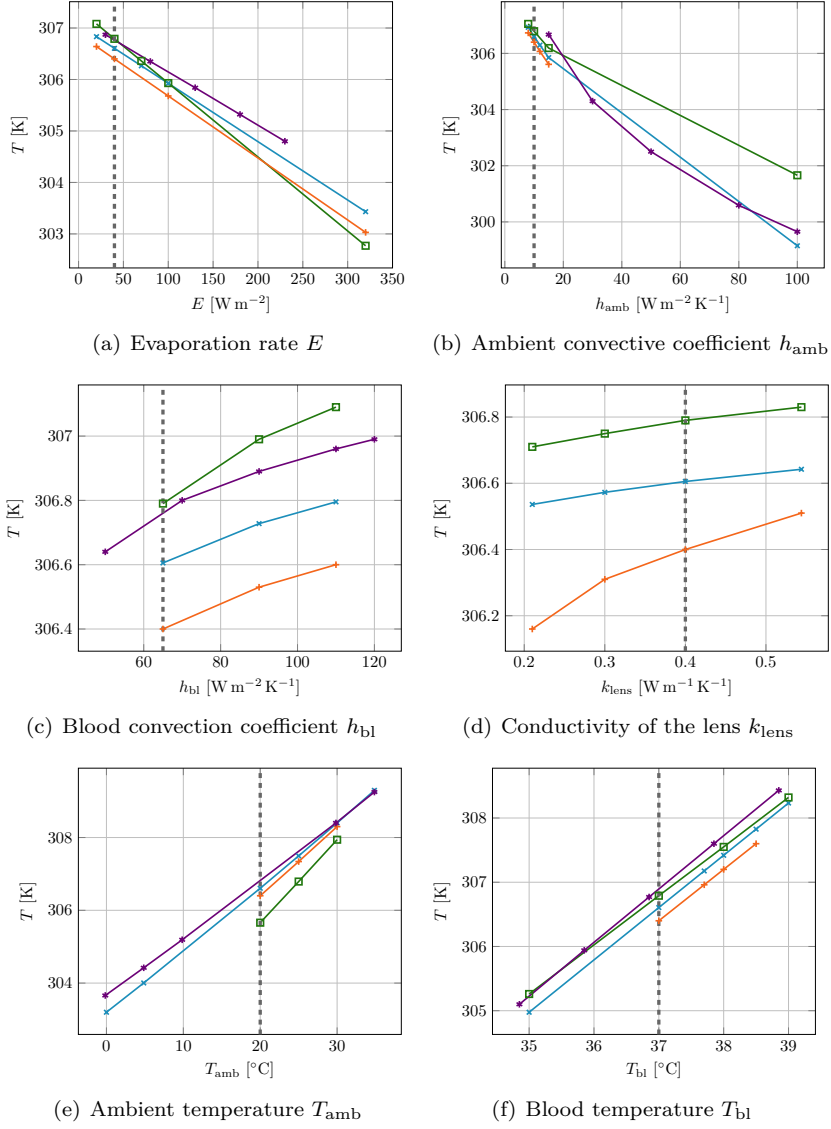
\begin{figure}
    \centering
    \subfigure[Evaporation rate $E$]{
    \begin{tikzpicture}[scale=\ech]
    \begin{axis}
        [grid=major, xlabel = {$E$ [\unit{\watt.\meter^{-2}}]}, ylabel = {$T$ [\unit{\kelvin}]},
        yticklabel style={scaled ticks=false, /pgf/number format/fixed, /pgf/number format/precision=0},
        xticklabel style={scaled ticks=false, /pgf/number format/fixed, /pgf/number format/precision=0}]
        \addplot[colorFeel3, mark=x, line width=1pt] table[x=E, y=OK, col sep=comma] {fig/eye/results/deterministic-SA/E_feel_K.csv};
        \addplot[colorOoi, mark=square, line width=1pt] table[x=E, y=OK, col sep=comma] {fig/eye/results/deterministic-SA/E_ooi_K.csv};
        \addplot[colorScott, mark=+, line width=1pt] table[x=E, y=OK, col sep=comma] {fig/eye/results/deterministic-SA/E_scott_K.csv};
        \addplot[colorLi, mark=asterisk, line width=1pt] table[x=E, y=OK, col sep=comma] {fig/eye/results/deterministic-SA/E_li_K.csv};
    \draw[dashed, color=black!60, line width=2pt] (40, \pgfkeysvalueof{/pgfplots/ymin}) -- (40, \pgfkeysvalueof{/pgfplots/ymax});
    \end{axis}
    \end{tikzpicture}
    }
    \subfigure[Ambient convective coefficient $h_\text{amb}$]{
    \begin{tikzpicture}[scale=\ech]
    \begin{axis}
        [grid=major, xlabel = {$h_\text{amb}$ [\unit{\watt.\meter^{-2}.\kelvin^{-1}}]}, ylabel = {$T$ [\unit{\kelvin}]},
        yticklabel style={scaled ticks=false, /pgf/number format/fixed, /pgf/number format/precision=0},
        xticklabel style={scaled ticks=false, /pgf/number format/fixed, /pgf/number format/precision=0}]
        \addplot[colorFeel3, mark=x, line width=1pt] table[x=h_amb, y=OK, col sep=comma] {fig/eye/results/deterministic-SA/h_amb_feel_K.csv};
        \addplot[colorOoi, mark=square, line width=1pt] table[x=h_amb, y=OK, col sep=comma] {fig/eye/results/deterministic-SA/h_amb_ooi_K.csv};
        \addplot[colorScott, mark=+, line width=1pt] table[x=h_amb, y=OK, col sep=comma] {fig/eye/results/deterministic-SA/h_amb_scott_K.csv};
        \addplot[colorLi, mark=asterisk, line width=1pt] table[x=h_amb, y=OK, col sep=comma] {fig/eye/results/deterministic-SA/h_amb_li_K.csv};
    \draw[dashed, color=black!60, line width=2pt] (10, \pgfkeysvalueof{/pgfplots/ymin}) -- (10, \pgfkeysvalueof{/pgfplots/ymax});
    \end{axis}
    \end{tikzpicture}
    }
    \subfigure[Blood convection coefficient $h_\text{bl}$]{
    \begin{tikzpicture}[scale=\ech]
    \begin{axis}
        [grid=major, , xlabel = {$h_\text{bl}$ [\unit{\watt.\meter^{-2}.\kelvin^{-1}}]}, ylabel = {$T$ [\unit{\kelvin}]},
        yticklabel style={scaled ticks=false, /pgf/number format/fixed, /pgf/number format/precision=1},
        xticklabel style={scaled ticks=false, /pgf/number format/fixed, /pgf/number format/precision=0}]
        \addplot[colorFeel3, mark=x, line width=1pt] table[x=h_bl, y=OK, col sep=comma] {fig/eye/results/deterministic-SA/h_bl_feel_K.csv};
        \addplot[colorOoi, mark=square, line width=1pt] table[x=h_bl, y=OK, col sep=comma] {fig/eye/results/deterministic-SA/h_bl_ooi_K.csv};
        \addplot[colorScott, mark=+, line width=1pt] table[x=h_bl, y=OK, col sep=comma] {fig/eye/results/deterministic-SA/h_bl_scott_K.csv};
        \addplot[colorLi, mark=asterisk, line width=1pt] table[x=h_bl, y=OK, col sep=comma] {fig/eye/results/deterministic-SA/h_bl_li_K.csv};
    \draw[dashed, color=black!60, line width=2pt] (65, \pgfkeysvalueof{/pgfplots/ymin}) -- (65, \pgfkeysvalueof{/pgfplots/ymax});
    \end{axis}
    \end{tikzpicture}
    }
    \subfigure[Conductivity of the lens $k_\text{lens}$]{
    \begin{tikzpicture}[scale=\ech]
    \begin{axis}
        [grid=major, , xlabel = {$k_\text{lens}$ [\unit{\watt.\meter^{-1}.\kelvin^{-1}}]}, ylabel = {$T$ [\unit{\kelvin}]},
        yticklabel style={scaled ticks=false, /pgf/number format/fixed, /pgf/number format/precision=1},
        xticklabel style={scaled ticks=false, /pgf/number format/fixed, /pgf/number format/precision=1}]
        \addplot[colorFeel3, mark=x, line width=1pt] table[x=k_lens, y=OK, col sep=comma] {fig/eye/results/deterministic-SA/k_lens_feel_K.csv};
        \addplot[colorOoi, mark=square, line width=1pt] table[x=k_lens, y=OK, col sep=comma] {fig/eye/results/deterministic-SA/k_lens_ooi_K.csv};
        \addplot[colorScott, mark=+, line width=1pt] table[x=k_lens, y=OK, col sep=comma] {fig/eye/results/deterministic-SA/k_lens_scott_K.csv};
    \draw[dashed, color=black!60, line width=2pt] (0.4, \pgfkeysvalueof{/pgfplots/ymin}) -- (0.4, \pgfkeysvalueof{/pgfplots/ymax});
    \end{axis}
    \end{tikzpicture}
    }
    \subfigure[Ambient temperature $T_\text{amb}$]{
    \begin{tikzpicture}[scale=\ech]
    \begin{axis}
        [grid=major, , xlabel = {$T_\text{amb}$ [\unit{\degreeCelsius}]}, ylabel = {$T$ [\unit{\kelvin}]},
        yticklabel style={scaled ticks=false, /pgf/number format/fixed, /pgf/number format/precision=0},
        xticklabel style={scaled ticks=false, /pgf/number format/fixed, /pgf/number format/precision=0}]
        \addplot[colorFeel3, mark=x, line width=1pt] table[x=T_amb, y=OK, col sep=comma] {fig/eye/results/deterministic-SA/T_amb_feel_K.csv};
        \addplot[colorOoi, mark=square, line width=1pt] table[x=T_amb, y=OK, col sep=comma] {fig/eye/results/deterministic-SA/T_amb_ooi_K.csv};
        \addplot[colorScott, mark=+, line width=1pt] table[x=T_amb, y=OK, col sep=comma] {fig/eye/results/deterministic-SA/T_amb_scott_K.csv};
        \addplot[colorLi, mark=asterisk, line width=1pt] table[x=T_amb, y=OK, col sep=comma] {fig/eye/results/deterministic-SA/T_amb_li_K.csv};
    \draw[dashed, color=black!60, line width=2pt] (20, \pgfkeysvalueof{/pgfplots/ymin}) -- (20, \pgfkeysvalueof{/pgfplots/ymax});
    \end{axis}
    \end{tikzpicture}
    }
    \subfigure[Blood temperature $T_\text{bl}$]{
    \begin{tikzpicture}[scale=\ech]
    \begin{axis}
        [grid=major, xlabel = {$T_\text{bl}$ [\unit{\degreeCelsius}]}, ylabel = {$T$ [\unit{\kelvin}]},
        yticklabel style={scaled ticks=false, /pgf/number format/fixed, /pgf/number format/precision=0},
        xticklabel style={scaled ticks=false, /pgf/number format/fixed, /pgf/number format/precision=0}]
        \addplot[colorFeel3, mark=x, line width=1pt] table[x=T_bl, y=OK, col sep=comma] {fig/eye/results/deterministic-SA/T_bl_feel_K.csv};
        \addplot[colorOoi, mark=square, line width=1pt] table[x=T_bl, y=OK, col sep=comma] {fig/eye/results/deterministic-SA/T_bl_ooi_K.csv};
        \addplot[colorScott, mark=+, line width=1pt] table[x=T_bl, y=OK, col sep=comma] {fig/eye/results/deterministic-SA/T_bl_scott_K.csv};
        \addplot[colorLi, mark=asterisk, line width=1pt] table[x=T_bl, y=OK, col sep=comma] {fig/eye/results/deterministic-SA/T_bl_li_K.csv};
    \draw[thick, dashed, color=black!60, line width=2pt]  (37, \pgfkeysvalueof{/pgfplots/ymin}) -- (37, \pgfkeysvalueof{/pgfplots/ymax});
    \end{axis}
    \end{tikzpicture}
    }

    \caption{Results of the DSA for the 6 parameters studied, among previous studies from the literature
        (markers \textcolor{colorFeel3}{\textsf{x} $\Em_\text{NL}(\mu)$}, \textcolor{colorOoi}{$\square$ \cite{NG2006268}}, \textcolor{colorScott}{\textsf{+} \cite{Scott_1988}}, \textcolor{colorLi}{$\ast$ \cite{li2010}}).
        The vertical dashed line corresponds to the baseline value of the parameter.}
    \label{fig:det-sa}
\end{figure}


\subsection{Deterministic sensitivity analysis}
\label{sec:DSA}

Our initial investigation of the impact of varying selected parameters is conducted through a \emph{deterministic sensitivity analysis} (DSA).
Specifically, we choose a parameter among the ones defined in \Cref{tab:parameters}, and we set the other to their baseline value.
Next, we vary the selected parameter among pre-defined values and compute the outputs of the high fidelity model.
Similar studies were performed in \cite{Scott_1988, NG2006268, li2010}.
We gather in the present study information about several parameters of interest for the heat transfer model from these studies,
namely baseline values and ranges.
These variations correspond not only to physiological conditions but also include some extreme situations.
We postpone a more in-depth discussion on this topic to \Cref{sec:SSA}, where the random distributions characterizing these parameters are set up.
Note that in this case, we do not need to use the reduced model, since only a relatively small number of simulations is required.

In \Cref{fig:det-sa}, we show the results of the DSA for the parameter $\mu=\{E, h_\text{amb}, h_\text{bl}, k_\text{lens}, T_\text{amb}, T_\text{bl}\}$, on point $O$ which is at the surface of the cornea.
The plain-line curves correspond to the results reported in the literature, which we compare with the results of our simulations;
the vertical dashed line corresponds to the baseline value for each parameter.
The results are in very good agreement with previous findings and show that temperature at the level of the cornea is strongly influenced by  $h_\text{amb}$, $T_\text{amb}$, $E$, and $T_\text{bl}$,
whereas the influence of $h_\text{bl}$ and $k_\text{lens}$ is less significant.
For instance, high air conductivity can result in a temperature \qty{7}{\kelvin} lower than the baseline value,
while the difference obtained for $h_\text{bl}$ and $k_\text{lens}$ in the computed temperature is at most of \qty{1}{\kelvin}.

\subsection[Stochastic sensitivity analysis]{Stochastic sensitivity analysis (SSA)}
\label{sec:SSA}

We consider an output quantity $Y$ depending on a set of input parameters $\mu\in\Dmu$,
and we estimate the sensitivity of $Y$ to each parameter $\mu_i$ for $i\in\llbracket 1,d\rrbracket$, where $d$ is the dimension of the parametric space.
To this end, we compute the \emph{Sobol' sensitivity indices} introduced in \cite{Sobol1993SensitivityEF} as follows.
We assume that each component $\mu_i$ of $\mu$ follows a random variable $X_i$, independent of the others.
The first-order indices are defined as:

\begin{equation}
    S_i := \frac{\var\left(\E\left[Y\middle|X_i\right]\right)}{\var(Y)}
\end{equation}
where $\var(Y)$ corresponds to the variance of $Y$ including the eventual non-linearity effect of the coefficient on the output,
and $\var\left(\E\left[Y\middle|X_i\right]\right)$ is the variance of the conditional expectation of $Y$ given $X_i$, corresponding to the first order effect of the parameter $\mu_i$ on the output:
if the parameter modeled by the distribution $X_j$ has a great impact on the output $Y$, then $\E\left[Y\middle|X_j\right]$ will vary has well, and so its variance.

We also define the \emph{total Sobol' index}:

\begin{equation}
    S_i^\text{tot} := \frac{\var\left(\E\left[Y\middle|X_{(-i)}\right]\right)}{\var(Y)} = 1 - S_{-i}
\end{equation}
where $X_{(-i)} = (X_1,\cdots,X_{i-1},X_{i+1},\cdots,X_d)$ is the set of parameters without the parameter $X_i$, and $S_{-i}$ is the sum of the indices where $X_i$ is not present.


To compute the Sobol' indices, we use an algorithm of functional chaos, implemented in the library OpenTURNS \cite{Baudin2016} by the class \texttt{FunctionalChaosAlgorithm}, using a bootstrap method%
\footnote{See documentation \url{https://openturns.github.io/openturns/latest/auto_meta_modeling/polynomial_chaos_metamodel/plot_chaos_sobol_confidence.html}}
for the confidence intervals.

\subsubsection{Choice of the distributions}

We discuss now the prior distributions for each parameter.
Each parameter does not depend on the others, resulting in a family of 6 random independent variables.
\Cref{fig:eye:distributions} shows the probability density function (PDF) of distributions of the parameters,
associated with the baseline values (see \Cref{tab:parameters}),
where the parameters used in the literature for the deterministic sensitivity analysis are represented with a vertical line.
We present hereafter some details on how the random distributions were constructed.

\begin{itemize}
    \item \textbf{Evaporation rate $E$:} According to \cite{Scott_1988}, the evaporation rate's range is of 40 to 100 \unit{\watt.\meter^{-2}}, using data from literature \cite{adler53}.
    The value $E=\qty{40}{\watt.\meter^{-2}}$ is chosen as the baseline value.
    Some high values are also considered, to study the impact of important evaporation rates.
    The values used in the literature run from 20 to \qty{320}{\watt.\meter^{-2}}.
    As this parameter varies by several orders of magnitude, we decided to use a log-normal distribution to represent it.
    More precisely we set $E\sim \logN(\mu_E, \sigma_E, \gamma_E)$, with $\sigma_E = 0.7$, $\mu_E=\log(40)-\frac{0.15^2}{2}$ and $\gamma_E=20$, restricted to $[20, 130]$.
    The distribution is presented in \Cref{fig:eye:distributions:E}.
    This choice of the distribution leads to a mean value of $\overline{E}=\qty{55.8}{\watt.\meter^{-2}}$.

    \item \textbf{Ambient air convection coefficient $h_\text{\normalfont amb}$:}
    In \cite{Scott_1988}, the sole value given for the ambient air convection coefficient is \qty{10}{\watt.\meter^{-2}.\kelvin^{-1}}, and similar values are used to run the DSA, from 8 to 15~\unit{\watt.\meter^{-2}.\kelvin^{-1}}.
    Other results in the literature coroborate this value: \cite[Table 12.2]{KOSKY2013259} reports a range of 2.5 to 25~\unit{\watt.\meter^{-2}.\kelvin^{-1}} for a free convection, and 10 to \qty{500}{\watt.\meter^{-2}.\kelvin^{-1}} for a forced convection.
    \cite{engeneer-edge} proposes a range of 10 to \qty{100}{\watt.\meter^{-2}.\kelvin^{-1}} for the air.
    In their DSA, \cite{NG2006268} and \cite{li2010} use higher values of $h_\text{amb}$, up to \qty{100}{\watt.\meter^{-2}.\kelvin^{-1}} to simulate a forced convection condition.
    As high values are not a common case, such a coefficient should not have a high frequency in the distribution.
    We chose to use a log-normal distribution: $h_\text{amb}\sim \logN\left(\log(10)-\frac{1}{2}, 1, 8\right)$.
    In Remark \ref{rem:linearization}, we discussed the linearization process of the model, inducing the usage of a fixed parameter $h_\text{r}$ chosen to fit temperature in usual ambient room conditions,
    which leads to a restriction of the distribution to the interval $[8, 100]~\unit{\watt.\meter^{-2}.\kelvin^{-1}}$.
    The distribution is presented in \Cref{fig:eye:distributions:hamb}.
    The mean value of the distribution is $\overline{h_\text{amb}}=\qty{17.6}{\watt.\meter^{-2}.\kelvin^{-1}}$.

    \item \textbf{Blood convection coefficient $h_\text{\normalfont bl}$:} A control value of \qty{65}{\watt.\meter^{-2}.\kelvin^{-1}}, derived from experimental data \cite{Lagendijk_1982} is provided in \cite{Scott_1988}.
    For the DSA, the values used run from 50 to \qty{120}{\watt.\meter^{-2}.\kelvin^{-1}}.
    This leads us to the following assumption for the distribution of the parameter: $h_\text{bl}\sim \logN\left(\log(65)-\frac{0.15^2}{2}, 0.15, 0\right)$, restricted over $[50, 120]$, see \Cref{fig:eye:distributions:hbl}.
    The mean of this distribution is $\overline{h_\text{bl}}=\qty{65.8}{\watt.\meter^{-2}.\kelvin^{-1}}$.

    \item \textbf{Lens conductivity $k_\text{\normalfont lens}$:} This parameter is chosen among all the conductivities since the water content of the lens varies with aging \cite{Scott_1988}.
    \cite{Scott_1988} and \cite{NG2006268} run the DSA with this parameter, using values from 0.21 to \qty{0.544}{\watt\metre^{-1}\kelvin^{-1}}.
    As the range of values is not very large, it seems reasonable to use a uniform distribution for this parameter: $k_\text{lens}\sim \unif(0.21, 0.544)$, see \Cref{fig:eye:distributions:klens}.

    \item \textbf{Ambient temperature $T_\text{\normalfont amb}$:} The baseline value of this parameter is taken to a usual room temperature of \qty{294}{\kelvin} (\qty{20}{\celsius}).
    The values used for the DSA vary from extreme conditions of \qty{273}{\kelvin} (\qty{0}{\celsius}) to \qty{308}{\kelvin} (\qty{35}{\celsius}).
    As these extreme values are not very common, we choose to restrict the values taken by $T_\text{amb}$ from \qty{283.15}{\kelvin} (\qty{-35}{\celsius}) to \qty{303.15}{\kelvin} (\qty{30}{\celsius}): $T_\text{amb}\sim \unif(283.15, 303.15)$.
    The distribution is presented in \Cref{fig:eye:distributions:Tamb}.

    \item \textbf{Blood temperature $T_\text{\normalfont bl}$:} The temperature of human blood is commonly accepted to be \qty{310}{\kelvin} (\qty{37}{\celsius}).
    For the DSA, cases of hypothermia and hyperthermia are considered, with a range from \qtyrange{308}{312.15}{\kelvin} (\qtyrange{35}{39}{\celsius}).
    We therefore take $T_\text{bl}\sim \unif(308, 312.15)$, see \Cref{fig:eye:distributions:Tbl}.
\end{itemize}

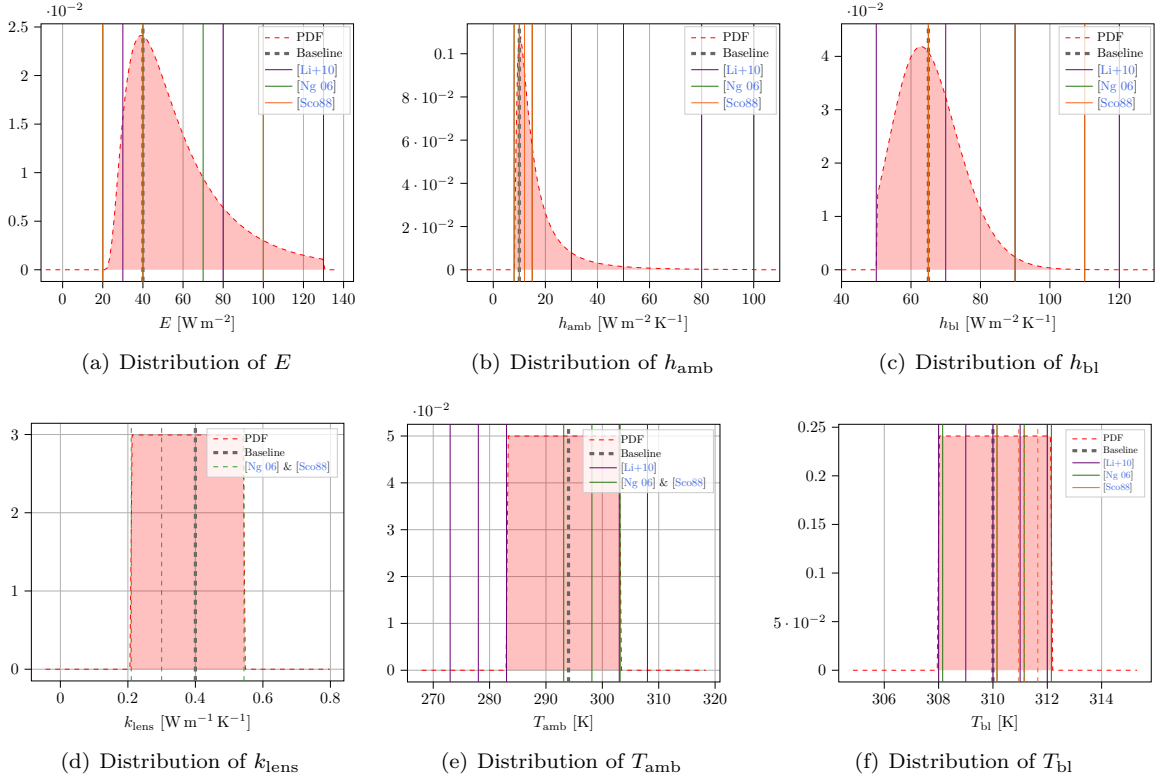
\begin{figure}
    \centering
    \def\chl{0.6}
    \subfigure[Distribution of $E$\label{fig:eye:distributions:E}]{
        \input{fig/eye/distributions/E.tikz}
    }
    \subfigure[Distribution of $h_\text{amb}$\label{fig:eye:distributions:hamb}]{
        \input{fig/eye/distributions/h_amb.tikz}
    }
    \subfigure[Distribution of $h_\text{bl}$\label{fig:eye:distributions:hbl}]{
        \input{fig/eye/distributions/h_bl.tikz}
    }
    \subfigure[Distribution of $k_\text{lens}$\label{fig:eye:distributions:klens}]{
        \input{fig/eye/distributions/k_lens.tikz}
    }
    \subfigure[Distribution of $T_\text{amb}$\label{fig:eye:distributions:Tamb}]{
        \input{fig/eye/distributions/T_amb.tikz}
    }
    \subfigure[Distribution of $T_\text{bl}$\label{fig:eye:distributions:Tbl}]{
        \input{fig/eye/distributions/T_bl.tikz}
    }
    \caption{Distributions of the parameters. The vertical lines represent the values chosen in literature for the DSA.}
    \label{fig:eye:distributions}
\end{figure}

\subsubsection{Uncertainty propagation}

We focus on the distribution of outputs of interest, from a random sample of the input parameters of get from the distributions presented $\pgfmathprintnumber{10000}$ points, which leads to a number of $\pgfmathprintnumber{10000}$ simulations.
The computational cost of the high fidelity simulations becomes in this case prohibitive and therefore we employ the reduced basis metamodel developed in \Cref{sec:rbm}.
\Cref{fig:uncertainty-propagation} presents the distribution of three outputs, namely the mean of the temperature over the cornea $T_\text{cornea}$,
and the temperature on points $O$ and $G$ respectively are the front and the back of the eyeball.
Note that $T_O$ and $T_\text{cornea}$ display a Gaussian distribution, whereas $T_G$ is more difficult to interpret, but could correspond to a uniform or bi-modal distribution.

We provide in \Cref{tab:results:uncertainty-propagation} results about mean values and standard deviation for the same quantities.
We note that the mean values of $T_O$ and $T_\text{cornea}$ are of the same order of magnitude as the experimental data in the validation section (\Cref{sec:validation}): the difference of temperature is about \qty{2}{\kelvin}, and standard deviations are in the same ranges.
The mean value of $T_G$ is very close to results reported in \Cref{fig:res:line} from the literature with a small standard deviation.

\begin{table}
    \centering
    \begin{tabular}{cccc}
        \toprule
        & $T_\text{cornea}$ & $T_O$ & $T_G$ \\
        \midrule
        Mean & 305.590082 & 303.185187 & 310.028526 \\
        Standard deviation & 1.788358 & 2.457063 & 1.055978 \\
        \bottomrule
    \end{tabular}
    \caption{Statistics of the outputs.}
    \label{tab:results:uncertainty-propagation}
\end{table}

\begin{figure}
    \centering
    \input{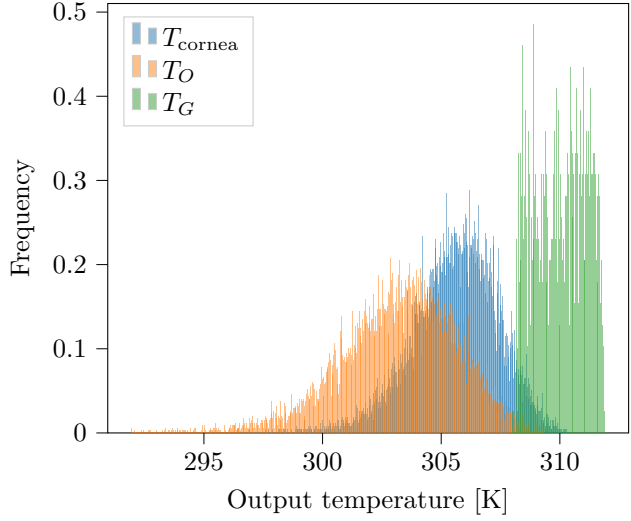}
    \caption{Distribution of the output, from the composed input distribution.}
    \label{fig:uncertainty-propagation}
\end{figure}

\subsubsection{Results of the SA}

We perform a sensitivity analysis to compute Sobol's indices, a convergence analysis varying the sampling size $N_\text{param}$.
\Cref{tab:results:SSA:convergence} reports the maximal deviation of these indices, and the time taken by the application to compute the 6 sets of Sobol's indices.

Additionally, we compute the \emph{predictivity factor} $Q_2$ for the polynomial chaos metamodel is defined as:
\begin{equation}
    Q_2:= 1 - \dfrac{\sum_{l=1}^{N}\big(Y_l - \widehat{f}(X_l)\big)^2}{\var(Y)},
\end{equation}
measuring how accurate the metamodel $\widehat{f}$ is at predicting the output $Y$ from the input $X$.
The closer $Q_2$ is to 1, the better the metamodel is.
In the context of the Sobol's indices experiment, the metamodel $\widehat{f}$ is the polynomial chaos expansion of the output $Y$.
The test of convergence is performed using the temperature on point $O$ as the output.
The convergence of Sobol's indices is reached for $N_\text{param} = 200$ with a $10^{-2}$ accuracy, which is a threshold used in the sequel.

\begin{table}
    \centering
    \begin{tabular}{cccc}
        \toprule
        $N_\text{param}$ & Max deviation & $t_\text{exec}$ & $Q_2$\\
        \midrule
        60  & 0.18102 & \qty{0.75609}{\second} & 0.999153\\
        100 & 0.03698 & \qty{1.99651}{\second} & 0.992648\\
        150 & 0.02969 & \qty{2.83743}{\second} & 0.99986 \\
        200 & 0.02923 & \qty{4.16046}{\second} & 0.998926\\
        400 & 0.00739 & \qty{8.36701}{\second} & 0.999931\\
        600 & 0.00496 & \qty{15.7947}{\second} & 0.9998  \\
        1000& 0.00248 & \qty{22.364 }{\second} & 0.999904\\
        \bottomrule
    \end{tabular}
    \caption{Convergence of the Sobol's indices.}
    \label{tab:results:SSA:convergence}
\end{table}

\Cref{fig:results:SSA:sobol} shows the results of the Sobol analysis for different outputs of interest.
Recall that \Cref{fig:outputs} shows where the points are in the eye.

In the deterministic sensitivity analysis conducted in \Cref{sec:DSA}, the impact of the variation of a sole parameter on the temperature at point $O$ was studied.
Using Sobol's indices, we are now able to measure the impact when all of them are varying.
The results of Sobol analysis at point $O$ presented in \Cref{fig:results:SSA:sobol:O} are in very good agreement with the deterministic findings:
the temperature at the level of the cornea is strongly influenced by external factors such as $h_\text{amb}$, as well subject-specific parameters such as $T_\text{amb}$, $E$, and $T_\text{bl}$,
Moreover, it is minimally influenced by the lens conductivity $k_\text{lens}$ and the blood convection coefficient $h_\text{bl}$.

Sobol's indices for several other locations are gathered in \Cref{fig:results:SSA:sobol}(b--f).
From these results, we can infer the following ranking of the influential parameters: $T_\text{amb}$, $h_\text{amb}$, $E$, and $T_\text{bl}$.
In particular, the dependence of the ambient temperature $T_\text{amb}$ decreases when we go deeper inside the eye.
Precisely, the impact of $T_\text{amb}$ is still significant for the mean temperature of the cornea, but the other parameters are equally influential.
These behaviors are coherent with physiological conditions.
Moreover, regardless of the output studied, the parameters $k_\text{lens}$ and $h_\text{bl}$ are minimally influencing the output.
Consequently in future simulations, their value can be set at baseline.
Surprisingly, the temperature at $B_1$, on the lens, is minimally influenced by $k_\text{lens}$, but this parameter has a minimal role in the modeling process.
On the other hand, $T_\text{bl}$ is very influential at $D_1$ and $G$, close to vascular beds,
again in a coherent manner with the physiological situation.
Finally, we can notice a slight difference between the first-order and total-order indices, mostly for $h_\text{amb}$ and $T_\text{amb}$,
implying that there are high-order interactions among these selected parameters.
To measure the impact of coupled parameters, second-order Sobol's indices computation is required, but the polynomial chaos expansion does not directly provide these values.
Alternatively, a Monte-Carlo based method could be implemented which is very costly to the computational viewpoint.

\pgfplotsset{
   textnumber/.style={
     /pgf/number format/.cd,
     fixed,
     use comma,
     fixed zerofill,
     precision=4,
     1000 sep={.},
     },
  }

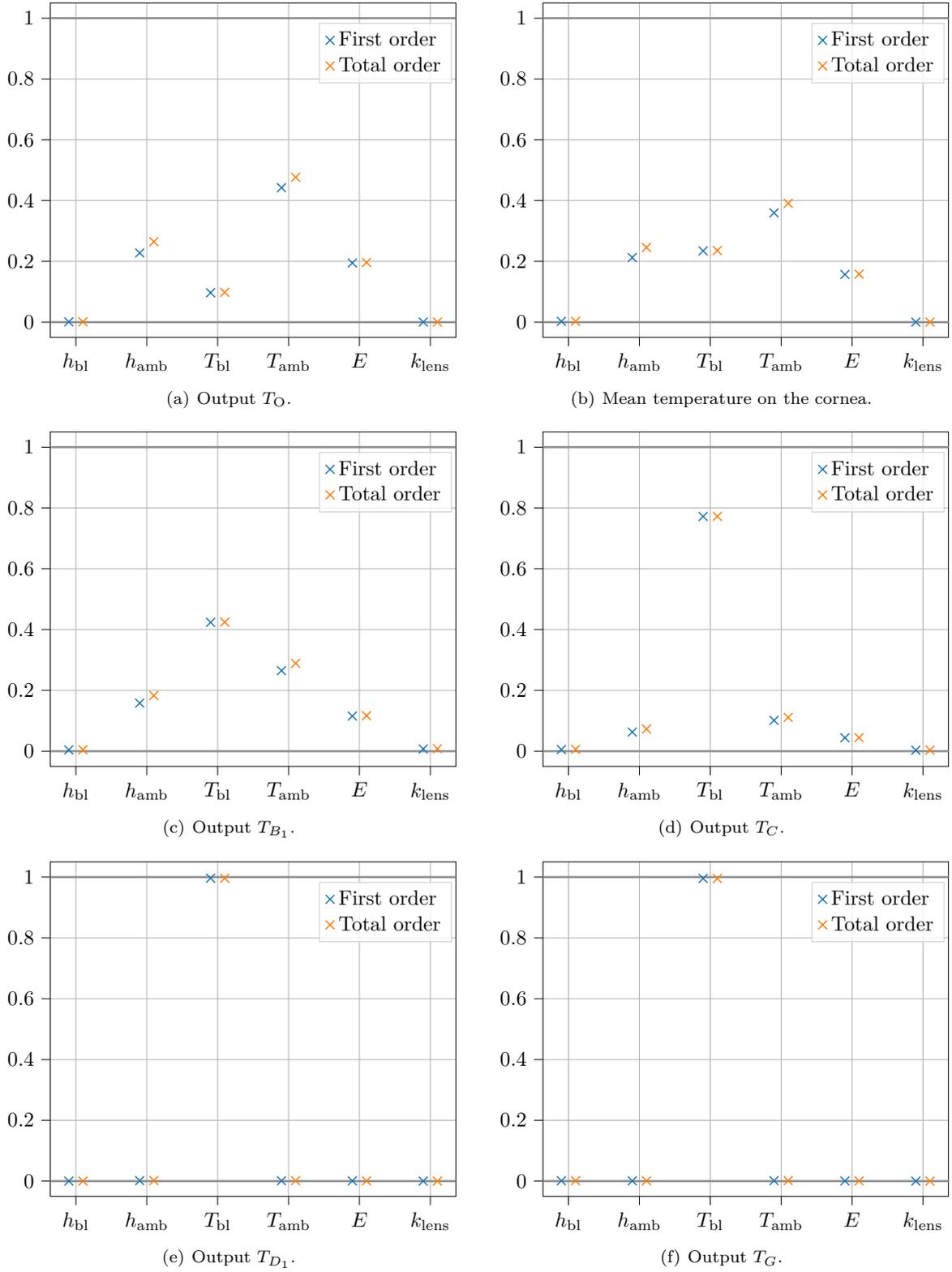
\begin{figure}
    \centering
    \def\mymarker{x}
    \def\chl{0.96}
    \subfigure[Output $T_\text{O}$.\label{fig:results:SSA:sobol:O}]{
        \def\sobolpath{fig/eye/results/SSA/sobol-O.tikz}
        \input{fig/eye/results/SSA/sobol-template.tikz}
    }
    \subfigure[Mean temperature on the cornea.]{
        \def\sobolpath{fig/eye/results/SSA/sobol-cornea.tikz}
        \input{fig/eye/results/SSA/sobol-template.tikz}
    }
    \subfigure[Output $T_{B_1}$.]{
        \def\sobolpath{fig/eye/results/SSA/sobol-B1.tikz}
        \input{fig/eye/results/SSA/sobol-template.tikz}
    }
    \subfigure[Output $T_C$.]{
        \def\sobolpath{fig/eye/results/SSA/sobol-C.tikz}
        \input{fig/eye/results/SSA/sobol-template.tikz}
    }
    \subfigure[Output $T_{D_1}$.]{
        \def\sobolpath{fig/eye/results/SSA/sobol-D1.tikz}
        \input{fig/eye/results/SSA/sobol-template.tikz}
    }
    \subfigure[Output $T_G$.]{
        \def\sobolpath{fig/eye/results/SSA/sobol-G.tikz}
        \input{fig/eye/results/SSA/sobol-template.tikz}
    }
    \caption{Sobol's indices for the SSA.}
    \label{fig:results:SSA:sobol}
\end{figure}

%% file: fig/eye/distributions/E.tikz
\begin{tikzpicture}[scale=\chl]

\definecolor{darkgray176}{RGB}{176,176,176}
\definecolor{darkorange25512714}{RGB}{255,127,14}
\definecolor{forestgreen4416044}{RGB}{44,160,44}
\definecolor{lightgray204}{RGB}{204,204,204}
\definecolor{steelblue31119180}{RGB}{31,119,180}

\begin{axis}[
legend cell align={left},
legend style={
  fill opacity=0.8,
  draw opacity=1,
  text opacity=1,
  at={(0.99,0.99)},
  anchor=north east,
  draw=lightgray204,
  nodes={scale=0.8, transform shape}
},
tick align=outside,
tick pos=left,
x grid style={darkgray176},
xlabel={$E$ [\unit{\watt.\meter^{-2}}]},
xmajorgrids,
xmin=-10.7699338462917, xmax=144.959770628964,
xtick style={color=black},
y grid style={darkgray176},
ymin=-0.00120604026816441, ymax=0.0253268456314527,
ytick style={color=black}
]
\addplot [dashed, red, fill=red, fill opacity=0.25]
table {%
-22.7822200065073 0
-21.534139719041 0
-20.2860594315747 0
-19.0379791441084 0
-17.7898988566421 0
-16.5418185691758 0
-15.2937382817095 0
-14.0456579942432 0
-12.7975777067768 0
-11.5494974193105 0
-10.3014171318442 0
-9.05333684437793 0
-7.80525655691162 0
-6.55717626944531 0
-5.30909598197901 0
-4.0610156945127 0
-2.8129354070464 0
-1.56485511958009 0
-0.316774832113783 0
0.931305455352522 0
2.17938574281883 0
3.42746603028514 0
4.67554631775144 0
5.92362660521775 0
7.17170689268405 0
8.41978718015036 0
9.66786746761667 0
10.915947755083 0
12.1640280425493 0
13.4121083300156 0
14.6601886174819 0
15.9082689049482 0
17.1563491924145 0
18.4044294798808 0
19.6525097673471 0
20.9005900548134 1.72547435768853e-06
22.1486703422797 0.000181582907477345
23.396750629746 0.00113414561660361
24.6448309172123 0.00310025379771601
25.8929112046786 0.00582684483753852
27.140991492145 0.00891274264757005
28.3890717796113 0.012010746463036
29.6371520670776 0.0148833462960416
30.8852323545439 0.0173946663405769
32.1333126420102 0.019484687905207
33.3813929294765 0.0211438273972586
34.6294732169428 0.0223930668785405
35.8775535044091 0.0232699726949402
37.1256337918754 0.0238194826100309
38.3737140793417 0.024088157572697
39.621794366808 0.0241208053632883
40.8698746542743 0.023958663398171
42.1179549417406 0.023638571637223
43.3660352292069 0.0231927515183586
44.6141155166732 0.0226489384553381
45.8621958041395 0.0220307056935559
47.1102760916058 0.0213578777587736
48.3583563790722 0.0206469714956011
49.6064366665385 0.0199116284495652
50.8545169540048 0.0191630187683871
52.1025972414711 0.018410207071761
53.3506775289374 0.0176604770060341
54.5987578164037 0.0169196148843155
55.84683810387 0.0161921548523157
57.0949183913363 0.0154815890300633
58.3429986788026 0.0147905464588627
59.5910789662689 0.0141209446887326
60.8391592537352 0.0134741176381079
62.0872395412015 0.0128509230478454
63.3353198286678 0.0122518324989238
64.5834001161341 0.0116770066055136
65.8314804036004 0.0111263576537601
67.0795606910667 0.0105996016428447
68.3276409785331 0.0100963014034393
69.5757212659994 0.00961590222049305
70.8238015534657 0.00915776117113344
72.071881840932 0.00872117120189821
73.3199621283983 0.00830538080958581
74.5680424158646 0.00790961005363437
75.8161227033309 0.00753306351210646
77.0642029907972 0.00717494069528157
78.3122832782635 0.00683444434801822
79.5603635657298 0.00651078700220824
80.8084438531961 0.00620319608185311
82.0565241406624 0.00591091781385667
83.3046044281287 0.00563322015609863
84.552684715595 0.00536939491948942
85.8007650030613 0.00511875923145685
87.0488452905277 0.00488065646378603
88.296925577994 0.00465445672717269
89.5450058654603 0.00443955701762145
90.7930861529266 0.00423538108538846
92.0411664403929 0.00404137908508453
93.2892467278592 0.00385702705544022
94.5373270153255 0.00368182626877279
95.7854073027918 0.00351530248311904
97.0334875902581 0.00335700512408507
98.2815678777244 0.00320650641852594
99.5296481651907 0.00306340049804798
100.777728452657 0.00292730248689243
102.025808740123 0.00279784758589962
103.27388902759 0.00267469016187615
104.521969315056 0.00255750284971448
105.770049602522 0.00244597567298067
107.018129889989 0.00233981518733538
108.266210177455 0.00223874365004065
109.514290464921 0.00214249821789144
110.762370752387 0.00205083017516283
112.010451039854 0.00196350419255592
113.25853132732 0.00188029761763203
114.506611614786 0.00180099979682978
115.754691902253 0.00172541142884338
117.002772189719 0.0016533439488925
118.250852477185 0.00158461894322068
119.498932764652 0.00151906759301218
120.747013052118 0.0014565301468076
121.995093339584 0.00139685542042074
123.243173627051 0.00133990032330569
124.491253914517 0.00128552941029182
125.739334201983 0.00123361445758846
126.987414489449 0.00118403406195963
128.235494776916 0.00113667326197789
129.483575064382 0.00109142318028403
130.731655351848 0
131.979735639315 0
133.227815926781 0
134.475896214247 0
135.723976501714 0
136.97205678918 0
};
\addlegendentry{PDF}
\addplot [dashed, color=black!60, line width=2pt]
table {%
40 -0.00120604026816441
40 0.0253268456314527
};
\addlegendentry{Baseline}
\addplot [semithick, colorLi]
table {%
30 -0.00120604026816441
30 0.0253268456314527
};
\addlegendentry{\cite{li2010}}
\addplot [semithick, colorLi, forget plot]
table {%
80 -0.00120604026816441
80 0.0253268456314527
};
\addplot [semithick, colorLi, forget plot]
table {%
130 -0.00120604026816441
130 0.0253268456314527
};
\addplot [semithick, colorLi, forget plot]
table {%
180 -0.00120604026816441
180 0.0253268456314527
};
\addplot [semithick, colorLi, forget plot]
table {%
230 -0.00120604026816441
230 0.0253268456314527
};
\addplot [semithick, colorOoi]
table {%
20 -0.00120604026816441
20 0.0253268456314527
};
\addlegendentry{\cite{NG2006268}}
\addplot [semithick, colorOoi, forget plot]
table {%
40 -0.00120604026816441
40 0.0253268456314527
};
\addplot [semithick, colorOoi, forget plot]
table {%
70 -0.00120604026816441
70 0.0253268456314527
};
\addplot [semithick, colorOoi, forget plot]
table {%
100 -0.00120604026816441
100 0.0253268456314527
};
\addplot [semithick, colorOoi, forget plot]
table {%
320 -0.00120604026816441
320 0.0253268456314527
};
\addplot [semithick, colorScott]
table {%
20 -0.00120604026816441
20 0.0253268456314527
};
\addlegendentry{\cite{Scott_1988}}
\addplot [semithick, colorScott, forget plot]
table {%
40 -0.00120604026816441
40 0.0253268456314527
};
\addplot [semithick, colorScott, forget plot]
table {%
100 -0.00120604026816441
100 0.0253268456314527
};
\addplot [semithick, colorScott, forget plot]
table {%
320 -0.00120604026816441
320 0.0253268456314527
};
\end{axis}

\end{tikzpicture}

%% file: fig/eye/distributions/h_amb.tikz
\begin{tikzpicture}[scale=\chl]

\definecolor{darkgray176}{RGB}{176,176,176}
\definecolor{darkorange25512714}{RGB}{255,127,14}
\definecolor{forestgreen4416044}{RGB}{44,160,44}
\definecolor{lightgray204}{RGB}{204,204,204}
\definecolor{steelblue31119180}{RGB}{31,119,180}

\begin{axis}[
legend cell align={left},
legend style={
  fill opacity=0.8,
  draw opacity=1,
  text opacity=1,
  at={(0.99,0.99)},
  anchor=north east,
  draw=lightgray204,
  nodes={scale=0.8, transform shape}
},
tick align=outside,
tick pos=left,
x grid style={darkgray176},
xlabel={$h_\text{amb}$ [\unit{\watt.\meter^{-2}.\kelvin^{-1}}]},
xmajorgrids,
xmin=-10, xmax=110,
xtick style={color=black},
y grid style={darkgray176},
ymin=-0.00542153685454793, ymax=0.113852273945507,
ytick style={color=black}
]
\addplot [dashed, red, fill=red, fill opacity=0.25]
table[x=x, y=pdf, col sep=comma] {fig/eye/distributions/h_amb_pdf.csv};
\addlegendentry{PDF}
\addplot [dashed, color=black!60, line width=2pt]
table {%
9.99999999999999 -0.00542153685454794
9.99999999999999 0.113852273945507
};
\addlegendentry{Baseline}
\addplot [semithick, colorLi]
table {%
15 -0.00542153685454794
15 0.113852273945507
};
\addlegendentry{\cite{li2010}}
\addplot [semithick, colorLi, forget plot]
table {%
30 -0.00542153685454794
30 0.113852273945507
};
\addplot [semithick, colorLi, forget plot]
table {%
50 -0.00542153685454794
50 0.113852273945507
};
\addplot [semithick, colorLi, forget plot]
table {%
80 -0.00542153685454794
80 0.113852273945507
};
\addplot [semithick, colorLi, forget plot]
table {%
100 -0.00542153685454794
100 0.113852273945507
};
\addplot [semithick, colorOoi]
table {%
8 -0.00542153685454794
8 0.113852273945507
};
\addlegendentry{\cite{NG2006268}}
\addplot [semithick, colorOoi, forget plot]
table {%
9.99999999999999 -0.00542153685454794
9.99999999999999 0.113852273945507
};
\addplot [semithick, colorOoi, forget plot]
table {%
15 -0.00542153685454794
15 0.113852273945507
};
\addplot [semithick, colorOoi, forget plot]
table {%
100 -0.00542153685454794
100 0.113852273945507
};
\addplot [semithick, colorScott]
table {%
8 -0.00542153685454794
8 0.113852273945507
};
\addlegendentry{\cite{Scott_1988}}
\addplot [semithick, colorScott, forget plot]
table {%
9.99999999999999 -0.00542153685454794
9.99999999999999 0.113852273945507
};
\addplot [semithick, colorScott, forget plot]
table {%
12 -0.00542153685454794
12 0.113852273945507
};
\addplot [semithick, colorScott, forget plot]
table {%
15 -0.00542153685454794
15 0.113852273945507
};
\end{axis}

\end{tikzpicture}

%% file: fig/eye/distributions/h_bl.tikz
\begin{tikzpicture}[scale=\chl]

\definecolor{darkgray176}{RGB}{176,176,176}
\definecolor{darkorange25512714}{RGB}{255,127,14}
\definecolor{forestgreen4416044}{RGB}{44,160,44}
\definecolor{lightgray204}{RGB}{204,204,204}
\definecolor{steelblue31119180}{RGB}{31,119,180}

\begin{axis}[
legend cell align={left},
legend style={
  fill opacity=0.8,
  draw opacity=1,
  text opacity=1,
  at={(0.99,0.99)},
  anchor=north east,
  draw=lightgray204,
  nodes={scale=0.8, transform shape}
},
tick align=outside,
tick pos=left,
x grid style={darkgray176},
xlabel={$h_\text{bl}$ [\unit{\watt.\meter^{-2}.\kelvin^{-1}}]},
xmajorgrids,
xmin=40, xmax=130,
xtick style={color=black},
ymin=-0.00219537964264221, ymax=0.0461029724954863,
ytick style={color=black}
]
\addplot [dashed, red, fill=red, fill opacity=0.25]
table[x=x, y=pdf, col sep=comma] {fig/eye/distributions/h_bl_pdf.csv};
\addlegendentry{PDF}
\addplot [dashed, color=black!60, line width=2pt]
table {%
65 -0.00219537964264221
65 0.0461029724954863
};
\addlegendentry{Baseline}
\addplot [semithick, colorLi]
table {%
50 -0.00219537964264221
50 0.0461029724954863
};
\addlegendentry{\cite{li2010}}
\addplot [semithick, colorLi, forget plot]
table {%
70 -0.00219537964264221
70 0.0461029724954863
};
\addplot [semithick, colorLi, forget plot]
table {%
90 -0.00219537964264221
90 0.0461029724954863
};
\addplot [semithick, colorLi, forget plot]
table {%
110 -0.00219537964264221
110 0.0461029724954863
};
\addplot [semithick, colorLi, forget plot]
table {%
120 -0.00219537964264221
120 0.0461029724954863
};
\addplot [semithick, colorOoi]
table {%
65 -0.00219537964264221
65 0.0461029724954863
};
\addlegendentry{\cite{NG2006268}}
\addplot [semithick, colorOoi, forget plot]
table {%
90 -0.00219537964264221
90 0.0461029724954863
};
\addplot [semithick, colorOoi, forget plot]
table {%
110 -0.00219537964264221
110 0.0461029724954863
};
\addplot [semithick, colorScott]
table {%
65 -0.00219537964264221
65 0.0461029724954863
};
\addlegendentry{\cite{Scott_1988}}
\addplot [semithick, colorScott, forget plot]
table {%
90 -0.00219537964264221
90 0.0461029724954863
};
\addplot [semithick, colorScott, forget plot]
table {%
110 -0.00219537964264221
110 0.0461029724954863
};
\end{axis}

\end{tikzpicture}

%% file: fig/eye/distributions/k_lens.tikz
\begin{tikzpicture}[scale=\chl]

\definecolor{darkgray176}{RGB}{176,176,176}
\definecolor{forestgreen4416044}{RGB}{44,160,44}
\definecolor{lightgray204}{RGB}{204,204,204}

\begin{axis}[
legend cell align={left},
legend style={fill opacity=0.8, draw opacity=1, text opacity=1, draw=lightgray204, nodes={scale=0.7, transform shape}},
tick align=outside,
tick pos=left,
x grid style={darkgray176},
xlabel={$k_\text{lens}$ [\unit{\watt.\meter^{-1}.\kelvin^{-1}}]},
xmajorgrids,
xmin=-0.0859240000000001, xmax=0.839924,
xtick style={color=black},
y grid style={darkgray176},
ymajorgrids,
ymin=-0.149700598802395, ymax=3.1437125748503,
ytick style={color=black}
]
\addplot [dashed, red, fill=red, fill opacity=0.25]
table {%
-0.04384 0
-0.037264375 0
-0.03068875 0
-0.024113125 0
-0.0175375 0
-0.010961875 0
-0.00438625000000004 0
0.00218937499999997 0
0.00876499999999997 0
0.015340625 0
0.02191625 0
0.028491875 0
0.0350675 0
0.041643125 0
0.04821875 0
0.054794375 0
0.06137 0
0.067945625 0
0.07452125 0
0.081096875 0
0.0876725 0
0.094248125 0
0.10082375 0
0.107399375 0
0.113975 0
0.120550625 0
0.12712625 0
0.133701875 0
0.1402775 0
0.146853125 0
0.15342875 0
0.160004375 0
0.16658 0
0.173155625 0
0.17973125 0
0.186306875 0
0.1928825 0
0.199458125 0
0.20603375 0
0.212609375 2.9940119760479
0.219185 2.9940119760479
0.225760625 2.9940119760479
0.23233625 2.9940119760479
0.238911875 2.9940119760479
0.2454875 2.9940119760479
0.252063125 2.9940119760479
0.25863875 2.9940119760479
0.265214375 2.9940119760479
0.27179 2.9940119760479
0.278365625 2.9940119760479
0.28494125 2.9940119760479
0.291516875 2.9940119760479
0.2980925 2.9940119760479
0.304668125 2.9940119760479
0.31124375 2.9940119760479
0.317819375 2.9940119760479
0.324395 2.9940119760479
0.330970625 2.9940119760479
0.33754625 2.9940119760479
0.344121875 2.9940119760479
0.3506975 2.9940119760479
0.357273125 2.9940119760479
0.36384875 2.9940119760479
0.370424375 2.9940119760479
0.377 2.9940119760479
0.383575625 2.9940119760479
0.39015125 2.9940119760479
0.396726875 2.9940119760479
0.4033025 2.9940119760479
0.409878125 2.9940119760479
0.41645375 2.9940119760479
0.423029375 2.9940119760479
0.429605 2.9940119760479
0.436180625 2.9940119760479
0.44275625 2.9940119760479
0.449331875 2.9940119760479
0.4559075 2.9940119760479
0.462483125 2.9940119760479
0.46905875 2.9940119760479
0.475634375 2.9940119760479
0.48221 2.9940119760479
0.488785625 2.9940119760479
0.49536125 2.9940119760479
0.501936875 2.9940119760479
0.5085125 2.9940119760479
0.515088125 2.9940119760479
0.52166375 2.9940119760479
0.528239375 2.9940119760479
0.534815 2.9940119760479
0.541390625 2.9940119760479
0.54796625 0
0.554541875 0
0.5611175 0
0.567693125 0
0.57426875 0
0.580844375 0
0.58742 0
0.593995625 0
0.60057125 0
0.607146875 0
0.6137225 0
0.620298125 0
0.62687375 0
0.633449375 0
0.640025 0
0.646600625 0
0.65317625 0
0.659751875 0
0.6663275 0
0.672903125 0
0.67947875 0
0.686054375 0
0.69263 0
0.699205625 0
0.70578125 0
0.712356875 0
0.7189325 0
0.725508125 0
0.73208375 0
0.738659375 0
0.745235 0
0.751810625 0
0.75838625 0
0.764961875 0
0.7715375 0
0.778113125 0
0.78468875 0
0.791264375 0
0.79784 0
};
\addlegendentry{PDF}
\addplot [dashed, color=black!60, line width=2pt]
table {%
0.4 -0.149700598802395
0.4 3.1437125748503
};
\addlegendentry{Baseline}
\addplot [semithick, forestgreen4416044, dashed]
table {%
0.21 -0.149700598802395
0.21 3.1437125748503
};
\addlegendentry{\cite{NG2006268} \& \cite{Scott_1988}}
\addplot [semithick, forestgreen4416044, dashed, forget plot]
table {%
0.3 -0.149700598802395
0.3 3.1437125748503
};
\addplot [semithick, forestgreen4416044, dashed, forget plot]
table {%
0.4 -0.149700598802395
0.4 3.1437125748503
};
\addplot [semithick, forestgreen4416044, dashed, forget plot]
table {%
0.544 -0.149700598802395
0.544 3.1437125748503
};
\end{axis}

\end{tikzpicture}

%% file: fig/eye/distributions/T_amb.tikz
\begin{tikzpicture}[scale=\chl]

\definecolor{darkgray176}{RGB}{176,176,176}
\definecolor{darkorange25512714}{RGB}{255,127,14}
\definecolor{forestgreen4416044}{RGB}{44,160,44}
\definecolor{lightgray204}{RGB}{204,204,204}
\definecolor{steelblue31119180}{RGB}{31,119,180}

\begin{axis}[
legend cell align={left},
legend style={fill opacity=0.8, draw opacity=1, text opacity=1, draw=lightgray204, nodes={scale=0.7, transform shape}},
tick align=outside,
tick pos=left,
x grid style={darkgray176},
xlabel={$T_\text{amb}$ [\unit{\kelvin}]},
xmajorgrids,
xmin=265.43, xmax=320.87,
xtick style={color=black},
y grid style={darkgray176},
ymajorgrids,
ymin=-0.0025, ymax=0.0525,
ytick style={color=black}
]
\addplot [dashed, red, fill=red, fill opacity=0.25]
table {%
267.95 0
268.34375 0
268.7375 0
269.13125 0
269.525 0
269.91875 0
270.3125 0
270.70625 0
271.1 0
271.49375 0
271.8875 0
272.28125 0
272.675 0
273.06875 0
273.4625 0
273.85625 0
274.25 0
274.64375 0
275.0375 0
275.43125 0
275.825 0
276.21875 0
276.6125 0
277.00625 0
277.4 0
277.79375 0
278.1875 0
278.58125 0
278.975 0
279.36875 0
279.7625 0
280.15625 0
280.55 0
280.94375 0
281.3375 0
281.73125 0
282.125 0
282.51875 0
282.9125 0
283.30625 0.05
283.7 0.05
284.09375 0.05
284.4875 0.05
284.88125 0.05
285.275 0.05
285.66875 0.05
286.0625 0.05
286.45625 0.05
286.85 0.05
287.24375 0.05
287.6375 0.05
288.03125 0.05
288.425 0.05
288.81875 0.05
289.2125 0.05
289.60625 0.05
290 0.05
290.39375 0.05
290.7875 0.05
291.18125 0.05
291.575 0.05
291.96875 0.05
292.3625 0.05
292.75625 0.05
293.15 0.05
293.54375 0.05
293.9375 0.05
294.33125 0.05
294.725 0.05
295.11875 0.05
295.5125 0.05
295.90625 0.05
296.3 0.05
296.69375 0.05
297.0875 0.05
297.48125 0.05
297.875 0.05
298.26875 0.05
298.6625 0.05
299.05625 0.05
299.45 0.05
299.84375 0.05
300.2375 0.05
300.63125 0.05
301.025 0.05
301.41875 0.05
301.8125 0.05
302.20625 0.05
302.6 0.05
302.99375 0.05
303.3875 0
303.78125 0
304.175 0
304.56875 0
304.9625 0
305.35625 0
305.75 0
306.14375 0
306.5375 0
306.93125 0
307.325 0
307.71875 0
308.1125 0
308.50625 0
308.9 0
309.29375 0
309.6875 0
310.08125 0
310.475 0
310.86875 0
311.2625 0
311.65625 0
312.05 0
312.44375 0
312.8375 0
313.23125 0
313.625 0
314.01875 0
314.4125 0
314.80625 0
315.2 0
315.59375 0
315.9875 0
316.38125 0
316.775 0
317.16875 0
317.5625 0
317.95625 0
318.35 0
};
\addlegendentry{PDF}
\addplot [dashed, color=black!60, line width=2pt]
table {%
294 -0.0025
294 0.0525
};
\addlegendentry{Baseline}
\addplot [semithick, colorLi]
table {%
273 -0.0025
273 0.0525
};
\addlegendentry{\cite{li2010}}
\addplot [semithick, colorLi, forget plot]
table {%
278 -0.0025
278 0.0525
};
\addplot [semithick, colorLi, forget plot]
table {%
283 -0.0025
283 0.0525
};
\addplot [semithick, colorLi, forget plot]
table {%
303 -0.0025
303 0.0525
};
\addplot [semithick, colorLi, forget plot]
table {%
308 -0.0025
308 0.0525
};
\addplot [semithick, colorOoi]
table {%
293.15 -0.0025
293.15 0.0525
};
\addlegendentry{\cite{NG2006268} \& \cite{Scott_1988}}
\addplot [semithick, colorOoi, forget plot]
table {%
298.15 -0.0025
298.15 0.0525
};
\addplot [semithick, colorOoi, forget plot]
table {%
303.15 -0.0025
303.15 0.0525
};
\end{axis}

\end{tikzpicture}

%% file: fig/eye/distributions/T_bl.tikz
\begin{tikzpicture}[scale=\chl]

\definecolor{darkgray176}{RGB}{176,176,176}
\definecolor{darkorange25512714}{RGB}{255,127,14}
\definecolor{forestgreen4416044}{RGB}{44,160,44}
\definecolor{lightgray204}{RGB}{204,204,204}
\definecolor{steelblue31119180}{RGB}{31,119,180}

\begin{axis}[
legend cell align={left},
legend style={fill opacity=0.8, draw opacity=1, text opacity=1, draw=lightgray204, nodes={scale=0.6, transform shape}},
tick align=outside,
tick pos=left,
x grid style={darkgray176},
xlabel={$T_\text{bl}$ [\unit{\kelvin}]},
xmajorgrids,
xmin=304.3231, xmax=315.8269,
xtick style={color=black},
y grid style={darkgray176},
ymin=-0.0120481927710844, ymax=0.253012048192772,
ytick style={color=black}
]
\addplot [dashed, red, fill=red, fill opacity=0.25]
table {%
304.846 0
304.927703125 0
305.00940625 0
305.091109375 0
305.1728125 0
305.254515625 0
305.33621875 0
305.417921875 0
305.499625 0
305.581328125 0
305.66303125 0
305.744734375 0
305.8264375 0
305.908140625 0
305.98984375 0
306.071546875 0
306.15325 0
306.234953125 0
306.31665625 0
306.398359375 0
306.4800625 0
306.561765625 0
306.64346875 0
306.725171875 0
306.806875 0
306.888578125 0
306.97028125 0
307.051984375 0
307.1336875 0
307.215390625 0
307.29709375 0
307.378796875 0
307.4605 0
307.542203125 0
307.62390625 0
307.705609375 0
307.7873125 0
307.869015625 0
307.95071875 0
308.032421875 0.240963855421688
308.114125 0.240963855421688
308.195828125 0.240963855421688
308.27753125 0.240963855421688
308.359234375 0.240963855421688
308.4409375 0.240963855421688
308.522640625 0.240963855421688
308.60434375 0.240963855421688
308.686046875 0.240963855421688
308.76775 0.240963855421688
308.849453125 0.240963855421688
308.93115625 0.240963855421688
309.012859375 0.240963855421688
309.0945625 0.240963855421688
309.176265625 0.240963855421688
309.25796875 0.240963855421688
309.339671875 0.240963855421688
309.421375 0.240963855421688
309.503078125 0.240963855421688
309.58478125 0.240963855421688
309.666484375 0.240963855421688
309.7481875 0.240963855421688
309.829890625 0.240963855421688
309.91159375 0.240963855421688
309.993296875 0.240963855421688
310.075 0.240963855421688
310.156703125 0.240963855421688
310.23840625 0.240963855421688
310.320109375 0.240963855421688
310.4018125 0.240963855421688
310.483515625 0.240963855421688
310.56521875 0.240963855421688
310.646921875 0.240963855421688
310.728625 0.240963855421688
310.810328125 0.240963855421688
310.89203125 0.240963855421688
310.973734375 0.240963855421688
311.0554375 0.240963855421688
311.137140625 0.240963855421688
311.21884375 0.240963855421688
311.300546875 0.240963855421688
311.38225 0.240963855421688
311.463953125 0.240963855421688
311.54565625 0.240963855421688
311.627359375 0.240963855421688
311.7090625 0.240963855421688
311.790765625 0.240963855421688
311.87246875 0.240963855421688
311.954171875 0.240963855421688
312.035875 0.240963855421688
312.117578125 0.240963855421688
312.19928125 0
312.280984375 0
312.3626875 0
312.444390625 0
312.52609375 0
312.607796875 0
312.6895 0
312.771203125 0
312.85290625 0
312.934609375 0
313.0163125 0
313.098015625 0
313.17971875 0
313.261421875 0
313.343125 0
313.424828125 0
313.50653125 0
313.588234375 0
313.6699375 0
313.751640625 0
313.83334375 0
313.915046875 0
313.99675 0
314.078453125 0
314.16015625 0
314.241859375 0
314.3235625 0
314.405265625 0
314.48696875 0
314.568671875 0
314.650375 0
314.732078125 0
314.81378125 0
314.895484375 0
314.9771875 0
315.058890625 0
315.14059375 0
315.222296875 0
315.304 0
};
\addlegendentry{PDF}
\addplot [dashed, color=black!60, line width=2pt]
table {%
310 -0.0120481927710844
310 0.253012048192772
};
\addlegendentry{Baseline}
\addplot [semithick, colorLi]
table {%
308 -0.0120481927710844
308 0.253012048192772
};
\addplot [semithick, colorLi, forget plot]
table {%
309 -0.0120481927710844
309 0.253012048192772
};
\addplot [semithick, colorLi, forget plot]
table {%
310 -0.0120481927710844
310 0.253012048192772
};
\addplot [semithick, colorLi, forget plot]
table {%
311 -0.0120481927710844
311 0.253012048192772
};
\addplot [semithick, colorLi, forget plot]
table {%
312 -0.0120481927710844
312 0.253012048192772
};
\addlegendentry{\cite{li2010}}
\addplot [semithick, colorOoi]
table {%
308.15 -0.0120481927710844
308.15 0.253012048192772
};
\addplot [semithick, colorOoi, forget plot]
table {%
310.15 -0.0120481927710844
310.15 0.253012048192772
};
\addplot [semithick, colorOoi, forget plot]
table {%
311.15 -0.0120481927710844
311.15 0.253012048192772
};
\addplot [semithick, colorOoi, forget plot]
table {%
312.15 -0.0120481927710844
312.15 0.253012048192772
};
\addlegendentry{\cite{NG2006268}}
\addplot [semithick, colorScott]
table {%
310.15 -0.0120481927710844
310.15 0.253012048192772
};
\addplot [semithick, colorScott, dashed, forget plot]
table {%
310.95 -0.0120481927710844
310.95 0.253012048192772
};
\addplot [semithick, colorScott, dashed, forget plot]
table {%
311.15 -0.0120481927710844
311.15 0.253012048192772
};
\addplot [semithick, colorScott, dashed, forget plot]
table {%
311.65 -0.0120481927710844
311.65 0.253012048192772
};
\addlegendentry{\cite{Scott_1988}}
\end{axis}

\end{tikzpicture}

%% file: fig/eye/results/SSA/sobol-template.tikz
\begin{tikzpicture}[scale=\chl]

\begin{axis}[
legend cell align={left},
legend style={
    fill opacity=0.8,
    draw opacity=1,
    text opacity=1,
    at={(0.99,0.94)},
    anchor=north east,
    draw=lightgray204
},
tick align=outside,
tick pos=left,
x grid style={darkgray176},
xmajorgrids,
xmin=-0.26, xmax=5.46,
xtick style={color=black},
xtick={0.1,1.1,2.1,3.1,4.1,5.1},
xticklabels={$h_\text{bl}$, $h_\text{amb}$, $T_\text{bl}$, $T_\text{amb}$, $E$, $k_\text{lens}$},
y grid style={darkgray176},
ymajorgrids,
ymin=-0.05, ymax=1.05,
ytick style={color=black}
]

\addplot [darkgray176!80!black, forget plot, line width=1pt]
table {%
-0.26 0
5.46 0
};
\addplot [darkgray176!80!black, forget plot, line width=1pt]
table {%
-0.26 1
5.46 1
};

\input{\sobolpath}
\addlegendentry{Total order}
\end{axis}

\end{tikzpicture}

%% file: tex/5-conclusion.tex
\section{Conclusion}
\label{sec:conclusion}

We have successfully developed a numerical model that accurately simulates heat transfer within the complex three-dimensional structure of the human eyeball,
enabling us to calculate the temperature distribution across various ocular tissues.
This model has undergone rigorous validation against both experimental data and numerical results from existing literature.
A key advancement in our study is the implementation of a certified reduced basis method.
This method significantly accelerates the simulations of our complex model while maintaining high accuracy,
making it highly efficient for many-queries computations essential in uncertainty quantification studies.
Our sensitivity analysis pinpointed four main physiological parameters as most influential in affecting the results:
blood temperature, ambient temperature, the ambient air convection coefficient, and the evaporation rate.
These findings build upon and enrich prior studies, such as those highlighted in \cite{Scott_1988,NG2006268,li2010},
underscoring the vital role of blood flow characteristics and environmental conditions, particularly in the inner ocular tissues.
Additionally, through Sobol' indices analysis, we identified the significant impact of parameter interactions, particularly those related to ambient temperature.
From a clinical standpoint, our insights into heat transport in the human eye could inform studies on the effects of electromagnetic wave radiation,
as explored in \cite{Hirata2007,NG2007829,doi:10.1142/S0219519409002936} and related references.
As a next step, we plan to couple the heat transfer model with models describing the aqueous humor flow, as in \cite{OOI2008252,10.1007/978-3-030-63591-6_45}.
This work is crucial to understand the formation of Krukenberg's spindle, to better understand the pathophysiology of pigment dispersion syndrome and to enhance drug delivery efficiency in the anterior chamber~\cite{Wang2016,BHANDARI2020286}. Our preliminary findings in this direction can be found in~\cite{saigre:hal-04558924}.
Ultimately, our work, in conjunction with previous initiatives such as the Ocular Mathematical Virtual Simulator \cite{https://doi.org/10.1002/cnm.3791},
lays the groundwork for a comprehensive, multi-physics, multiscale framework in ophthalmology, tailored for personalized medical applications.

%% file: tex/A-reproduce.tex
\section{Reproductibility of results}
\label{app:reproduce}

All the codes used to run the simulations are available on the \fpp{}~\cite{christophe_prud_homme_2023_8272196} GitHub repository\footnote{\url{https://github.com/feelpp/feelpp}}.
Details are given in the \texttt{README.md} file of the repository: \url{https://github.com/feelpp/feelpp/tree/develop/mor/examples/eye2brain}